\documentclass[12pt]{amsart}
\usepackage[toc,page]{appendix}

\usepackage{fancyhdr}
\newcommand{\nc}{\newcommand}
\newcommand{\rc}{\renewcommand}

\usepackage[ansinew]{inputenc}
\usepackage{color}
\usepackage{amscd}
\usepackage{enumerate,latexsym}
\usepackage{amsmath,amssymb}
\usepackage{graphicx}
\usepackage{amsmath,amssymb}
\usepackage{amsxtra}
\usepackage{amsfonts}
\usepackage{enumerate}

\definecolor{b}{rgb}{.1,.1,.7}
\definecolor{rr}{rgb}{.8,0,.3}
\definecolor{g}{rgb}{0,.5,0}
\definecolor{pp}{rgb}{.5,0,.7}
\definecolor{r}{rgb}{.6,0,.3}
\definecolor{y}{rgb}{.9,.99,.9}

\newcommand{\bbb}{\textcolor{b}}

\newcommand{\bblock}{\begin{block}}
\newcommand{\eblock}{\end{block}}

\def\R{\mathbb{R}}

\newcommand{\bit}{\begin{itemize}}
\newcommand{\een}{\end{enumerate}}
\newcommand{\eit}{\end{itemize}}
\newcommand{\ed}{\end{document}}

\nc{\Aut}{{	\operatorname{Aut}	}}
\nc{\codim}{{	\operatorname{codim}	}}
\nc{\Ob}{{	\operatorname{Ob}	}}
\nc{\PGL}{{	\operatorname{PGL}	}}
\nc{\supp}{{	\operatorname{supp}	}}
\nc{\tr}{{	\operatorname{tr}	}}

\nc{\Rep}{{	{\cal{R}}ep		}}
\nc{\one}{{	\mbox{\bf{1}}		}}
\nc{\iso}{	\overset{\sim}{\lra}	}

\nc{\nen}{\newenvironment}

\nc{\pr}{\protect}

\nc{\nn}{{\newline}}
\nc{\np}{{\newpage}}	

\nc{\lab}{	\label}
\nc{\npp}{{	\newpage\setcounter{page}{0}	}}
\nc{\setpa}{		\setcounter{part}		}
\nc{\setse}{		\setcounter{section}	}
\nc{\setsus}{		\setcounter{subsection}		}
\nc{\setsss}{		\setcounter{subsubsection}	}
\nc{\setpage}{		\setcounter{page}	}

\nc{\nfd}{ $$\text{ This version is preliminary and approximate, 
		             it is not for distribution. }$$	}

\nc{\noi}{{\noindent}}
\nc{\pf}{{	\noindent {\em Proof.}		}}
					
\nc{\epf}{ \fbox{\bf QED}	}

\nc{\heart}{{\tiny \cen{\tiny $\heartsuit $ }	}} 
\nc{\cont}{\tableofcontents}

\nc{\sbr}{{	\smallpagebreak	}}
\nc{\mbr}{{	\medpagebreak	}}
\nc{\bbr}{{	\bigpagebreak	}}

\nc{\bib}{		}

\rc{\b}{ 	\big         			}  
\nc{\lam}[1]{{ 	\text{\large $#1$	}	}}  
\nc{\smm}[1]{{ 	\text{\small $#1$	}	}}  
\nc{\fom}[1]{{ 	\text{\footnotesize $#1$	}	}}  
\nc{\tinm}[1]{{ \text{\tiny $#1$	}	}}  

\nc{\bu}{ \bullet         }  			
\nc{\bbu}{ \aa{\bbb \bullet}         }  	
\nc{\bus}{{	^\bullet	}}	 	
\nc{\bui}{{	_\bullet	}}	 	

\nc{\bem}{{	\begin{em}	}}
\nc{\eem}{{	\end{em} 	}}

\nc{\bbox}{{	\blackbox	}}	
\nc{\bx}{	\boxed	}		
\nc{\tbx}[1]{{\boxed{\tx{#1}}}}		

\nc{\mmbox}[1]{{	\mbox{$#1$}	}}	
\nc{\tbox}[1]{{		\mbox{\tx{#1}}	}}
 	
\nc{\ot}{		\leftarrow			}
\nc{\tto}{		\longrightarrow			}
\nc{\ott}{		\longleftarrow			}
\nc{\too}[1]{{		\aa{#1}\rightarrow			}}
\nc{\oot}[1]{{		\aa{#1}\leftarrow			}}
\nc{\ttoo}[1]{{		\aa{#1}\longrightarrow			}}
\nc{\oott}[1]{{		\aa{#1}\longleftarrow			}}

\nc{\Too}[2]{{		\aa{#1}{\bb{#2}\rightarrow}		}}
\nc{\ooT}[2]{{		\aa{#1}{\bbb{#2}\leftarrow}		}}
\nc{\TToo}[2]{{		\aa{#1}{\bb{#2}\longrightarrow}		}}
\nc{\ooTT}[2]{{		\aa{#1}{\bbb{#2}\longleftarrow}		}}
\nc{\toot}[2]{{		\aa{#1}{\bb{#2}\rightleftarrows}	}}
\nc{\ttoot}[2]{{	\aa{#1}{\bb{#2}\rightleftrightarrows}	}}

\nc{\ra}{{	\rightarrow		}}
\nc{\laa}{{	\leftarrow	}}	
\nc{\lra}{{\longrightarrow}}

\nc{\lr}{{\leftrightarrow}}     	
\nc{\lrs}{{\rightleftarrows}}     	

\nc{\imp}{{\Rightarrow}}        	
\nc{\impp}{{\Leftarrow}}        	
\nc{\eq}{{\Leftrightarrow}}        	
\nc{\impl}{{\Longrightarrow}}        	
\nc{\imppl}{{\Longleftarrow}}        	
\nc{\eql}{{\Longleftrightarrow}}        	
	\nc{\Ra}{{\Rightarrow}}         	
	\nc{\LRa}{{\Leftrightarrow}}        	

\nc{\inj}{{\pr	\hookrightarrow	}}    		
\nc{\injj}{{\pr	\hookleftarrow	}}    		

\nc{\sur}{{	\twoheadrightarrow	}}	
\nc{\surr}{{	\twoheadleftarrow	}}	
			
\nc{\mm}{{	\mapsto		}}     		
\nc{\mmm}{{	\leftarrow\shortmid }}		
\nc{\ainj}[1]{{\aa{#1}{\pr\hookrightarrow}	}}    	
\nc{\ainjj}[1]{{\aa{#1}{\pr\hookleftarrow}	}}    	

\nc{\asur}[1]{{	\aa{#1}\twoheadrightarrow	}}	
\nc{\asurr}[1]{{\aa{#1}\twoheadleftarrow	}}	
			
\nc{\amm}[1]{{	\aa{#1}\mapsto		}}     	
\nc{\ammm}[1]{{	\aa{#1}\leftarrow\shortmid }}	

\nc{\syp}[1]{	^{ (#1) }		} 	
\nc{\up}[1]{	^{ (#1) }		} 	
\nc{\lp}[1]{	_{ (#1) }		}	
\nc{\hp}[1]{	^{ [#1] }		}	
\nc{\cle}{\preceq}		
\nc{\cl}{\prec}			
\nc{\cge}{\succeq}		
\nc{\cg}{\succ}			

\nc{\bb}{	\pr\underset 	}           
\rc{\aa}{ 	\pr\overset 	}            

\nc{\indd}{{ ${} \ \ \ \ \  \ \        {} $	}}	
\nc{\inddd}{{ 	\indd\indd			}}	
\nc{\nnd}{{ 	\nn  \indd 			}}	
\nc{\nndb}{{ 	\nn  \indd $\bullet$		}}	
		
\nc{\bce}{	\begin{center}	}
\nc{\ece}{	\end  {center}	}
\nc{\cen}[1]{	\begin{center}	{  #1}	\end  {center}	}

\nc{\bss}{{\backslash}}           		
\nc{\barr}{ 	\overline 	}      		
\nc{\ud}{	\underline	}		

\nc{\ti}{\tilde}              
\nc{\tii}{\widetilde}         

\nc{\hatt}{\widehat}				
\nc{\hata}{{	\bbb{ \hat{} }		}}	

\nc{\ch}{\check}              			
\nc{\cha}{{ 	\bbb{ \check{} }	}}      

\nc{\sub}{{	\subseteq	}}         
\nc{\subb}{{	\supseteq	}}         
\nc{\nsub}{{	\nsubseteq	}}         
\nc{\nsubb}{{	\nsupseteq	}}         %

\nc{\nin}{{	\notin	}}

\nc{\lb}{\langle}             				
\nc{\rb}{\rangle}

\nc{\lB}{	\left(	}             			
\nc{\rB}{	\right)	}
\nc{\BBl}{{	\bbb{ \left( \right.}	}}             	
\nc{\BBr}{{	\bbb{ \left. \right)}	}}

\nc{\Pa}[2]{ {\lb} #1 {,} #2 {\rb} }				

\nc{\cD}[1]{ \tx{ $$\CD {#1} \endCD $$ }  }		

\nc{\mat} {		\left(		\matrix	}	
\nc{\emat}{		\endmatrix	\right)	}
\nc{\sm} {		\left(		\smallmatrix	}	
\nc{\esm}{		\endsmallmatrix	\right)	}
\nc{\smat} {		\left(		\smallmatrix	}	
\nc{\esmat}{		\endsmallmatrix	\right)	}

\nc{\matr} {		\left[		\matrix	}	
\nc{\ematr}{		\endmatrix	\right]	}
\nc{\smr} {		\left[		\smallmatrix	}	
\nc{\esmr}{		\endsmallmatrix	\right]	}
\nc{\smatr} {		\left[		\smallmatrix	}	
\nc{\esmatr}{		\endsmallmatrix	\right]	}

\nc{\imat} {		\left.		\matrix	}	
\nc{\eimat}{		\endmatrix	\right.	}
\nc{\ism} {		\left.		\smallmatrix	}	
\nc{\eism}{		\endsmallmatrix	\right.	}

\nc{\ca}{		\left\{		\smallmatrix	}	
\nc{\eca}{		\endsmallmatrix	\right\}	}
\nc{\Ca}{		\left\{		\matrix		}	
\nc{\Eca}{		\endmatrix	\right.		}	
\nc{\eCa}{		\endmatrix	\right\}	}	

\nc{\com}{	\begin{diagram}	}
\nc{\ecom}{	  \end{diagram}	}

\nc{\tab}{	\begin{tabular}		}
\nc{\etab}{	\end{tabular}		}	
\nc{\hl}{{	\hline			}}

\nc{\Eq}{	\begin{equation}	}
\nc{\Eeq}{	\end{equation}	}
\nc{\aln}{	\begin{align}	}
\nc{\ealn}{	\end{align}	}

\nc{\Rpart}{	\rc{\thepart}{\Roman{part}}	}
\nc{\Apart}{	\rc{\thepart}{\arabic{part}}	}
		
\nc{\rref}[2]{\ref{#1}.\ref{#2}}

\nc{\pa}[1]{ 	\part{#1}		}

\nc{\se}[1]{ 	\section{\bf#1}		}
\nc{\ses}[1]{ 	\section*{\bf#1}		}
\nc{\sus}{ 	\subsection		}
\nc{\sss}{ 	\subsubsection		}

\nc{\Lem}{ 	\subsection{Lemma}		}
\nc{\lem}{ 	\subsubsection{Lemma}		}
\nc{\slem}{ 	\subsubsection*{Lemma}		}
\nc{\sublem}{ 	\subsubsection{ Sublemma}	}
\nc{\ssublem}{ \subsubsection*{ Sublemma}	}

\nc{\Lemm}{ 	\subsection{Lemma}		}
\nc{\lemm}{ 	\subsubsection{Lemma}		}
\nc{\slemm}{ 	\subsubsection*{Lemma}		}
\nc{\sublemm}{ 	\subsubsection{ Sublemma}	}
\nc{\ssublemm}{ \subsubsection*{ Sublemma}	}

\nc{\Pro}{ 	\subsection{Proposition}	}
\nc{\pro}{ 	\subsubsection{Proposition}	}
\nc{\spro}{ 	\subsubsection*{Proposition}	}

\nc{\Cor}{ 	\subsection{Corollary}		}
\nc{\cor}{ 	\subsubsection{Corollary}	}
\nc{\scor}{ 	\subsubsection*{Corollary}	}

\nc{\Corr}{ 	\subsection{Corollary}		}
\nc{\corr}{ 	\subsubsection{Corollary}	}
\nc{\scorr}{ 	\subsubsection*{Corollary}	}

\nc{\Theo}{ 	\subsection{Theorem}		}		
\nc{\theo}{ 	\subsubsection{Theorem}		}
\nc{\stheo}{ 	\subsubsection*{Theorem}	}

\nc{\pretheo}{ 	\subsubsection{Pretheorem}	}

\nc{\rem}{ 	\subsubsection{Remark}		}
\nc{\srem}{ 	\subsubsection*{Remark}	}

\nc{\rems}{ 	\subsubsection{Remarks}		}

\nc{\srems}{ 	\subsubsection*{Remarks}	}

\nc{\Def}{ 	\subsection{Definition}		}
\nc{\ddef}{ 	\subsubsection{Definition}	}

\nc{\comm}{ 	\subsubsection{Comment}		}
\nc{\scomm}{ 	\subsubsection*{Comment}	}
\nc{\comms}{ 	\subsubsection{Comments}		}
\nc{\scomms}{ 	\subsubsection*{Comments}	}

\nc{\claim}{ 	\subsubsection{Claim}		}
\nc{\sclaim}{ 	\subsubsection*{Claim}	}

\nc{\nota}{ 	\subsubsection{Notation}	}

\nc{\conj}{ 	\subsubsection{Conjecture}	}
\nc{\sconj}{ 	\subsubsection*{Conjecture}	}

\nc{\ex}{ 	\subsubsection{Example}		}
\nc{\sex}{ 	\subsubsection*{Example}	}
\nc{\exs}{ 	\subsubsection{Examples}	}
\nc{\sexs}{ 	\subsubsection*{Examples}	}
\nc{\Ex}{ 	\subsection{Example}		}
\nc{\sEx}{ 	\subsection*{Example}	}
\nc{\Exs}{ 	\subsection{Examples}	}
\nc{\sExs}{ 	\subsection*{Examples}	}

\nc{\que}{ 	\subsubsection{Question}	}
\nc{\ques}{ 	\subsubsection{Questions}	}
\nc{\sque}{ 	\subsubsection*{Question}	}
\nc{\sques}{ 	\subsubsection*{Questions}	}

\nc{\bi}{	\begin{itemize}\item		}
\rc{\i}{	\item			}
\nc{\ei}{ \end{itemize}	} 
\nc{\ben}{	\begin{enumerate}\item		}
\nc{\ftt}[1]{{\footnote{#1}}}
\nc{\fttt}[1]{{$^($\footnote{#1}$^)$}}
\nc{\bftt}[1]{\footnote{#1}}
\nc{\f}[1]{ \fbox{$ $}\footnote{ \fbox{!}#1 }\fbox{$ $}		}

\nc{\Ao}{{	\A^1	}}
\nc{\Po}{{	\P^1	}}
\nc{\So}{{	S^1	}}

\nc{\h}{{	\hslash	}}	

\nc{\All}{{	\forall		}}
\nc{\Exx}{{	\exists 	}}
\nc{\yy}{\infty}                       
\nc{\ys}{{  \frac{\infty}{2}  }}

\nc{\ii}{{i\in I}}
\nc{\ww}{{w\in W}}
\nc{\SES}[5]{{	0 @>>> {#1} @>{#2}>> {#3} @>{#4}>> {#5} @>>> 0	}}
\nc{\Ses}[3] {{	0 @>>> {#1} @>>>     {#2} @>>>     {#3} @>>> 0	}}
\nc{\pl}{{\oplus}}              		
\nc{\tim}{{\times}}             
\nc{\btim}{{\boxtimes}}
\nc{\ltim}{\ltimes}                  	%
\nc{\rtim}{\rtimes}			%
\nc{\ltr}{\triangleleft}        %
\nc{\rtr}{\triangleright}       %

\nc{\ten}{{	\otimes		}}            
\nc{\Lten}{{	\aa{L}\otimes	}}            
\nc{\Ltim}{{	\aa{L}\times	}}            
\nc{\Lcap}{{	\aa{L}\cap	}}            
\nc{\tenA}{	\bb{A}\ten	}
\nc{\tenB}{	\bb{B}\ten	}
\nc{\tenZ}{	\bb{\Z}\ten	}
\nc{\tenR}{	\bb{\R}\ten	}
\nc{\tenC}{	\bb{\C}\ten	}
\nc{\tenk}{	\bb{\k}\ten	}
\nc{\bten}{{\boxtimes}}         		

\nc{\con}{{ @>{\protect\cong}>> }}  	
\nc{\conl}{{ 	@>{\cong}>>	}}  	
\nc{\conn}{{    @<{\cong}<<  	}}  	
\nc{\Con}{{	\equiv		}}	
\nc{\appr}{{	\sim		}}	
\nc{\eqr}{{	\sim		}}	
\nc{\equi}{{	\sim		}}	

\nc{\fra}{ 	\frac	}     	
\nc{\ffr}[2]{{ 	\text{\footnotesize $\frac{#1}{#2}$	}	}}  

\nc{\ha}{{ \frac{1}{2} }}     		
	\nc{\half}{{ \frac{1}{2} }}    	

\nc{\ci}{{\circ}}               
\nc{\cd }{{\cdot}}            	
\nc{\cddd}{{\cdot\cdot\cdot}}	

\nc{\ox}{{	\OO_X		}}               
\nc{\omx}{{	\om_X		}}               
\nc{\Omx}{{	\Om_X^1		}}               
\nc{\Coh}{{	\CC oh		}}               %
\nc{\qcoh}{{	q\CC oh		}}               %

\nc{\xt}{{	X_*(T)		}}
\nc{\Xt}{{	X^*(T)		}}

\nc{\cfm}{{	co\fm		}}	

\nc{\cupp}{\bigcup}             
\nc{\capp}{\bigcap}
\nc{\pll}{\bigoplus}

\nc{\pii}{\prod}                
\nc{\ppii}{\bigprod}            
\nc{\cci}{\sqcup}              
\nc{\ccii}{\bigsqcup}

\nc{\wwe}{\bigwedge}            
\nc{\cce}{\bigcoprod}           

\nc{\aaa}{	\stackerel	}	

\nc{\edd}{{ \end{document}	}} 

\nc{\tx}{	\text		}		
\nc{\df}{{ \protect\overset{ \text{def}}= 	}}		
\nc{\dff}{{ \ \df\				}}		
\nc{\inv}{{ {}^{-1}      }}			

\nc{\thh}{	^{\text{th}}	}                     	
\nc{\st}{	^{\text{st}}	}                     	
\nc{\nd}{	^{\text{nd}}	}                     	
\nc{\rd}{	^{\text{rd}}	}                     	

\nc{\pmo}{{ 	\pm 1		}}
\nc{\mpo}{{ 	\mp 1		}}

\nc{\htt}{  \text{ht}}				

\nc{\emp}{{   \emptyset}}      			

\nc{\cowe}{{	\vee	}}			
\nc{\we}{{\wedge}}				
\nc{\wee}{{	\aa{\bullet}\wedge	}}		
\nc{\wetwo}{{     \pr\overset{2}\wedge       }}	

\nc{\limp}{{	\pr\underset {\leftarrow} \lim		}}	
\nc{\Limp}{{	\pr\underset {\leftarrow} {\bbb\lim}	}}	
\nc{\limi}{{	\pr\underset {\rightarrow}\lim		}}      
\nc{\Limi}{{	\pr\underset {\rightarrow}{\bbb\lim}	}}	
\nc{\llim}[1]{	 \bb{#1}\lim        	}   
\nc{\llimp}[1]{ \bb{#1}{ \pr\underset {\leftarrow} \lim       } }
\nc{\LLimp}[1]{ \bb{#1}{ \pr\underset {\leftarrow} {\bbb\lim} } } 	
\nc{\llimi}[1]{ \bb{#1}{ \pr\underset {\rightarrow}\lim       } }
\nc{\LLimi}[1]{ \bb{#1}{ \pr\underset {\rightarrow}{\bbb\lim} } }	
						
\let\Bbb\mathbb
\nc{\ppn}{{ {\Bbb P}^n }}            		
\nc{\pt}{	{ \text{pt} }	}		
				
\nc{\qlb}{{ \barr{{\Bbb Q}_l} }}      		
\nc{\ffq}{{  {\Bbb F}_q  }}           		
\nc{\ffp}{{  {\Bbb F}_p  }}           		
\nc{\tw}{   {}^{(1)}	}		

\nc{\Ab}{{ 	\AA b 		}}      		%
\nc{\Set}{{ 	\SS et 		}}      		%
\nc{\Top}{{ 	\TT op 		}}      		%
\nc{\Pic}{{ 	\tx{Pic}	}}      		%

\nc{\del}{{\partial }}
\nc{\delb}{{\partial }}
\nc{\dd}[2]{	\fra{d{#1}}{d{#2}}		}
\nc{\ddel}[2]{	\fra{\del{#1}}{\del {#2}}	}
                                        
\nc{\Spec}{{ 	\text{Spec}      		}} 
\nc{\Specf}{{ 	\text{Specf}      		}} 
\nc{\Spf}{{ 	\text{Spf}      		}} 
\nc{\hk}{{     \text{hyperk\"ahler} 	}}
\nc{\susy}{{\text{supersymmetry}}}

\nc{\ie}{{,\ \     \text{i.e.,}\ \ 	}}
\nc{\iif}{{\ \     \text{if}\ \ 	}}
\nc{\aand}{{\ \ \  \text{and}\ \ \ 	}}
\nc{\hence}{{\ \ \ \text{hence}\ \ \ 	}}
\nc{\while}{{\ \ \ \text{while}\ \ \ 	}}
\nc{\with}{{\ \ \  \text{with}\ \ \ 	}}
\nc{\oor}{{\ \     \text{or}\ \ 	}}
\nc{\foor}{{\ \     \text{for}\ \ 	}}
\nc{\suchthat}{{\ \     \text{such that}\ \ 	}}

\nc{\rk}{{\operatorname{rk}}}
\nc{\Ker}{{\operatorname{Ker}}}
\nc{\Coker}{{\operatorname{Coker}}}
\rc{\Im}{{ 	\text{Im} 	}}
\nc{\rank}{{	\ \text{rank} 	}}
\nc{\Res}{{	\  \text{Res}   }}
\nc{\Hom}{{\operatorname{Hom}}}
\nc{\End}{{	\text{End}	}}
\nc{\RHom}{{	\text{RHom}	}}
\nc{\HHom}{{	\text{$\HH$om}	}}
\nc{\RHHom}{{	\text{R$\HH$om} }}
\nc{\RGa}{{	\text{R$\Ga$}	}}
\nc{\EEnd}{{	\text{$\EE nd$}	}}
\nc{\AAut}{{	\text{$\AA ut$}	}}
\nc{\Ext}{{\operatorname{Ext}}}
\nc{\Tor}{{\operatorname{Tor}}}
\nc{\Der}{{	\text{Der}	}}

\nc{\ord	}{{ \text{ord} }}			
\nc{\divv	}{{ \text{div} }}			
\nc{\Lie	}{{ \text{Lie} }}

\nc{\timA} {{   \pr\underset{A}\tim             }}
\nc{\timB} {{   \pr\underset{B}\tim             }}
\nc{\timC} {{   \pr\underset{C}\tim             }}
\nc{\timG} {{   \pr\underset{G}\tim             }}
\nc{\timH} {{   \pr\underset{H}\tim             }}
\nc{\timN} {{   \pr\underset{N}\tim             }}
\nc{\timP}{{    \pr\underset{P}\tim             }}
\nc{\timQ}{{    \pr\underset{Q}\tim             }}
\nc{\timS} {{   \pr\underset{S}\tim             }}
\nc{\timT} {{   \pr\underset{T}\tim             }}
\nc{\timU} {{   \pr\underset{U}\tim             }}
\nc{\timV} {{   \pr\underset{V}\tim             }}
\nc{\timX} {{   \pr\underset{X}\tim             }}
\nc{\timY} {{   \pr\underset{Y}\tim             }}
\nc{\timZ} {{   \pr\underset{Z}\tim             }}

\nc{\ab}{{       ^{\text{ab}}   		}}
\nc{\af}{{       ^{\text{aff}}  		}}

\nc{\cod}{\text{codim}}	

\rc{\AA}{{\cal A}}
\nc{\BB}{{\cal B}} 
\nc{\CC}{{\cal C}}
\nc{\DD}{{\cal D}}
\nc{\EE}{{\cal E}}
\nc{\FF}{{\cal F}}
\nc{\GG}{{\cal G}}
\nc{\HH}{{\cal H}}
\nc{\II}{{\cal I}}
\nc{\JJ}{{\cal J}}
\nc{\KK}{{\cal K}}
\nc{\LL}{{\cal L}}
\nc{\MM}{{\cal M}}
\nc{\NN}{{\cal N}}
\nc{\OO}{{\cal O}}
\nc{\PP}{{\cal P}}
\nc{\QQ}{{\cal Q}}
\nc{\RR}{{\cal R}}
\rc{\SS}{{\cal S}}
\nc{\TT}{{\cal T}}
\nc{\UU}{{\cal U}}
\nc{\VV}{{\cal V}}
\nc{\WW}{{\cal W}}
\nc{\ZZ}{{\cal Z}}
\nc{\XX}{{\cal X}}
\nc{\YY}{{\cal Y}}

\nc{\A}{{\Bbb A }}
\nc{\B}{{\Bbb B}}
\nc{\C}{{\Bbb C}}
		\nc{\cc}{{\Bbb C}}
\nc{\Cs}{{\Bbb C^*}}
		\nc{\cs}{{\Bbb C^*}}
		\nc{\ccs}{{\Bbb C^*}}
\nc{\D}{{\Bbb D}}
\nc{\E}{{\Bbb E}}
\nc{\F}{{\Bbb F}}
\nc{\G}{{\Bbb G}}
	\nc{\hH}{{\Bbb H}}
\nc{\I}{{\Bbb I}}
\nc{\J}{{\Bbb J}}
\nc{\K}{{\Bbb K}}
	\nc{\lL}{{\Bbb L}}
\nc{\M}{{\Bbb M}}
\nc{\N}{{\Bbb N}}
	\nc{\oO}{{\Bbb O}}
	\nc{\pP}{{\Bbb P}}      
\nc{\Q}{{\Bbb Q}}
	\nc{\sS}{{\Bbb S}}
\nc{\T}{{\Bbb T}}
\nc{\U}{{\Bbb U}}
\nc{\V}{{\Bbb V}}
\nc{\W}{{\Bbb W}}
\nc{\Z}{{\Bbb Z}}
\nc{\X}{{\Bbb X}}
\nc{\Y}{{\Bbb Y}}

\let\P\pP

\nc{\fA}{{\frak A}}
\nc{\fB}{{\frak B}}
\nc{\fC}{{\frak C}}
\nc{\fD}{{\frak D}}
\nc{\fE}{{\frak E}}
\nc{\fF}{{\frak F}}
\nc{\fG}{{\frak G}}
\nc{\fH}{{\frak H}}
\nc{\fI}{{\frak I}}
\nc{\fJ}{{\frak J}}
\nc{\fK}{{\frak K}}
\nc{\fL}{{\frak L}}
\nc{\fM}{{\frak M}}
\nc{\fN}{{\frak N}}
\nc{\fO}{{\frak O}}
\nc{\fP}{{\frak P}}
\nc{\fQ}{{\frak Q}}
\nc{\fR}{{\frak R}}
\nc{\fS}{{\frak S}}
\nc{\fT}{{\frak T}}
\nc{\fU}{{\frak U}}
\nc{\fV}{{\frak V}}
\nc{\fW}{{\frak W}}
\nc{\fZ}{{\frak Z}}
\nc{\fX}{{\frak X}}
\nc{\fY}{{\frak Y}}
\nc{\fa}{{\frak a}}
\nc{\fb}{{\frak b}}
\nc{\fc}{{\frak c}}
\nc{\fd}{{\frak d}}
\nc{\fe}{{\frak e}}
\nc{\ff}{{\frak f}}
\nc{\fg}{{\frak g}}
\nc{\fh}{{\frak h}}
\nc{\fiI}{{\frak i}}  
	\nc{\ffi}{{\frak i}}  
\nc{\fj}{{\frak j}}
\nc{\fk}{{\frak k}}
\nc{\fl}{{\frak{l}}}
\nc{\fm}{{\frak m}}
\nc{\fn}{{\frak n}}
\nc{\fo}{{\frak o}}
\nc{\fp}{{\frak p}}
\nc{\fq}{{\frak q}}
\nc{\fr}{{\frak r}}
\nc{\fs}{{\frak s}}
\nc{\ft}{{\frak t}}
\nc{\fu}{{\frak u}}
\nc{\fv}{{\frak v}}
\nc{\fw}{{\frak w}}
\nc{\fz}{{\frak z}}
\nc{\fx}{{\frak x}}
\nc{\fy}{{\frak y}}

\nc{\al}{{\alpha }}
\nc{\be}{{\beta }}
\nc{\ga}{{\gamma }}
\nc{\de}{{\delta }}
\nc{\ep}{{\varepsilon }}
\nc{\vap}{{\epsilon }}

\nc{\ze}{{\zeta }}
\nc{\et}{{\eta }}
\rc{\th}{{\theta }}
\nc{\vth}{{\vartheta }}

\nc{\io}{{\iota }}
\nc{\ka}{{\kappa }}
\nc{\la}{{\lambda }}
\nc{\vpi}{{	\varpi		}}
\nc{\vrho}{{	\varrho		}}
\nc{\si}{{	\sigma 		}}
\nc{\ups}{{	\upsilon 	}}
\nc{\vphi}{{	\varphi 	}}
\nc{\om}{{	\omega 		}}

\nc{\Ga}{{\Gamma }}
\nc{\De}{{\Delta }}
\nc{\nab}{{\nabla}}
\nc{\na}{{\nabla}}
\nc{\Th}{{\Theta }}
\nc{\La}{{\Lambda }}
\nc{\Si}{{\Sigma }}
\nc{\Ups}{{\Upsilon }}
\nc{\Om}{{\Omega }}

\nc{\Aa}{{	\text{A}	}}
\nc{\Bb}{{	\text{B}	}}
\nc{\Cc}{{	\text{C}	}}
\nc{\Dd}{{	\text{D}	}}
\nc{\Ee}{{	\text{E}	}}
\nc{\Ff}{{	\text{F}	}}
\nc{\Gg}{{	\text{G}	}}
\nc{\Hh}{{	\text{H}	}}
\nc{\Ii}{{	\text{I}	}}
\nc{\Jj}{{	\text{J}	}}
\nc{\Kk}{{	\text{K}	}}
\nc{\Ll}{{	\text{L}	}}
\nc{\Mm}{{	\text{M}	}}
\nc{\Nn}{{	\text{N}	}}
\nc{\Oo}{{	\text{O}	}}
\nc{\Pp}{{	\text{P}	}}
\nc{\Qq}{{	\text{Q}	}}
\nc{\Rr}{{	\text{R}	}}
\nc{\Ss}{{	\text{S}	}}
\nc{\Tt}{{	\text{T}	}}
\nc{\Uu}{{	\text{U}	}}
\nc{\Vv}{{	\text{V}	}}
\nc{\Ww}{{	\text{W}	}}
\nc{\Zz}{{	\text{Z}	}}
\nc{\Xx}{{	\text{X}	}}
\nc{\Yy}{{	\text{Y}	}}

\nc{\bGa}{{	\bbb{\Ga}	}}
\nc{\bA}{{	\bbb{A}		}}
\nc{\bB}{{	\bbb{B}		}}
\nc{\bC}{{	\bbb{C}		}}
\nc{\bD}{{	\bbb{D}		}}
\nc{\bE}{{	\bbb{E}	}}
\nc{\bF}{{	\bbb{F}	}}
\nc{\bG}{{	\bbb{G}	}}
\nc{\bH}{{	\bbb{H}	}}
\nc{\bI}{{	\bbb{I}	}}
\nc{\bJ}{{	\bbb{J}	}}
\nc{\bK}{{	\bbb{K}	}}
\nc{\bL}{{	\bbb{L}	}}
\nc{\bM}{{	\bbb{M}	}}
\nc{\bN}{{	\bbb{N}	}}
\nc{\bO}{{	\bbb{O}	}}
\nc{\bP}{{	\bbb{P}	}}
\nc{\bQ}{{	\bbb{Q}	}}
\nc{\bR}{{	\bbb{R}	}}
\nc{\bS}{{	\bbb{S}	}}
\nc{\bT}{{	\bbb{T}	}}
\nc{\bU}{{	\bbb{U}	}}
\nc{\bV}{{	\bbb{V}	}}
\nc{\bW}{{	\bbb{W}	}}
\nc{\bX}{{	\bbb{X}	}}
\nc{\bY}{{	\bbb{Y}	}}
\nc{\bZ}{{	\bbb{Z}	}}
\nc{\ba}{{	\bbb{a}	}}
			\nc{\bbbb}{{	\bbb{b}	}}
\nc{\bc}{{	\bbb{c}	}}
\nc{\bd}{{	\bbb{d}	}}
			\nc{\bbe}{{	\bbb{e}	}}
			\nc{\bbf}{{	\bbb{f}	}}
\nc{\bg}{{	\bbb{g}	}}
\nc{\bh}{{	\bbb{h}	}}
			\nc{\bbi}{{	\bbb{i}	}}
\nc{\bj}{{	\bbb{j}	}}
			\nc{\bbk}{{	\bbb{k}	}}
\nc{\bl}{{	\bbb{l}	}}
\nc{\bm}{{	\bbb{m}	}}
\nc{\bn}{{	\bbb{n}	}}
\nc{\bo}{{	\bbb{o}	}}
\nc{\bp}{{	\bbb{p}	}}
\nc{\bq}{{	\bbb{q}	}}
\nc{\br}{{	\bbb{r}	}}
\nc{\bs}{{	\bbb{s}	}}
\nc{\bt}{{	\bbb{t}	}}
			\nc{\bbbu}{{	\bbb{u}	}}
\nc{\bv}{{	\bbb{v}	}}
\nc{\bw}{{	\bbb{w}	}}
\nc{\bxx}{{	\bbb{x}	}}
\nc{\by}{{	\bbb{y}	}}
\nc{\bz}{{	\bbb{z}	}}

\nc{\sA}{{\mathsf A}}
\nc{\sB}{{\mathsf B}}
\nc{\sC}{{\mathsf C}}
\nc{\sD}{{\mathsf D}}
\nc{\sE}{{\mathsf E}}
\nc{\sF}{{\mathsf F}}
\nc{\sG}{{\mathsf G}}
\nc{\sH}{{\mathsf H}}
\nc{\sI}{{\mathsf I}}
\nc{\sJ}{{\mathsf J}}
\nc{\sK}{{\mathsf K}}
\nc{\sL}{{\mathsf L}}
\nc{\sM}{{\mathsf M}}
\nc{\sN}{{\mathsf N}}
\nc{\sO}{{\mathsf O}}
\nc{\sP}{{\mathsf P}}
\nc{\sQ}{{\mathsf Q}}
\nc{\sR}{{\mathsf R}}
\rc{\sS}{{\mathsf S}}
\nc{\sT}{{\mathsf T}}
\nc{\sU}{{\mathsf U}}
\nc{\sV}{{\mathsf V}}
\nc{\sW}{{\mathsf W}}
\nc{\sX}{{\mathsf X}}
\nc{\sY}{{\mathsf Y}}
\nc{\sZ}{{\mathsf R}}
\nc{\sa}{{\mathsf a}}
\rc{\sb}{{\mathsf b}}
\rc{\sc}{{\mathsf c}}
\nc{\sd}{{\mathsf d}}
\nc{\sg}{{\mathsf g}}
\nc{\sh}{{\mathsf h}}
\nc{\sj}{{\mathsf j}}
\nc{\sk}{{\mathsf k}}
\nc{\sn}{{\mathsf n}}
\nc{\so}{{\mathsf o}}
\nc{\sq}{{\mathsf q}}
\nc{\sr}{{\mathsf r}}
\nc{\su}{{\mathsf u}}
\nc{\sv}{{\mathsf v}}
\nc{\sw}{{\mathsf w}}
\nc{\sx}{{\mathsf x}}
\nc{\sy}{{\mathsf y}}
\nc{\sz}{{\mathsf z}}

\nc{\toc}{{ 	\small{\tableofcontents} }}
\nc{\addl}{	\addcontentsline{toc}{subsection}	}

\nc{\all}{{ ^{(\alpha)} }}
\nc{\bee}{{ ^{(\beta)} }}
\nc{\gaa}{{ ^{(\gamma)} }}
\nc{\nnnn}{{ ^{ ( n ) } }}     
\nc{\nnn}{{ ^{ [ n ] } }}     
\nc{\GK}{{  	G(\KK)		}}
\nc{\GO}{{  	G(\OO)		}}

\nc{\Kh}{\tx{Ka\"hler\ }}
\nc{\Khs}{\tx{Ka\"hler structure\ }}
\nc{\Khss}{\tx{Ka\"hler structures\ }}

\nc{\GKh}{\tx{Generalized Ka\"hler\ }}
\nc{\GKs}{\tx{Generalized Ka\"hler structure\ }}
\nc{\GKss}{\tx{Generalized Ka\"hler structures\ }}
\nc{\gKs}{\tx{Generalized Ka\"hler structure\ }}
\nc{\gKss}{\tx{Generalized Ka\"hler structures\ }}

\nc{\sYM}{\text{super Young-Mills\ }}
\nc{\tFT}{\text{topological Field Theory\ }}
\rc{\top}{\tx{topological\ }}
\rc{\Top}{\tx{Topological\ }}
\nc{\TFT}{\text{Topological Field Theory\ }}
\nc{\TQFT}{\text{Topological Quantum Field Theory\ }}
\nc{\TQFTs}{\text{Topological Quantum Field Theories\ }}
\nc{\QFT}{\text{Quantum Field Theory\ }}
\nc{\QFTs}{\text{Quantum Field Theories\ }}
\nc{\FT}{\text{Field Theory\ }}

\nc{\HM}{\text{Hitchin moduli\ }}
\nc{\Hf}{\text{Hitchin fibration\ }}
\nc{\Wi}{\text{Wilson\ }}
\nc{\Wo}{\text{Wilson operator\ }}
\nc{\Wos}{\text{Wilson operators\ }}
\nc{\tH}{\text{t'Hooft\ }}
\nc{\tHo}{\text{t'Hooft operator\ }}
\nc{\Ho}{\text{t'Hooft operator\ }}
\nc{\tHos}{\text{t'Hooft operators\ }}
\nc{\Hos}{\text{t'Hooft operators\ }}
\nc{\Sd}{\text{S-duality\ }}
\nc{\Ld}{\text{Langlands duality\ }}
\nc{\Be}{\text{Bogomolny equations\ }}
\rc{\d}{\text{duality\ }}
\nc{\He}{\text{Hecke\ }}
\nc{\Heo}{\text{Hecke operators\ }}
\nc{\Hem}{\text{Hecke modifications\ }}
\nc{\Bgs}{\tx{Bogomolny equations\ }}
\nc{\Bg}{\tx{Bogomolny equation\ }}
\nc{\Hk}{{\text{Hyperk$\ddot{a}$hler} }}

\nc{\eqq}{{ 	\ =\			}}
\nc{\Cy}{{ 	C_\yy	}}
\nc{\Ay}{{ 	A_\yy	}}
\nc{\Ly}{{ 	L_\yy	}}
\nc{\Cm}{{ 	C_m	}}
\nc{\Am}{{ 	A_m	}}
\nc{\Lm}{{ 	L_m	}}

\rc{\Cy}{{	Cat_\yy			}}
\nc{\Cey}{{	Cat^{ex}_\yy		}}
\nc{\Cyt}{{	\tii{Cat}_\yy		}}

\nc{\Gm}{{	G_m			}}

\nc{\Uo}{{	U(1)	}}
\nc{\sqt}{{	\sqrt{2}	}}

\nc{\BGB}{{B\bss G/B}}

\nc{\Glb}{{ \barr{\GG_\la}	}}

\nc{\mut}{{	\mu_2	}}

\nc{\dx}{{	\dot x		}}
\nc{\ddx}{{	\ddot x		}}

\nc{\dy}{{	\dot y		}}
\nc{\ddy}{{	\ddot y		}}

\nc{\du}{{	\dot u		}}
\nc{\ddu}{{	\ddot u		}}

\nc{\Cyy}{	C^\yy		}

\nc{\hh}{{	\hatt\fh		}}
\nc{\hhp}{{	\hatt\fh_+	}}
\nc{\hhm}{{	\hatt\fh_-	}}
\nc{\jh}{{	J_{\fh}		}}
\nc{\jhs}{{	J_{\fh}		}}
\nc{\negg}{{ _{<0} 	}}
\nc{\pp}{{ _{>0} 	}}

\nc{\zb}{{\bar z}}
\let\d\del
\nc{\dbar}{{\bar \del}}

\nc{\p}{\psi}
\nc{\pb}{{	\bar\psi	}}

\nc{\pd}{{	\dot\psi	}}
\nc{\pbd}{{	\dot{\bar\psi}	}}

\nc{\pos}{{	\tx{\tiny{po}}	}}
\nc{\mom}{{	\tx{\tiny{mo}}	}}

\nc{\vac}{{	\tx{\tiny{vac}}	}}

\nc{\mb}[1]{{	\mbox{$#1$}	}}

\nc{\GGG}{{	\bbb \GG	}}

\nc{\cdG}{{	\bb{G}\cd	}}
\nc{\cdGs}{{	\bb{G^*}\cd	}}

\rc{\l}{{	\bbb l		}}
\nc{\pms}{{	\{\pm\}	}}

\nc{\yh}{{	\hatt\yy	}}
\nc{\hy}{{	\hatt\yy	}}

\nc{\bpl}{{\boxplus}}
\nc{\bcd}{{\boxdot}}

\nc{\gau}{{	e^{-x^2/2}		 }}
\nc{\istp}{{	\fra{ 1 }{ \sqrt{2\pi} } }}
\nc{\stp}{{	\sqrt{2\pi}		 }}

\nc{\lrb}[2]{	\lb #1,#2\rb	}
\nc{\lrbb}[2]{	\lb #1|#2\rb	}

\rc{\sq}{{	\sqcup		}}

\nc{\Qg}{{	\Q_{\ge 0}	}}
\nc{\Rg}{{	\R_{\ge 0}	}}
\nc{\Qgg}{{\Q_{> 0}	}}
\nc{\Rgg}{{\R_{> 0}	}}

\nc{\dagg}{{\dagger}}

\nc{\raa}[1]{\xrightarrow{#1}}
\rc{\laa}[1]{\xleftarrow{#1} }

\begin{document}





\title[]{
A Categorical Formulation of Algebraic Geometry
}

\author{	Bradley M. Willocks		}


\maketitle

\begin{abstract}                
We construct a category, $\Omega$, of which the objects are pointed categories and the arrows are pointed correspondences. The notion of a ``spec datum" is introduced, as a certain relation between categories, of which one has been given a Grothendieck topology.  A ``geometry" is interpreted as a sub-category of $\Omega$, and a formalism is given by which such a subcategory is to be associated to a spec datum, reflecting the standard construction of the category of schemes from the category of rings by affine charts.
\end{abstract}

\toc


\se{Introduction}


We construct in the present work a category $\Omega$ of equivalence classes of pointed correspondences between pointed categories. It serves as a common category for various geometries, with the intuition that a ``standard geometric" object should be a distinguished sheaf $\mathcal{O}_{X}$ in a category of sheaves of some type. Given a topological space $(X,\tau_{X})$, and a category $\mathcal{S}$, we consider the category $Sh((X,\tau_{X}),\mathcal{S})$ of $\mathcal{S}$-valued sheaves on the topological space $(X,\tau_{X})$. Then a ``standard geometric" object is a pair $(\mathcal{O}_{X},Sh((X,\tau_{X}),\mathcal{S})^{opp}) \in Ob(\Omega)$, determined by a sheaf $\mathcal{O}_{X}$ in the category $Sh((X,\tau_{X}),\mathcal{S})$. \ftt{For technical reasons the opposite categories $(Sh((Y,\tau_{Y}),\mathcal{O}_{Y}))^{opp}$ are taken.} A ``geometry" should be considered as a subcategory of $\Omega$.

By ``standard geometry" we mean a subcategory of $\Omega$ constructed from a notion of ``affine spaces". For this we define the notion of a ``spec datum", generally denoted by ``$\bar{sp}$", consisting of a pair of functors $sp : \mathcal{R} \rightarrow \mathcal{T}$ and $\mathcal{O} : \mathcal{R} \rightarrow (\mathcal{S})^{opp}$, with the category $\mathcal{T}$ being equipped with a Grothendieck topology and an ``admissibility structure". From a given spec datum $\bar{sp}$ we construct a functor $\tilde{sp} : \mathcal{R} \rightarrow \Omega$, referred to as an ``$\Omega$-lift" of the given spec datum. In classical algebraic geometry this is the construction of the structure sheaf of an affine scheme from the topological spec functor $sp : \mathfrak{Ring}^{opp} \longrightarrow \mathfrak{Top}$ going from the opposite category of commutative rings to topological spaces. The image of this functor $\tilde{sp}$ is called the ``category of affine schemes" associated to the spec datum. 

In classical algebraic geometry, a morphism of schemes $(f,f_{\sharp}) : (X,\mathcal{O}_{X}) \rightarrow (Y,\mathcal{O}_{Y})$ is a pair of arrows, $f : X \rightarrow Y$ in $\mathfrak{Top}$ and $f_{\sharp} : \mathcal{O}_{Y} \rightarrow f_{*}(\mathcal{O}_{X})$ in the category $Sh((Y,\tau_{Y}),\mathfrak{Ring})$ of sheaves on $Y$. Each of these is associated to the functor $F : (Sh((X,\tau_{X}),\mathfrak{Ring})^{opp})^{opp} \times Sh((Y,\tau_{Y}),\mathfrak{Ring})^{opp} \longrightarrow \mathfrak{Set}$, this functor being the composition of the functor \newline $((f_{*})^{opp})^{opp} \times_{\mathfrak{Cat}} (id)^{opp})$ with the $Hom$ functor $(Sh((Y,\tau_{Y}),\mathfrak{Ring})^{opp} \times Sh((Y,\tau_{Y}),\mathfrak{Ring}) \longrightarrow \mathfrak{Set}$ (here the functor $f_{*}$ is the usual pushforward of sheaves). We will refer to the arrows in $\Omega$ of the form $[(F,f_{\sharp})] \in Arr(\Omega)$ as the ``classical" arrows of schemes.

Besides merely containing the usual morphisms of schemes, there are $\Omega$-arrows \newline $(Sh((X,\tau_{X}),\mathfrak{Ring}),\mathcal{O}_{X}) \rightarrow (Sh((Y,\tau_{Y}),\mathfrak{Ring}),\mathcal{O}_{Y})$ for any co-correspondence \newline $Sh((X,\tau_{X}),\mathfrak{Ring}) \xrightarrow{F} C \xleftarrow{G} Sh((Y,\tau_{Y}),\mathfrak{Ring})$ between the categories of sheaves. This allows for ``non-standard" arrows, which may arise from hom sets in categories $C$ besides the categories of sheaves on $X$  or $Y$. In particular, any two objects in $\Omega$ have arrows between them.

Section 2 concerns categorical preliminaries, namely ($sk$)-limits and weakly enriched categories. The former is a generalization of the concept of a limit, intended for situations in which the strict limit either does not exist, or is too restrictive. In particular, they are used in defining limits of diagrams in $\mathfrak{Cat}$, for which diagrams are not required to commute strongly, but only up to natural isomorphisms.

In (2.4), we develop some notions of ``enriched sets" (a weak version, an $sk$-associative version, and a version requiring a compatible underlying category). An enrichment of a set $S$ over a tensor category $(A,\otimes)$ is an assignment of objects $h(x,y) \in Ob(A)$ for each $x,y \in S$, with composition arrows $h(x,y)\otimes h(y,z) \rightarrow h(x,z)$. In particular, sets of ``$A$-enriched functors" are given natural $A$-enrichments. Enriched sets and limits of $\mathfrak{Cat}$-valued diagrams are used in defining n-categories, in (2.6).

Section 3 of the present work defines the large category $\Omega$ and the canonical functors $\kappa_{C} : C \longrightarrow \Omega$ associated to every small category $C$. We define in (3.3), for any $\mathfrak{Cat}$-valued functor $F : I \longrightarrow \mathfrak{Cat}$, a subcategory $\Pi(F) \subseteq \Omega$,\ftt{In the text we denote $\Pi(F)$ by ``$\Pi_{\Gamma}(F)$".} consisting of arrows $[(\Phi,\phi)] \in Arr(\Omega)$ whose functors $\Phi$ are compositions of hom functors $Hom_{F(i)}$ (for some $i \in Ob(I)$) with functors of the form $F(f)$ (for some $f \in Arr(I)$). Given a spec datum $\bar{sp} = (sp : \mathcal{R} \longrightarrow \mathcal{T},...)$, we define a functor $F_{\bar{sp}} : \mathcal{T} \longrightarrow \mathfrak{Cat}$ so that the functors $F(f)$ are the pushforwards of sheaves. 

Recall that, given an $\Omega$-lift $\tilde{sp}$ of a spec datum $\bar{sp}$, we have defined a category of ``affine schemes", as the image of the functor $\tilde{sp} : \mathcal{R} \longrightarrow\Omega$. These affine schemes are a subcategory of $\Pi(F_{\bar{sp}})$. Now we define the ``category of schemes" for $(\bar{sp},\tilde{sp})$ as the subcategory $\Pi_{Sch}(\bar{sp},\tilde{sp}) \subseteq \Pi(F_{\bar{sp}})$ given by locally affine arrows, in analogy to the usual definition of schemes as locally affine spaces. In the classical case, in which the spec datum $\bar{sp} = (sp : \mathfrak{Ring} \longrightarrow \mathfrak{Top},...)$ is given by the usual topological spectrum, the resulting category $\Pi_{Sch}(\bar{sp},\tilde{sp})$ is equivalent to the classical category of schemes. This equivalence follows from the faithfulness of $\kappa_{\mathfrak{Ring}}$. 

Section 4 contains miscellaneous results intended to form some preliminary to future work. Section (4.1) concerns technicalities for manipulating ``admissibility structures" (a sort of ``rigid" Grothendieck topology), in order to accommodate pseudo-functors (which might respect composition only in a weakened sense). Section (4.2) contains a functor $\mathfrak{Cat} \longrightarrow \mathfrak{PreGrp}$, where $\mathfrak{PreGrp}$ is the category of semi-groups, consisting of sets with associative composition laws, which might lack inverses. We hope to use these to define homotopy and (co)homology groups in the setting of $\Omega$.


\se{Preliminaries, Limits and Enrichments}

\sss{} Sections (2.1) and (2.2) concerns set theoretical preliminaries and notation. Details regarding the set theory that we use are relevant to the inductive definition of higher categories, in (\ref{nCat}), and $\Omega$, in (\ref{omega}).

\sss{The Notion of an $(sk)$-(co)limit of a Functor} In (\ref{SkLim}) we define a variation of the concept of a limit of a functor. Given functors $J \xrightarrow{e} I \xrightarrow{F} A \xrightarrow{sk} B$, the $(sk,e)-limit$ of $F$ is defined as follows. 

One first constructs a certain category, denoted by ``$P = P(sk,e)$." Its objects are essentially pairs $(a,\alpha)$, where $a \in Ob(A)$ and $\alpha : Ob(J) \rightarrow Arr(A)$ is a natural transformation \ftt{Recall that for any categories $C,D \in Ob(\mathfrak{Cat})$, given any object $d \in Ob(D)$, we denote by $\Delta_{(C,D)}(d) : C \longrightarrow D$ the constant functor, which sends every arrow in $C$ to $id_{d}$; see (1.2.24).} $\Delta_{(J,A)}(a) \xrightarrow{\alpha} F\circ e$ which lifts to a natural transformation $\tilde{\alpha} : \Delta_{(I,B)}(sk(a)) \rightarrow sk \circ F$. This is to say that, given the notation of (2.2.23) regarding functor categories, the pullback of $\tilde{\alpha}$, the natural transformation $e^{*}(\tilde{\alpha}) = Hom_{U-\mathfrak{Cat}^{2}(1)}(e,id_{A})_{(1)} (\tilde{\alpha}) : \Delta_{(J,B)}(sk(a)) \rightarrow sk \circ F \circ e$, is equal to the pushforward of $\alpha$, the natural transformation $sk_{*}(\alpha) = Hom_{U-\mathfrak{Cat}^{2}(1)}(id_{J},sk)_{(1)}(\alpha) : \Delta_{(J,B)}(sk(a)) \rightarrow sk \circ F \circ e$. An arrow $\phi : (a,\alpha) \rightarrow (b,\beta)$ in the category $P$ is an arrow $f : a \rightarrow b$ in the category $A$ for which $\alpha(j) = \beta(j) \cdot f$ for any $j \in Ob(J)$. 

We now define the $(sk,e)$-limit of $F$ to be the colimit of the functor $P \longrightarrow A$ defined by $(a,\alpha) \mapsto a$. A ``colimit" is a pair $(l,\lambda)$ consisting of an object $l$ and a natural transformation $\lambda$ from the functor in question to the constant functor of $l$, satisfying a universal property. 

\sss{The Main Example of $sk$-limits} In cases concerning diagrams of categories, i.e. $A = U-\mathfrak{Cat}$, the functor $sk$ will usually be the quotient functor $Skel : U-\mathfrak{Cat} \longrightarrow U-\mathfrak{SCat}$, to the ``skeleton category of categories" (see \ref{SCat}), which identifies isomorphic pairs of functors. This is to say that the object set of each hom category $Hom_{U-\mathfrak{Cat}^{2}}^{(1)}(C,D)$ is replaced by the set of isomorphism classes of objects (the objects of a skeleton subcategory of the original hom category). The functor $e : J \longrightarrow I$ should be the inclusion of the subcategory of $I$ consisting of all of the identity arrows. Then the objects of the category $P(Skel,e)$ are pairs $(C,f)$ of a category $C$  and a family of functors $f_{i}:C\to F(i)$ for $i \in Ob(I)$,  
which  is a ``transformation'' of functors from a constant functor $\De_{I,C}:I\to \mathfrak{Cat}$
with value $C$, to $F$, which is ``natural up to isomorphism".
The meaning of this  is that for any arrow $g:i\to j$ in $ I$, the functors 
$F(g) \circ f(i)$ and $ f(j	)$ from $C$ to $F(j)$ are isomorphic. Then the $(sk,e)$-limit 
of $F$ is the colimit of all such categories $C$.

\sss{} We define in (2.2), for each tensor category $(A,\otimes)$ and functor $sk : A \longrightarrow B$ the category $WE_{(A,\otimes)(sk)}$ of $weakly$ $enriched$ $sets$.  These are sets $S$, with an assignment for every pair of elements $s,t \in S$ of an object $h_{S}(s,t) \in Ob(A)$, and a composition $h(s,t) \otimes h(t,u) \rightarrow h(s,u)$ for any elements $s,t,u \in S$. An enriched set is to be considered as a ``category" for which one is given ``hom objects in $A$" rather than hom sets. With this intuition, for any two enriched sets $S,T \in Ob(WE_{(A,\otimes)(sk)})$, the hom set $Hom_{WE_{(A,\otimes)(sk)}}(S,T)$ should be considered to be a ``category of $A$-functors". We give in (2.3) a construction by which one can enrich the set of ``functors" $Hom_{WE(A,\otimes)(sk)}(S,T)$ to an enriched set. 


\sss{} We define a notion $(U,n)-\mathfrak{Cat}$ of the category of $n$-categories in (2.4).\ftt{This is not one of the standard notions of an $n$-category. In particular, $(U,n)-\mathfrak{Cat}$ is not an $(n+1)$-category, i.e. an object of $(U',n+1)-\mathfrak{Cat}$. See (\ref{assncat}).} By the lemmas of the ``Enrichments" section (2.3), $(U,n)-\mathfrak{Cat}$ is naturally enriched over itself. In (2.5) we define a sort of co-simplicial structure on the category $(U,n)-\mathfrak{Cat}$ of $n$-categories. 
As usual denote by $\Delta$ the category of finite ordered sets, its arrows are order-preserving maps. Let $[n]=
\{0<\cddd< n-1\} \in Ob(\Delta)$ and denote by $\Delta_{[n]/} $ 
the category of arrows under the ordered set $[n]$.
For each $n$ we construct a functor 
$\rho : \Delta_{[n]/} \longrightarrow 
U'-\mathfrak{Cat}_{ (U,n)-\mathfrak{Cat} /} 
$ (denoted  also $\downarrow_{(U'-\mathfrak{Cat})}(ob_{(U'-\mathfrak{Cat})}(`(U.n)-\mathfrak{Cat}),id_{U'-\mathfrak{Cat}})
$).

The arrow $f_k : [m+1] = \{...,k,k+1,... \} \rightarrow [m] = \{...,k,...\}$ 
which identifies $k$ and $k+1$ is sent to the functor which ``collapses all $(n-k)$-categories into their component $(n-k-1)$-categories".\ftt{
By the collapse of a category $C$ to a set we mean the set of all morphisms
$\coprod_{(a,b)\in Ob(C)^2}\ C(a,b)$. This 
generalizes for  any $(m+1)$-category $C$, the $m$-category $\rho(f_k)(C)$ 
is such that every $k$-category $h(a,b)$ appearing in the structure of $C$ is replaced by the $(k-1)$-category $\coprod_{a',b' \in Ob(h(a,b))}\ h_{h(a,b)}(a',b')$.
}

The arrow $g_k : [m] \rightarrow [m+1]$ which omits $k$\ie ``skips the $k^{th}$ step", replaces all 
$(m-k)$-categories by their classifying $(m-k+1)$-categories.\ftt{
 having a trivial underlying set of objects and the same $(m-k)$-category as the only hom object, i.e. for any $m$-category $D$, the $(m+1)$-category $\rho(g)(D)$ is such that every $k$-category $h(a,b)$ appearing in the structure of $D$ is replaced by the $(k+1)$-category whose underlying set $\{ \emptyset \}$, is the singleton, so that $h(\emptyset,\emptyset) = h(a,b)$. Composition is trivial, being given by projection to the left (or right) component}.


\sus{Axioms - ZFC with Universes}
\label{Axioms}












We outline here the standard first order set theory, Zermelo Fraenkel with the Axiom of Choice, describing the generation of the formal statements of that language. We use the notion of Grothendieck universes, as sub-models of this same set theory, in order to handle the difficulties arising from classes.

In principle, all definitions, propositions, etc. are to be written in the language of first order logic. 

In any first order theory begins, a primitive statement may be introduced at any time, by its use.\ftt{ The list of primitive statements is a priori infinite. Assuming that any given work is ``complete," one could list all introduced primitive statements in the beginning.} We start with a class of primitive statements, given by the $=$ relation and the $\in$ relation (see \ref{Equality} and \ref{Element}). Any string of the form ``$\ulcorner x = y \urcorner$" or ``$\ulcorner x \in y \urcorner$" is a statement (Quotes ``$\ulcorner$" and ``$\urcorner$" appear around any non-negated statement, acting effectively as parentheses). The set of all statements is defined to be the closure of the set of primitive statements under the operations of quantification, negation, and conjunction. All further primitive statements are to be introduced with some assumption or definition by which they are rendered to be equivalent to some statement in the closure generated by the primitive statements given by the relations $=$ or $\in$. This is to say that, in principle, the only primitive statements are given by $=$ and $\in$, all others being essentially shorthand.

Negation is a statement modifier, formally denoted by ``$\neg$," with the usual interpretations (with Excluded Middle). Any variable symbol is brought into being by its use in any context, and each variable symbol, e.g. ``$x$" is associated to quantifiers denoted ``$\forall x$" and ``$\exists x$," which may modify any statement. A given variable, e.g. $x$ in a statement $\Psi$ is ``bound" if that statement is either of the form $\Psi = \forall x \Phi (x)$, or obtained from a statement $\Phi$ in which $x$ is bound by application either of quantifiers $\forall y$ (for variables $y$ distinct from $x$), or of negation $\neg$, i.e. of the form $\Psi = \forall y \Phi$ or $\Psi = \neg \Phi$. A statement may be regarded as meaningful, (in the sense in which a proof or disproof should be held to be, in principle,possible) only so far as every variable symbol appearing within it is bound. 

It is presumed that the variable symbols appearing within a statement may be clearly discerned (e.g. in the statement $\forall f, \ulcorner f \text{ is finite} \urcorner$, we presume that the ``f" in ``finite" should not be confused for the variable $f$). This is, in some sense, equivalent to the presumption that the quantifiers are rightly applied, or interpreted, i.e. that quantified statements appear where they are meant to appear, and the symbols appearing within a statement are rightly understood. 


The remainder of this subsection is a list of the axioms of $ZFC$. Several of the required set constructions are introduced as function symbols, i.e. maps from some tuple of the universe of discourse to itself. Parentheses indicate expected variable symbols. An axiom regarding the existence of Grothendieck universes is included.

\sss{Type}
The first order language used here is of one ``type". Every variable symbol refers to a set. There are two essential consequences. The first is that any variable symbol may be used in any formula in any position. The second is that all quantifiers have the same ``meaning", in that two statements which differ only by the substitution of one bound variable for another are semantically equivalent. \ftt{Since there is a universal equality sign (\ref{Element}), used to compare any two elements of the domain of this discourse, one can describe sets in an implicit fashion, by their use within a statement. 

In particular, in defining the set of $n$-categories (\ref{nCat}), we construct a finite set of sequences of sets $S_{i}(n)$ (each set $S_{i}$ being a sequence of sets) implicitly and inductively by requiring that each set $S_{i}(n+1)$ should be equal to a set constructed in a particular fashion from the sets $S_{j}(n)$, i.e. should satisfy a certain statement in which both $S_{i}$ and the $S_{j}$ appear as variable symbols. Since there is one type, the equality relations which appear in the statements in that definition are valid (i.e. there is no concern that one might wrongfully compare sets and classes, though the quantifiers $\forall S_{i}$ appearing in the definition range over the universe of discourse, since classes have no separate existence as referents.)

Furthermore, for any two variable symbols $x$ and $y$, for any statement $\Phi$, the formulae $\forall x \Phi(x)$ and $\forall y \Phi(y)$ are equivalent.}

This is to say that we presume that there is a single ``universe of discourse". By this we mean a class which implicitly contains all variables. The meaning of a quantifier in any given statement is determined by the consideration, in the place of the quantified variable, of all elements of this ``universe of discourse".

\sss{Equality}\label{Equality}
There is a binary relation, ``$=$", so that that statements of the form $\ulcorner x = y \urcorner$ are primitive.

\sss{Element}\label{Element}
There is a binary relation, ``$\in$," so that that statements of the form $\ulcorner x \in y \urcorner$ are primitive.

\sss{Empty Set}
The symbol $\emptyset$ denotes the ``empty set" (is a 0-ary function symbol), such that $\ulcorner\forall x, \neg\ulcorner x \in \emptyset \urcorner\urcorner$.

\sss{Singleton}
The unary function symbol ``$\{ ( \ ) \}$" denotes the singleton construction, so that 
$$
\ulcorner \forall x, \ulcorner \forall y, \ulcorner\ulcorner x \in \{y \} \urcorner \Longleftrightarrow \ulcorner x = y \urcorner\urcorner\urcorner\urcorner
$$

\sss{Union}
The binary function symbol ``$( \ ) \cup ( \ )$" denotes the union construction, so that 
$$
\ulcorner\forall x,\ulcorner\forall y, \ulcorner\forall z, \ulcorner\ulcorner x \in y \cup z \urcorner \Longleftrightarrow \ulcorner\ulcorner x \in y \urcorner \text{ or } \ulcorner x \in z \urcorner\urcorner\urcorner\urcorner\urcorner\urcorner
$$

\sss{Power Set}
The unary function symbol ``$2^{( \ )}$" denotes the power set construction, so that 
$$
\ulcorner\forall x, \ulcorner \forall y, \ulcorner\ulcorner x \in 2^{y} \urcorner \Longleftrightarrow \ulcorner \forall z, \ulcorner\ulcorner z \in x \urcorner \Longrightarrow \ulcorner z \in y \urcorner\urcorner\urcorner\urcorner\urcorner\urcorner
$$

\sss{Extension}
Extension, that a set is determined by the elements contained therein
$$
\ulcorner\forall x, \ulcorner\forall y, \ulcorner\ulcorner x = y \urcorner \Longleftrightarrow \ulcorner\forall z, \ulcorner\ulcorner z \in x \urcorner \Longleftrightarrow z \in y \urcorner\urcorner\urcorner\urcorner\urcorner\urcorner
$$

\sss{Comprehension Schema}
A schema is here used, since there is no separate type for statements of the language. By this we mean that the following is to be understood as determining an infinite list of function symbols and corresponding statements. 

For any statement $\Phi(x)$ of the language in which $x$ appears as an open $(un-quantified)$ variable, there is a unary function symbol ``$\{ ( \ ) ; \Phi(x) \}$". This sends a set $z$ to the set of all elements $y \in z$,satisfying the statement $\Phi(x)$, generally denoted by ``$\{ y \in z ; \Phi(y) \}$". Formally, we require that
$$
\ulcorner \forall z, \ulcorner \forall y, \ulcorner , \ulcorner\ulcorner y \in \{ z ; \Phi(x) \} \urcorner \Longleftrightarrow \ulcorner\ulcorner y \in z \urcorner \text{ and } \Phi(y) \urcorner\urcorner\urcorner\urcorner.
$$

\sss{Codomain}
For any set $x$ which defines a function on a set $w$, there exists a set $w'$ such that for any $(z_{1},z_{2}) \in x$, $z_{2} \in w'$. Formally,
$$
\ulcorner\forall w, \ulcorner \forall x,
$$
$$
\ulcorner\ulcorner\ulcorner\forall y, \ulcorner\ulcorner y \in x \urcorner \Longrightarrow \ulcorner \exists z_{1}, z_{1},
$$
$$
\ulcorner\ulcorner z_{1} \in w \urcorner \text{ and }  \ulcorner y = \{\{ z_{1}\}\} \cup \{\{z_{1}\} \cup \{ z_{2}\}\} \urcorner \text{ and }
$$
$$
\ulcorner\forall z_{2}', \ulcorner\ulcorner \{\{ z_{1} \}\} \cup \{\{ z_{1} \} \cup \{z_{2}'\}\} \in x \urcorner \Longrightarrow \ulcorner z_{2}' = z_{2} \urcorner\urcorner\urcorner\urcorner\urcorner\urcorner\urcorner \text{ and }
$$
$$
\ulcorner\forall z_{1}, \ulcorner\ulcorner z_{1} \in w \urcorner \Longrightarrow \ulcorner\exists z_{2}, \ulcorner \{\{ z_{1} \}\} \cup \{\{ z_{1} \} \cup \{z_{2} \}\} \in x \urcorner\urcorner\urcorner\urcorner\urcorner
$$
$$
\Longrightarrow \ulcorner \exists w', \ulcorner \forall z_{1},z_{2}, \ulcorner\ulcorner \{\{ z_{1} \}\} \cup \{\{ z_{1} \} \cup \{ z_{2} \} \} \in x \urcorner
\Longrightarrow \ulcorner z_{2} \in w' \urcorner\urcorner\urcorner\urcorner\urcorner\urcorner\urcorner
$$

\sss{Foundation}
Every set must have a minimal element with respect to the relation $\in$, of (\ref{Element}.1).
$$
\ulcorner \forall x, \ulcorner \exists y, \ulcorner\ulcorner y \in x \urcorner \text{ and } \ulcorner \forall z, \ulcorner\ulcorner z \in x \urcorner \Longrightarrow \neg \ulcorner z \in y \urcorner\urcorner\urcorner\urcorner\urcorner\urcorner
$$

\sss{Natural Numbers}
There is given a natural numbers object, i.e. nullary function symbols $\mathbb{N}$, $+_{\mathbb{N}}$, $\times_{\mathbb{N}}$, $0$, and $1$, such that $+_{\mathbb{N}},\times_{\mathbb{N}} : \mathbb{N} \times \mathbb{N} \rightarrow \mathbb{N}$ are commutative and distributive, with units $0$ and $1$ respectively, such that $\mathbb{N}, \{\{\{x \}\} \cup \{\{x \} \cup \{ +_{\mathbb{N}}(x,1) \} \} \in 2^{2^{\mathbb{N}}} ; \ulcorner x \in \mathbb{N} \urcorner \} , 0 )$ is a model of Peano arithmetic, translated into the present language for set theory.

\sss{Choice}
This is assumed in the form of Zorn's Lemma, i.e. that every inductively ordered set has a maximal element.

\sss{$\bold{Definition}$ of a Universe}
For any set $U$, $U$ is defined to be a $universe$ iff the tuple 
$$
(U, =|_{U}, \in |_{U}, \emptyset, \{ ( \ ) \} |_{U}, ( \ )\cup( \ )|_{U}, 2^{(\ )} |_{U}, (\mathbb{N},+_{\mathbb{N}},\times_{\mathbb{N}},0,1))
$$ 

satisfies the previous axioms, (\ref{Axioms}.2) - (\ref{Axioms}.13). It is assumed that
$$
\ulcorner \forall x, \ulcorner \exists U, \ulcorner\ulcorner U \text{ is a universe}\urcorner \text{ and } \ulcorner x \in U \urcorner\urcorner\urcorner\urcorner
$$

Within this setting, we define the notion of a category.

\sss{$\bold{Definition}$ of Pairs ``$(x,y)$"}
$\ulcorner \forall x, \ulcorner \forall y, \ulcorner (x,y) := \{ \{x \} \} \cup \{ \{ x \} \cup \{ y \} \} \urcorner\urcorner\urcorner$

\sss{$\bold{Definition}$ of Products ``$x \times_{\mathfrak{Set}} y$"}
$\ulcorner \forall x, \ulcorner \forall y, \ulcorner x \times_{\mathfrak{Set}} y := $ \newline $\{ (x',y') \in 2^{2^{2^{x \cup y}}} ; \ulcorner\ulcorner x' \in x \urcorner \text{ and } \ulcorner y' \in y \urcorner\urcorner \} \urcorner\urcorner\urcorner$

\sss{$\bold{Definition}$ of a Function ``$f : x \rightarrow y$"}
$\ulcorner \forall x, \ulcorner \forall y, \ulcorner \forall f, \ulcorner\ulcorner f : x \rightarrow y\urcorner \Longleftrightarrow$ \newline $\ulcorner\ulcorner f \subseteq x \times_{\mathfrak{Set}} y \urcorner \text{ and } \ulcorner \forall x', \ulcorner\ulcorner x' \in x \urcorner \Longrightarrow \ulcorner \exists ! y', \ulcorner\ulcorner y' \in y \urcorner \text{ and } \ulcorner (x',y') \in f \urcorner\urcorner\urcorner\urcorner\urcorner\urcorner\urcorner\urcorner\urcorner\urcorner$

Given two functions $f : x \rightarrow y$ and $g : y \rightarrow z$, we denote by ``$f \circ g : x \rightarrow z$" the compostion.

\sss{$\bold{Definition}$ of Associating Functions ``a"}
$\ulcorner \forall x, \ulcorner \forall y, \ulcorner \forall z, \ulcorner\forall a,$ \newline $ \ulcorner\ulcorner a$ $\text{ associates}$ $(x,y,z) \urcorner$ $\Longleftrightarrow$  $\ulcorner\ulcorner a : ((x \times_{\mathfrak{Set}} y) \times_{\mathfrak{Set}} z) \rightarrow (x \times_{\mathfrak{Set}} (y \times_{\mathfrak{Set}} z)) \urcorner \text{ and }$ \newline $\ulcorner \forall x', \ulcorner \forall y', \ulcorner \forall z', \ulcorner (((x,y),z),(x,(y,z))) \in a\urcorner\urcorner\urcorner\urcorner\urcorner\urcorner\urcorner$

\sss{$\bold{Definition}$ of a Category}\label{Category}
For any 5-tuple, $C = (O,A,H,\circ,id)$, we say that $C$ is a category iff it satisfies the following.

\ref{Category}.1. The function $H : O^{2} \rightarrow 2^{A}$ is such that $A = \bigcup_{x,y \in O}H(x,y)$ and $\forall x,x',y,y' \in O, \ulcorner\ulcorner (x,y) \neq (x',y') \urcorner \Longrightarrow \ulcorner H(x,y) \cap H(x'y') = \emptyset \urcorner\urcorner$, i.e. the set of arrows is the disjoint union of all hom sets.

\ref{Category}.2. The function $\circ$ has for its domain $O^{3}$, so that for any $x,y,z \in O$, the composition $\circ(x,y,z) : h(x,y) \times_{\mathfrak{Set}} h(y,z) \rightarrow h(x,z)$ is a function. We require the composition to be ``associative".

\ref{Category}.3. The function $id : O \rightarrow A$ is such that for any $x \in O$, for any $y \in O$, $id(x) \in h(x,x)$ is neutral with respect to the composition maps $\circ(x,x,y)$ and $\circ(y,x,x)$.

\sss{$\bold{Definition}$ of a Functor, $F = (F_{(0 )},F_{(1 )})$}\label{Functor}
A functor 
$$
F = (F_{(0)},F_{(1)}) : C = (O_{C},A_{C},H_{C},\circ_{C},id_{C}) \longrightarrow (O_{D},A_{D},H_{D},\circ_{D},id_{D}) = D
$$ 

is a pair of functions $F_{(0)} : O_{C} \rightarrow O_{D}$, $F_{(1)} : A_{C} \rightarrow A_{D}$, such that the following hold.

\ref{Functor}.1. For any $x,y \in O_{C}, \{ F_{(1)}(f) ; f \in H_{C}(x,y) \} \subseteq H_{D}(F_{(0)}(x),F_{(0)}(y))$.

\ref{Functor}.2. For any $x,y,z \in O_{C}$, for any $f \in H_{C}(x,y),g \in H_{C}(y,z)$, 
$$
\circ_{D}(F_{(0)}(x),F_{(0)}(y),F_{(0)}(z)) ((F_{(1)}(f),F_{(1)}(g))) = F_{(1)}(\circ_{C}(x,y,z)((f,g)))
$$

\ref{Functor}.3. For any $x \in O_{C}$, $F_{(1)}(id_{C}(x)) =  id_{D}(F_{(0)}(x))$.

If F is a functor, $F_{(0)}$ is the map on ``objects" and $F_{(1)}$ is the map on ``arrows". Often the subscripts will be omitted, and ``$F$" will refer to either the object map or the arrow map, e.g. $F(c)$, $F(\phi)$. The reader used to the standard notation can ignore the subscripts.


\sus{Notation} 
\label{Notation}
Parentheses indicate expected arguments.

\sss{}
Temporary definitions are denoted by $:_{t}=$. They are valid only within the definition, proposition, paragraph, etc. in which they appear. Definitions denoted by $:=$ are valid for all subsequent definitions, propositions, etc.

\sss{}
``$U$" will generally denote a universe, i.e. a model of set theory, and $U'$ will denote some universe containing $U$, i.e. $U \in U' \in U'' \in ...$.

\sss{}
If $U$ is some universe, then $U-\mathfrak{Cat}$ is the category of $U-small$ categories, i.e. its objects are categories $C = (O,A,H,\circ,id)$ for which $O,A \in U$, i.e. the set of arrows is $U$-small, and its arrows are functors $F : C \rightarrow D$.
We usually denote $U-\mathfrak{Cat}$
just by $\mathfrak{Cat}$ and unless otherwise stated ``category'' means a ``$U$-category''.

\sss{}
``$a\circ b$" denotes ``$a \circ_{C} b$", the composition of arrows in a category, as well as the composition of functions, where the context should eliminate the ambiguity.

\sss{}
For any category $C = (O,A,H,\circ,id)$, $Ob(C) = O$ is the class of objects of the category, and $Arr(C) = A$ is the class of arrows of $C$.

\sss{}
The functions $dom_{(C)}, codom_{(C)} : Arr(C) \longrightarrow Ob(C)$ are the domain and codomain maps for the category $C$.

\sss{}
The category
$
\star \text{ is the terminal category, having one arrow.}
$

\sss{}
The dual category functor\
$
( )^{opp} : \mathfrak{Cat} \longrightarrow \mathfrak{Cat} 
\text{ sends a category to its opposite}
$.
It may carry an  index based on   (i) when one wants to differentiate
between the map on objects, and the map on arrows.
Then it would be written as
$_{( )}( )^{opp}$.\ftt{ 
We move the subscript to the left because the aesthetic sense resists the appearance of a subscript of a superscript. The attachment of the subscript should be to the functor $^{opp}$, rather than to the category or functor on which it acts, and there is a fear that placement of the subscript after the term would suggest its attachment to the category or functor in the argument.
}

\sss{Functorial Products}

Let ``$\times$" denote the product of two categories (i.e. the objects of $C \times D$ are pairs of objects $(c,d)$ from the component categories and the arrows are pairs $(f,g) : (c,d) \rightarrow (c',d')$ of arrows from the component categories). Then there are two canonical product functors, one for categories and one for sets, defined as follows.

For sets, 

$\times_{U-\mathfrak{Set}} : (U-\mathfrak{Set})^{opp} \times U-\mathfrak{Set} \longrightarrow U-\mathfrak{Set}$

sends a pair of sets to their product, $(a,b) \mapsto \{ \pi \in Hom_{(U-\mathfrak{Set})}(\{a,b \} , a\cup b) ;  \pi(a) \in a \text{ and } \pi(b) \in b \}$ and a pair of functions to their product, so that $(f,g) \in Arr(U-\mathfrak{Set} \times U-\mathfrak{Set})$ is sent to the map from $a \times_{U-\mathfrak{Set}} b$ to $a' \times_{U-\mathfrak{Set}} b'$ determined by applying $f$ to the $a$ component and $g$ to the $b$ component, i.e. $(x,y) \mapsto (f(x),g(y))$. 

For categories,

$\times_{U-\mathfrak{Cat}} : (U-\mathfrak{Cat})^{opp} \times U-\mathfrak{Cat} \longrightarrow U-\mathfrak{Cat}$

sends a pair of categories to their product i.e. (a $U$-small restriction of $\times$ above) and sends a pair of functors to their product, $(F,G) \mapsto \newline ( ((a,b) \mapsto (F_{(0)}(a),G_{(0)}(b))_{a \in Ob(dom(F)), b \in Ob(dom(G))} , \newline ((f,g) \mapsto (F_{(1)}(f) , G_{(1)}(g)))_{f \in Arr(dom(F)), g \in Arr(dom(G))} )$.

Note that $\times_{U-\mathfrak{Set}}$ and $\times_{U-\mathfrak{Cat}}$ do {\em not} denote fibred products in the category of $U'$-categories.

\sss{}

For any category $C$, for any $\phi, \psi,f,g \in Arr(C)$, we define a relation \newline $(f,g) fibres_{(C)}(\phi,\psi)$ (``$f$ and $g$ fibre $\phi$ and $\psi$ in $C$") iff $codom(\phi) = codom(\psi)$ and $dom(f) = dom(g)$ and $f,g$ form a fibred product of $\phi,\psi$ (as in [1]). Informally, the $C$ may be omitted if it is understood.

\sss{}
The hom functor map $Hom_{()}$ is a map of large sets which assigns to a category $C$ the functor 
$(C)^{opp} \times C \longrightarrow \mathfrak{Set}$ which sends a pair of objects $(a,b)$ to their hom set, i.e. $(a,b) \mapsto Hom_{C}(a,b)$, and sends a pair of arrows $(f,g)$ to the map of hom sets given by composition, i.e. $(g,f) \mapsto (h \mapsto f \circ h \circ g )_{h \in Hom_{C}(codom(g),dom(f))}$, which lies in $Hom_{\mathfrak{Set}} (Hom_{C}(codom(g),dom(f)) , Hom_{C}(dom(g),codom(f)))$. In the standard notation it is written ``$Hom_{C}$." We include the parentheses from the point of view that ``$Hom_{C}"$ is the function $Hom_{( \ )}$ evaluated at $C$. \ftt{The hom map can be thought of as a single map (that on objects, assigning the hom set to a pair of objects), or a pair of maps (as given, a functor). Strictly speaking, the former is part of the data which determines a category, while the latter exists only after a definite category is constructed}

\sss{Category $
\downarrow_{( )}(,)
$ of arrows (in a third category) between two categories}. 

 If $A\raa{F}C$ and $B\raa{G}C$ are functors with 
the same codomians 
$codom_{(\mathfrak{Cat})}(F) = C = $ \newline $codom_{(\mathfrak{Cat})}(G)$, 
then 
$\downarrow_{(C)}(F,G)$ is the category of arrows from 
$A$ to $B$ in $C$ with respect to $F$ and $G$. 
So. the objects are triples $(a,\phi,b)$ where $a \in Ob(dom(F))$ and $b \in Ob(dom(G))$ and $\phi \in Hom_{C}(F(a),G(b))$, and the arrows $(a,\psi,b) \longrightarrow (a',\phi,b')$ are pairs $(f,g)$ where $f : a \rightarrow a'$ and $g : b \rightarrow b'$ such that $G(g) \cdot \psi = \phi \cdot F(f)$.

\sss{}
The symbol $ob_{( )}( )$ assigns to any category $C$ and an object $c \in Ob(C)$ the ``object functor," $ob_{(C)}(c)$, which sends the category with one arrow $ \star =( \{ \emptyset \} , \{ \emptyset \}, ... ) \xrightarrow{ob_{(C)}(c)} C$ to $C$ by mapping $\emptyset \mapsto c$ on objects and $\emptyset \mapsto id_{c}$ on arrows.

\sss{}
The {\em domain object functor}
$
dob
\downarrow_{( )}( , )$. 
If again $A\raa{F}C\laa{G}B$, then 
$$
\downarrow_{(C)}(F,G) 
\raa{
dob\downarrow_{(C )}(F ,G )
} 
dom_{(\mathfrak{Cat})}(F)
$$
is defined by sending any triple
$(a,\phi,b)$ in $\downarrow_{(C)}(F,G) $
to $a$\ie
one only remembers the domain of the arrow.

\sss{}
The {\em codomain object functor} $cob\downarrow_{( )}( , )$. If again $A\raa{F} C \laa{G} B$, then 
$$
\downarrow_{(C)}(F,G) 
\raa{
cob\downarrow_{(C )}(F ,G )
} 
dom_{(\mathfrak{Cat})}(G)
$$

is defined by sending a triple $(a,\phi,b)$ in $\downarrow_{(C)}(F,G)$ to $b$, i.e. one remembers only the codomain of the arrow.

\sss{
Symbol $
\langle ( ) \rangle_{Full( )}
$}
Here,
  $\langle S \rangle_{Full(C)}$ means the full subcategory of $C$ generated from $S \subseteq Ob(C)$.

\sss{}
$
\langle ( ) \rangle_{CatA( )}
$

  $\langle S \rangle_{CatA(C)}$ generates a subcategory from $S \subseteq Arr(C)$

\sss{}
$
\langle ( ) \rangle_{Equiv( )}
$

  $\langle R \rangle_{Equiv(S)}$ is the equivalence relation on $S$ generated by $R \subseteq S \times S$.

\sss{}
$
[ ]_{( )}
$

  $[f]_{(R)}$ is the equivalence class of $f$ with respect to the equivalence relation $R$.

\sss{}
\label{TCat}
$
\mathfrak{TCat}
$

is the $\textit{category of tensor categories}$. Objects are pairs $(A,\otimes : A \times A \rightarrow A)$, and arrows are pairs $(F : A \rightarrow B, \rho : \otimes_{B} \circ (F \times F) \rightarrow F \cdot \otimes_{A} )$, where $\rho$ is a natural transformation of functors (see [1]). Note that we do not require that $\rho$ should be an isomorphism, or that $(A,\otimes)$ should be equipped with an associator or a unit.

\sss{}
\label{ATCat}
$
\mathfrak{ATCat}
$

is the $\textit{category of associative tensor categories}$. Objects are triples $(A,\otimes,\alpha)$, where $(A,\otimes)$ $\in$ $Ob(\mathfrak{TCat})$ is a tensor category and $\alpha : \otimes \cdot (\otimes \times_{\mathfrak{Cat}} id_{A}) \rightarrow \otimes \cdot ( id_{A} \times_{\mathfrak{Cat}} \otimes) \cdot \alpha_{\mathfrak{Cat}}(A,A,A)$ is a natural isomorphism, where $\alpha_{\mathfrak{Cat}}(A,A,A) \in Hom_{\mathfrak{Cat}^{2}}( (A \times_{\mathfrak{Cat}} A) \times_{\mathfrak{Cat}} A , A \times_{\mathfrak{Cat}} (A \times_{\mathfrak{Cat}} A)))$ is the usual associator for the product category functor $\times_{\mathfrak{Cat}} : \mathfrak{Cat} \times \mathfrak{Cat} \longrightarrow \mathfrak{Cat}$. We refer to the natural isomorphism $\alpha$ as an ``associator" for $(A,\otimes)$.\ftt{In contrast to the usual definition, we do not require a self-consistency
condition for the associativity constraint $\al$ (the pentagon condition which requires that the two ways of ``associating" the tensor product of four objects should be the same).
Nether do we require unital structures or properties.}

\sss{The Functor $Hom_{U-\mathfrak{Cat}^{2}}$}
We do not define $U-\mathfrak{Cat}^{2}$, the 2-category of $U$-small categories, itself, but it appears as part of the symbol $Hom_{U-\mathfrak{Cat}^{2}}$, since this functor essentially constitutes the enrichment data associated to $U-\mathfrak{Cat}^{2}$. The ``enrichment" of $U-\mathfrak{Cat}$ over itself is given by a functor 
$$
Hom_{U-\mathfrak{Cat}^{2}} := (Hom_{U-\mathfrak{Cat}^{2}(0)} , Hom_{U-\mathfrak{Cat}^{2}(1)} ) 
: \ U-\mathfrak{Cat}^{opp} \times U-\mathfrak{Cat} \longrightarrow U-\mathfrak{Cat},
$$

defined by the two functions, $Hom_{U-\mathfrak{Cat}^{2}(i)}$, for $i = 0,1$, defined below. \footnote{Note also that the term ``enrichment" has not yet been defined, and is not yet necessary. This will be done in section (2.2).
}

The functor $Hom_{U-\mathfrak{Cat}^{2}}$ on objects is the function
$Hom_{U-\mathfrak{Cat}^{2}(0)}$ that sends a pair of categories $(C,D)$ to the $U$-category $Hom_{U-\mathfrak{Cat}^{2}(0)}(C,D)$ of 
functors $C \rightarrow D$. (Its set of objects is the set of functors $F : C \longrightarrow D$ and the set of arrows is the set of natural transformations $F \xrightarrow{\alpha} G$; see [1] or [2].)

On arrows the functor $Hom_{U-\mathfrak{Cat}^{2}}$ is given by the function $Hom_{U-\mathfrak{Cat}^{2}(1)}$ 
which sends 
an arrow $(G,F) \in Arr(U-\mathfrak{Cat}^{opp} \times_{U'-\mathfrak{Cat}} U-\mathfrak{Cat})$, given by a pair of functors $(C' \xrightarrow{G} C) \in (U-\mathfrak{Cat})^{opp}$ and $(D \xrightarrow{F} D') \in U-\mathfrak{Cat}$, to a functor $Hom_{U-\mathfrak{Cat}^{2}(1)}(G,F) : Hom_{U-\mathfrak{Cat}^{2}(0)}(C,D) \longrightarrow Hom_{U-\mathfrak{Cat}^{2}(0)}(C',D')$. 

We define $Hom_{U-\mathfrak{Cat}^{2}(1)}(G,F)$ on objects by forwards and backwards composition, i.e. the object map is given by $( H \mapsto F \circ H \circ G)_{H \in Hom_{U-\mathfrak{Cat}}(C,D)}$. 

We define the arrow map $Hom_{U-\mathfrak{Cat}^{2}(1)}(G,F)$ as follows. For a natural transformation $\alpha : H_{1} \rightarrow H_{2}$ between functors $H_{1},H_{2} : C \longrightarrow D$, we define 
$$
Hom_{U-\mathfrak{Cat}^{2}(1)}(G,F)(\alpha) := (i \mapsto F(\alpha(G(i))) )_{i\in Ob(C')} \in Arr(Hom_{U-\mathfrak{Cat}^{2}(0)}(C',D')),
$$

so that $ Hom_{U-\mathfrak{Cat}^{2}(1)}(G,F)(\alpha) : F \circ H_{1} \circ G \rightarrow F \circ H_{2} \circ G$.

\sss{}
Suppose that $I$ and $C$ are categories. Then the diagonal functor

$
\Delta_{(I,C)} : C \rightarrow Hom_{U-\mathfrak{Cat}^{2}(0)}(I,C)
$

sends an object $c \in Ob(C)$ to its constant functor, which sends every object in $I$ to $c$ and every arrow to the identity arrow of $c$. It sends an arrow $\phi : c_{1} \rightarrow c_{2}$ to the natural transformation $\Delta_{(I,A)(0)}(c_{1}) \rightarrow \Delta_{(I,A)(0)}(c_{2})$ which assigns to every $i \in Ob(I)$ the arrow $\phi : c_{1} \rightarrow c_{2}$. I.e., it is defined on objects by

$
\forall c \in Ob(C), \forall f \in Arr(I), \ \Delta_{(I,C)(0)}(c)_{(1)} (f) = id_{c}
$

and on arrows by the following

$
\forall \phi \in Arr(A), \ \Delta_{(I,C)(1)}(\phi) := (i \mapsto \phi)_{i \in Ob(I)}.
$

\sss{}
For any category $C \in Ob(U-\mathfrak{Cat})$, the Yoneda functors
$$
Yo_{(C)} : C \hookrightarrow Hom^{(1)}_{\mathfrak{Cat}^{2}}(C^{opp},U-\mathfrak{Set})
$$
$$
Yo^{opp}_{(C)} : C^{opp} \hookrightarrow Hom^{(1)}_{\mathfrak{Cat}^{2}}(C,U-\mathfrak{Set})
$$

send an object $c \in Ob(C)$ to the functors given by $Yo_{(C)(0)}(c) : x \mapsto Hom_{C}(x,c)$ and $Yo^{opp}_{(C)(0)}(c) : x \mapsto Hom_{C}(c,x)$.
\ftt{I imagine that it may be desirable to restrict the codomain of the Yoneda functors to a $U$-small sub-category, that their use might not force the unnecessary invocation of higher universes in certain circumstances.}

\sus{A Variation on Limits ($(sk,e)$-limits)}

\sss{The Use of $dob\downarrow_{(-)}$} Recall that $\Delta_{(J,A)} : A \longrightarrow Hom_{U-\mathfrak{Cat}^{2}(0)}(J,A)$ by sending an object $c$ to the $c$-valued constant functor, and $ob_{(Hom_{U-\mathfrak{Cat}^{2}(0)}(J,A))}(F \cdot e) : \star \longrightarrow Hom_{U-\mathfrak{Cat}^{2}(0)}(J,A)$ by sending the one arrow in $\star$ to $id_{F \cdot e}$. Recall also that the category $\downarrow_{(Hom_{U-\mathfrak{Cat}^{2}(0)}(J,A))} (\Delta_{(J,A)},ob_{Hom_{U-\mathfrak{Cat}^{2}(0)}(J,A)}(F\cdot e))$ of arrows is defined so that its objects are triples $(a,\alpha,\emptyset )$, where $a \in Ob(A)$, $\alpha : \Delta_{(J,A)}(a) \rightarrow F \cdot e$ is a natural transformation, and $\emptyset$ is the object in the category $\star$ (the category with one arrow). An arrow between $(a_{1},\alpha_{1},\emptyset ) \xrightarrow{(\phi,id_{\emptyset})} (a_{2},\alpha_{2},\emptyset )$ is a pair of arrows $(( a_{1} \xrightarrow{\phi} a_{2}) , id_{\emptyset} ) \in Arr(A) \times Arr(\star)$ for which $\alpha_{2} \cdot \Delta_{(J,A)(1)}(\phi) = \alpha_{1}$. 

An isomorphic category is given by forgetting both the $\emptyset$ symbol, and the $a$ term (since for any $j \in Ob(J)$, $a = dom(\alpha(j))$, so that $a$ is determined by $\alpha$), so that its objects are natural transformations $\alpha$, where $a \in Ob(A)$ and $\alpha : \Delta_{(J,A)}(a) \rightarrow F \cdot e$. If $\alpha_{1} : \Delta_{(J,A)}(a_{1}) \rightarrow F \cdot e$ and $\alpha_{2} : \Delta_{(J,A)}(a_{2}) \rightarrow F \cdot e$, then an arrow $\alpha_{1} \xrightarrow{\phi} \alpha_{2}$ is an arrow $( a_{1} \xrightarrow{\phi} a_{2}) \in Arr(A)$ for which $\alpha_{2} \cdot \Delta_{(J,A)(1)}(\phi) = \alpha_{1}$

\sss{$\bold{Definition}$ of an $(sk,e)$-Limit} \label{SkLim}

Consider functors
$J \xrightarrow{e} I \xrightarrow{F} A \xrightarrow{sk} B$. 

Consider the set of maps $\alpha : Ob(I) \rightarrow Arr(A)$ such that $sk \cdot \alpha$ defines a natural transformation from a diagonal functor to $sk\cdot F$. Let
$$
\mathcal{C} :_{t}= \downarrow_{(Hom_{U-\mathfrak{Cat}^{2}(0)}^{(1)}(I,B))} ( \Delta_{(I,B)} , ob_{(Hom_{U-\mathfrak{Cat}^{2}}^{(1)}( I,B ))}(sk\cdot F) )
$$

and
$$
\mathcal{D} :_{t}= \downarrow_{(Hom_{U-\mathfrak{Cat}^{2}(0)}^{(1)}(J,A))} ( \Delta_{(J,A)} , ob_{(Hom_{U-\mathfrak{Cat}^{2}}^{(1)}( J,A ))}(F\circ e) )
$$

and
$$
\mathcal{E} :_{t}= \downarrow_{(Hom_{U-\mathfrak{Cat}^{2}(0)}^{(1)}(J,B))} ( \Delta_{(J,B)} , ob_{(Hom_{U-\mathfrak{Cat}^{2}}^{(1)}( J,B ))}(sk\circ F\circ e) )
$$

Let $\varepsilon : P :_{t}= \mathcal{C} \times_{\mathcal{E}} \mathcal{D} \longrightarrow \mathcal{D}$ be one of the arrows of a fibred product, an arrow in $U'-\mathfrak{Cat}$. If $For$ is the functor which takes the object $a$ from an arrow $\Delta_{(J,A)}(a) \longrightarrow F \circ e$, an $(sk,e)$-limit is a colimit of $For \circ \varepsilon
$. 

This is explained in the following sections.

\ref{SkLim}.1. Let $P$ be the full sub-category of the category 
$$
\downarrow_{(Hom_{U-\mathfrak{Cat}^{2}(0)}(J,A))} (\Delta_{(J,A)},ob_{Hom_{U-\mathfrak{Cat}^{2}(0)}(J,A)}(F\circ e)) \subseteq Hom_{U-\mathfrak{Cat}^{2}(0)}(J,A)_{/ F \circ e}
$$

whose objects are natural transformations $\alpha$, such that for some $a \in Ob(A)$ we have $\alpha : \Delta_{(J,A)}(a) \rightarrow F \cdot e$, such that there exists a natural transformation $\tilde{\alpha} : \Delta_{(I,B)}(sk(a))\rightarrow sk\cdot F$ such that the natural transformation $Hom_{U-\mathfrak{Cat}^{2}(1)}(id_{J},sk)(\alpha):  \Delta_{(J,B)}(sk(a)) \rightarrow sk \cdot F \cdot e$ given by sending $j \in Ob(J)$ to $sk(\alpha (j)) : sk(a) \rightarrow sk(F(e(j)))$ is equal to the natural transformation $Hom_{U-\mathfrak{Cat}^{2}(1)}(e,id_{B})(\tilde{\alpha}) : \Delta_{(J,B)}(sk(a)) \rightarrow sk \cdot F \cdot e$ given by sending $j \in Ob(J)$ to $\tilde{\alpha}(e(j)) : sk(a) \rightarrow sk(F(e(j)))$.

\ref{SkLim}.2. Denote by $\varepsilon : P \longrightarrow \downarrow_{(Hom_{U-\mathfrak{Cat}^{2}(0)}(J,A))}(\Delta_{(J,A)},ob_{(Hom_{U-\mathfrak{Cat}^{2}(0)}(J,A))}(F\cdot e))$ the inclusion, given by $\alpha \mapsto (a,\alpha,\emptyset )$. Denote also by $p$ the functor, 
$$
p :_{t}= dob\downarrow_{(Hom_{U-\mathfrak{Cat}^{2}(0)}(J,A))}(\Delta_{(J,A)},ob_{(Hom_{U-\mathfrak{Cat}^{2}(0)}(J,A))}(F \cdot e) \cdot \varepsilon : P \longrightarrow A.
$$

Thus, $p$ is given by sending $\alpha \mapsto a$, and the fibre of $p$ over any given $a \in Ob(A)$ is the set of natural transformations $\Delta_{(J,A)}(a) \xrightarrow{\alpha} F\cdot e $ such that for some natural transformation $\tilde{\alpha} : sk\cdot \Delta_{(I,A)}(a) = \Delta_{(I,B)}(sk(a)) \longrightarrow sk \cdot F$, one has
$$
Hom_{U-\mathfrak{Cat}^{2}(1)}(id_{B},e)(\tilde{\alpha}) = Hom_{U-\mathfrak{Cat}^{2}(1)}(id_{J},sk)(\alpha)
.$$

In other words, 
$$
P \subseteq \downarrow_{(Hom_{U-\mathfrak{Cat}^{2}(0}(J,A))}(\Delta_{(J,A)},ob_{(Hom_{U-\mathfrak{Cat}^{2}(0)}(J,A))}(F \cdot e)
$$ 

is the full subcategory which contains all objects $\alpha$ such that the image of $\alpha$ in \newline $Hom_{U-\mathfrak{Cat}^{2}(0)}(J,B)$ under the functor $Hom_{U-\mathfrak{Cat}^{2}(1)}(id_{J},sk)$ has a lift to \newline $Hom_{U-\mathfrak{Cat}^{2}(0)}(I,B)$ by the functor $Hom_{U-\mathfrak{Cat}^{2}(1)}(e,id_{B})$ (we denote this lift by $\tilde{\alpha}$).


\ref{SkLim}.3.
Then the $(sk,e)$-limit of $F$ is the colimit of $p$, i.e., for any pair $(l,\lambda) \in Ob(A) \times Arr(Hom_{U-\mathfrak{Cat}^{2}(0)}(P,A))$ for which $\lambda : p \rightarrow \Delta_{(P,A)}(l)$, we say that $(l,\lambda)$ is an $(sk)-limit(F)$ iff $(l,\lambda)$ is a $colimit(p)$ (in the sense in which $(\lambda,l)$ is a universal arrow in $Hom_{U-\mathfrak{Cat}^{2}(0)}(P,A)$, from $p$ to the constant functor $\Delta_{(P,A)} : A \longrightarrow Hom_{U-\mathfrak{Cat}^{2}(0)}(P,A)$).







 





\begin{ex}
In the above, if either $e$ or $sk$ is an identity functor, then the $(sk)$-limit of $F$ is the limit $(l,\lambda)$ of $F$, if the latter exists. 

If $e = id_{I}$, then $P$ consists of all natural transformations $\alpha : \Delta_{(I,A)}(a) \rightarrow F \circ e = F$ for which there exists some lift $\tilde{\alpha} : \Delta_{(I,B)}(sk(a)) \rightarrow sk \circ F$. But $\tilde{\alpha} = Hom_{U-\mathfrak{Cat}^{2}(1)}(sk,id_{I})(\alpha)$ would be such a lift. Therefore any $\alpha : \Delta_{(I,A)}(a) \rightarrow F$ has a lift. Furthermore, for any $i \in Ob(I)$, $\alpha(i) : a \rightarrow F(i)$, and for any $\beta : \Delta_{(I,A)}(b) \rightarrow F$, and any arrow $\phi : \alpha \rightarrow \beta$ in $P$, by definition of $P$, we have that $\beta(i) \circ \phi = \alpha(i)$. Therefore, by the definition of a colimit, there exists a unique $\alpha_{l}(i) : l \rightarrow F(i)$ such that for any $(\Delta_{(I,A)}(a) \xrightarrow{\alpha} F) \in Ob(P)$, $\alpha(i) = \alpha_{l}(i) \circ \lambda(\alpha)$. Since each colimit arrow $\alpha_{l}(i)$ is determined by the arrows $\alpha (i)$ which come from natural transformations $\alpha$, the assignment $\alpha_{l} = (i \mapsto \alpha_{l}(i))_{i \in Ob(I)}$ determines a natural transformation $\Delta_{(I,A)}(l) \rightarrow F$. Therefore $\alpha_{l} \in Ob(P)$. If the limit of $F$ exists, then it is isomorphic to a terminal object in $P$, $\alpha_{t} \in Ob(P)$. But by the above argument, this terminal object $\alpha_{t}$ determines a colimit arrow $\lambda (\alpha_{t}) : a_{t} \rightarrow l$, and being a terminal object in $P$ there is a unique arrow $e_{l} : l \rightarrow a_{t} = dom(\alpha_{t}(i))$ in $P$. By the definition of terminal objects, $e_{l} \circ \lambda(\alpha_{t}) = id_{a_{t}}$ 
\end{ex}

\begin{lem}
(Inclusion, via right exactness) Given $sk, F, e : J \rightarrow I$, $\varepsilon : P \rightarrow$

$\downarrow_{(Hom_{(U-\mathfrak{Cat}^{2})(0)} (J, A) )}(\Delta_{(J,A)}, ob_{(Hom_{(U-\mathfrak{Cat}^{2})(0)}(J,A) )} (F \circ e) )$, $l$, and $\lambda$ as above, suppose further that $sk$ is right exact, and $Ob(I) = Ob(J)$. For each $i \in Ob(I) = Ob(J)$, consider the arrow induced from the colimit $l$ to $F(i)$ by $\alpha \mapsto \alpha (i)$, where $(p,\alpha,\cdot) \in Ob(P)$ is an object in $P$. Then this assignment determines an object $(l,\alpha_{l} , \cdot ) \in Ob(P)$.

I.e. the $(sk)$-limit determines an object,
$$
(l, (j \mapsto \lambda ( ( (b,f,\cdot) \mapsto f(j) )_{(b,f,\cdot) \in Ob(P)} ) )_{j \in Ob(J)}, \cdot ) \in Ob(P)
$$

in $P$.
\end{lem}
\begin{proof}
If the colimit $(l,\lambda)$ is sent to the colimit of the forward composition by $sk$ of the $dob\downarrow$ diagram on $P$, then arrows from $l$ to it are yet determined by their pullbacks to the components of the forward composition of the colimit diagram, which commute after forward composition.
\end{proof}

\begin{lem}
(Uniqueness, via monic) For any $sk : A \rightarrow B, F : I \rightarrow A \in Arr(U-\mathfrak{Cat})$, for any $(l, \lambda) \in Ob(A) \times Arr( Hom_{U-\mathfrak{Cat}^{2}(0)} (P,A) )$, ($P$ being as above) if the arrow from the $(sk)$-limit$(F)$ to the product $\prod_{j\in J} F(j)$ induced by the arrows from the previous lemma, i.e. $\lambda_{\prod}( (j \mapsto \lambda ( ((b,f,\cdot) \mapsto f(j) )_{(b,f,\cdot) \in Ob(P)}) ) )_{j \in Ob(J)} ) ) : l \rightarrow \prod_{j \in Ob(J)} F_{(0)}\circ e_{(0)} (j)$, is monic, then for each $(b,f,\cdot) \in Ob(P)$, the coproduct arrow $b \rightarrow l$ is the unique arrow $\lambda'$ for which $\lambda ( ( (b,f,\cdot) \mapsto f(j) )_{(b,f,\cdot) \in Ob(P)} ) \circ \lambda' = f(j)$.
\end{lem}
\begin{proof}
Trivial.
\end{proof}

\begin{rem}
This is the uniqueness of factorization usually associated to limits. 
\end{rem}

\sss{$\bold{Definition}$ of the Skeleton Functor}\label{SCat}
Define the category $U-\mathfrak{SCat} \in Ob(U'-\mathfrak{Cat})$ so that
$$
Ob(U-\mathfrak{SCat}) = Ob(U-\mathfrak{Cat})
$$ 

and for any $x,y \in Ob(U-\mathfrak{SCat})$, $Hom_{(U-\mathfrak{SCat})} (x,y)$ is the the set
$$
Hom_{(U-\mathfrak{SCat})} (x,y) := \{ [F] \subseteq Ob(Hom_{U-\mathfrak{Cat}^{2}(0)}(x,y)); F \in Hom_{U-\mathfrak{Cat}}(x,y) \}
$$ 

of isomorphism classes of functors $x \xrightarrow{F} y$, where $[F] = [G]$ iff $F \cong G$, i.e. iff there exists an isomorphism $F \xrightarrow{\alpha} G$ of functors.

Define the functor 
$$
Skel : U-\mathfrak{Cat} \rightarrow U-\mathfrak{SCat}
$$

so that $Skel$ is the identity map on the objects and the quotient map $F \mapsto [F]$ on the arrows.

\begin{ex}
Consider $\phi,\psi \in Arr(U-\mathfrak{Cat})$, with the same codomain. The ($Skel$)-limit of the diagram is the subcategory $L$ of $dom(\phi) \times_{U-\mathfrak{Cat}} dom(\psi)$ such that $Arr(L) = \{ f \in Arr(dom(\phi)\times_{U-\mathfrak{Cat}}dom(\psi) ; \exists u,v \in Arr(codom(\phi)), u,v \text{ are isomorphisms and } u \circ \pi_{\phi}(f) = \pi_{\psi}(f) \circ v \}$. Any category with such functors into the two domain categories that the composition of functors on one side is isomorphic to the composition of functors on the other side factors through $L$ via the compositions of the projections with the embedding into the product. By the monic lemma the factorization is unique. However the conclusion of the inclusion lemma might not apply to it, i.e. the two compositions $L \rightarrow dom(\phi) \rightarrow codom(\phi)$ and $L \rightarrow dom(\psi) \rightarrow codom(\psi) = codom(\phi)$ might not be isomorphic, since I might imagine having two different pairs of arrows $(f_{1},g_{1})$, and $(f_{2},g_{2})$, such that the isomorphisms $u_{1},v_{1} \in Arr(codom(\phi))$ which form the commuting square $u_{1} \circ f_{1} = g_{1} \circ v_{1}$ differ from the isomorphisms $u_{2}, v_{2} \in Arr(codom(\phi))$ which form the commuting square $u_{2}\circ f_{2} = g_{2} \circ v_{2}$. 
\end{ex}

\sss{Lemma, for Reduction to the Standard Limit}
If 
$$
(l, (j \mapsto \lambda ( ( (b,f,\cdot) \mapsto f(j) )_{(b,f,\cdot) \in Ob(P)} ) )_{j \in Ob(J)}, \cdot ) \in Ob(P)
$$

and the limit arrows are unique then this is the usual limit.

\begin{proof}
Trivial.
\end{proof}

\sss{Functoriality}\label{LimFun} An arrow of functors $F \circ e \rightarrow G \circ e$ which lifts to an arrow of functors $sk \circ F \rightarrow sk \circ G$ (i.e. an arrow in the fibred product of the two functors $Hom_{U-\mathfrak{Cat}^{2}(1)}(e,id_{B})$ and $Hom_{U-\mathfrak{Cat}^{2}(1)}(id_{J},sk)$ ) induces a map from the $(sk,e)$-limit of $F$ to that of $G$, using the colimit map. I.e. $\alpha : F \rightarrow G$ implies that $\alpha (dom(\phi)) \circ F(\phi) = G(\phi) \circ \alpha(codom(\phi))$, so that for any arrow $\beta : \Delta_{(J,C)(0)}(c) \rightarrow F\circ e$ associated to $(a,\beta,\emptyset) \in Ob(P)$ (notation as in the first definition), $Hom_{U-\mathfrak{Cat}^{2}(1)}(e,id_{A})(\alpha) \circ \beta$ also commutes after applying $sk$ (i.e. comes from an arrow in $Hom_{U-\mathfrak{Cat}^{2}(0)}(I,B)$). Therefore each such $a$ has an arrow into the $sk$-limit of $G$ from the colimit diagram of the definition, which induces a map from the colimit diagram which determines the $sk$-limit of $F$.

\ref{LimFun}.1. Given a diagram $F : I' \longrightarrow Hom_{U-\mathfrak{Cat}^{2}(0)}(I,A)$, and a choice of an $(sk,e)$-limit $(l(i),\lambda(i))$ for any object $i \in Ob(I')$, the construction of (\ref{LimFun}) determines a function $Arr(I') \longrightarrow Arr(A)$

\ref{LimFun}.2. If for any $i \in Ob(I')$, the $(sk,e)$-limit $(l(i),\lambda(i))$ is included in $P(i)$ ($P(i)$ being as in the definition of the $(sk,e)$-limit for $F(i)$) then (\ref{LimFun}.1) determines a functor $I' \longrightarrow A$.

\begin{rem}
Roughly speaking, one takes the colimit of the domains of all limit diagrams on the trivial category which, when forwards composed with $sk$, are the backwards composition by $e$ of an actual limit diagram of $sk \circ F$. Definition (2.3) following this remark is dual to Definition (2.1).
\end{rem}

\sss{$\bold{Definition}$ of the $(sk)$-Colimit}\label{SkColim} Consider functors $J \xrightarrow{e} I \xrightarrow{F} A \xrightarrow{sk} B$.

\ref{SkColim}.1. Let $P$ be the full sub-category of the category 
$$
\downarrow_{(Hom_{U-\mathfrak{Cat}^{2}}(J,A))} (ob_{Hom_{U-\mathfrak{Cat}^{2}}(J,A)}(F\circ e),\Delta_{(J,A)}) \subseteq Hom_{U-\mathfrak{Cat}^{2}(0)}(J,A)_{ \backslash F \circ e}
$$ 

of arrows, whose objects are given by natural transformations from functor $F \circ e$
to functor $\Delta_{(J,A)}(a)$, i.e. triples $(\emptyset,\alpha,a)$ for varying $a \in Ob(A)$, such that there exists a natural transformation $\tilde{\alpha}$ from functor $sk\circ F$ to functor $\Delta_{(I,B)}(a)$ such that the natural transformation from $sk \circ F$ to $\Delta_{(J,B)}(sk(a))$ is equal to the natural transformation given by sending $j \in Ob(J)$ to $\tilde{\alpha}(e(j))$, i.e. by the set
$$
\{ \alpha : \Delta_{(J,A)}(p) \xrightarrow{\alpha} F\circ e;
$$
$$
\exists \tilde{\alpha} : sk\circ \Delta_{(I,A)}(p) = \Delta_{(I,B)}(sk(p)) \longrightarrow sk \circ F, Hom_{U-\mathfrak{Cat}^{2}(1)}^{(1)}(id_{B},e)(\tilde{\alpha}) = \alpha \}
,$$ 


so as to be given by the category of arrows from the diagonal functor to the object functor of $F \circ e$ in the category of functors from $J$ to $A$.

\ref{SkColim}.2. Suppose that $\varepsilon : P \longrightarrow \downarrow_{(Hom_{U-\mathfrak{Cat}^{2}}(J,A))}(\Delta_{(J,A)},ob_{(Hom_{U-\mathfrak{Cat}^{2}}(J,A))}(F\circ e))$ is the inclusion.

\ref{SkColim}.3.
For any $sk : A \rightarrow B, F : I \rightarrow A \in Arr(U-\mathfrak{Cat})$, $codom(F) = dom(sk)$ implies that any pair $(l,\lambda ) \in Ob(A) \times Arr(Hom_{(U-\mathfrak{Cat}^{2})(0)} (A,U-\mathfrak{Set}) )$, $(l,\lambda)$ is a $(sk,e)-colimit(F)$ iff $(l,\lambda)$ is a limit $(cob\downarrow_{(Hom_{(U-\mathfrak{Cat}^{2})(0)} (J,A) )}$

$( ob_{(Hom_{(U-\mathfrak{Cat}^{2})(0)} (J,A) )} (F \circ e) , \Delta_{(J,A)} ) \circ \varepsilon_{c} )$.

\begin{lem}
(Inclusion via exactness) Dual to the above.
\end{lem}

\begin{lem}
(Uniqueness via epic) Dual to the above.
\end{lem}

\begin{ex}
Consider $\phi$,$\psi \in Arr(U-\mathfrak{Cat})$, with the same domain. The ($Skel$)-colimit of the diagram is the category $L$ such that its set of objects is the disjoint union of the objects of the codomain categories and the arrows are the formal compositions of the disjoint union of arrows in $Arr(codom(\phi))$, $Arr(codom(\psi))$, and arrows $e_{a} : \phi_{(0)}(a) \rightarrow \psi_{(0)}(a)$, $e_{a}^{-1} : \psi_{(0)}(a) \rightarrow \phi_{(0)}(a)$ formally added for each $a \in Ob(dom(\phi)) = Ob(dom(\psi))$, with the relation generated by requiring that $\forall f \in Arr(dom(\phi)), \phi_{(1)}(f) \circ e_{dom(f)} = e_{codom(f)} \circ \psi_{(1)}(f)$. If $l_{\phi} : codom(\phi) \rightarrow L$ and $l_{\psi} : codom(\psi) \rightarrow L$ are given by the $U-\mathfrak{Set}$ coproduct maps then for any $l'_{\phi}, l'_{\psi} \in Arr(U-\mathfrak{Cat})$ such that $l'_{\phi}\circ \phi \cong \l'_{\psi} \circ \psi$, there is an arrow $q : L \rightarrow codom(l'_{\phi}) = codom(l'_{\psi})$ such that $l'_{\phi} = q\circ l_{\phi}$ and $l'_{\psi} = q \circ l_{\psi}$. If an isomorphism $\alpha : l'_{\phi} \circ \phi \rightarrow l'_{\psi} \circ \psi$ is specified (or vice versa), then there is a unique $q : L \rightarrow codom(l'_{\phi})$ such that  $Hom_{(U-\mathfrak{Cat}^{2})(1)}((id_{dom(\phi)},q))_{(1)}( (a \mapsto e_{a})_{a \in Ob(dom(\phi))} ) = \alpha$ (and vice versa).
\end{ex}

\begin{lem}
(Reduction) Dual to the above.
\end{lem}

\begin{lem}
(Functoriality) An arrow of functors $F \rightarrow G$ induces a map from the ($sk$)-colimit of $F$ to that of $G$, using the limit map.
\end{lem}

\begin{rem}
Products and coproducts are not affected by $sk$.
\end{rem}

\sus{Definitions regarding Enrichments}
We will define weak enrichment of sets and categories. 
Sets will be enriched over tensor categories $(A,\otimes)$
and categories over triples $(A,\otimes,F)$
where tensor category $(A,\otimes)$ comes with a tensor 
functor $F: (A,\ten) \longrightarrow (\mathfrak{Set},\times)$.

A weak  enrichment of a set $s$ over $(A,\ten)$
adds to $s$ a category-like structure, a version of $\Hom$ which has values in $A$
(rather than in sets) but without any associativity or unital requirements. We later introduce, for each functor $sk : A \longrightarrow B$, a category of weakly enriched sets, ``associative up to $sk$," in that the associativity diagrams are commutative after the functor $sk$ is applied to them. A weak enrichment of a category $C$ over $(A,\ten)$ with respect to a tensor functor $(F,\rho) : (A,\otimes) \rightarrow (\mathfrak{Set},\times_{\mathfrak{Set}})$ is
a weak enrichment of the set $Ob(C)$ over $(A,\otimes)$ which is compatible with the
$Hom_C$, this compatibility being formulated in terms of the tensor functor $(F,\rho)$.

\sss{$\bold{Definition}$ of a Weakly Enriched Set}
A 
weak enrichment of a set $s\in Ob(U-\mathfrak{Set})$ over a tensor category $(A, \otimes) \in Ob(U-\mathfrak{TCat})$ (see \ref{TCat}) 
is a 
pair
of a map $h : s^{2} \rightarrow Ob(A)$
and a ``composition map'' 
$ \circ : s^{3} \rightarrow Arr(A) )$ such that for any $a,b,c \in s$, 
$$
\circ (a,b,c) :h(a,b) \otimes h(b,c) \longrightarrow h(a,c).
$$

\let\tensor\ten
\newcommand{\USet}{{ U-\mathfrak{Set} }}
\newcommand{\Cat}{{ \mathfrak{Cat} }}
\newcommand{\UCat}{{ U-\mathfrak{Cat} }}
\newcommand{\UTCat}{{ U-\mathfrak{TCat} }}

\sss{$\bold{Definition}$ of the Category of Weak Enrichments}\label{CatWE}
 For any $(A, \otimes ) \in Ob(U-\mathfrak{TCat})$, 
the category of $(A,\otimes)$-enriched sets $WE( A , \otimes ) \in Ob(U-\mathfrak{Cat})$ has as objects weak enrichements of sets  $S=(s,h_S,\circ_S)$,
and for two weak enrichments $S$ and $T$ an arrow 
$f: S = (s, h_{S}, \circ_{S}) \rightarrow T = (t, h_{T}, \circ_{T} )$ is a pair of functions 
$f=( f_{1} : s \rightarrow t , f_{2} : s^{2} \rightarrow Arr(A) )$
such that the following hold.

\ref{CatWE}.1. $\forall a,b \in s,\ \  f_{2}(a,b) : h_{S}(a,b) \longrightarrow h_{T}(f_{1}(a),f_{1}(b))$, and 

\ref{CatWE}.2. $\forall a,b,c \in s$,  
$$
\circ_{T} ( f_{1}(a), f_{1}(b), f_{1}(c) ) \circ ( f_{2}(a,b) \otimes f_{2}(b,c) ) = f_{2}(a,c) \circ \circ_{S} ( a,b,c ),
$$ 

i.e. the compositions commute with the arrows defining a ``functor from $S$ to $T$". 

\begin{lem}
\label{TCatFun}

The above construction, of $WE(A,\otimes)$, extends to a functor $WE : U-\mathfrak{TCat} \longrightarrow U'-\mathfrak{Cat}$, from the category of tensor categories to the category of categories.

For any functor of tensor categories $(F,\rho) : (A,\otimes_{A}) \rightarrow (B,\otimes_{B})$ define a functor $WE(F,\rho) : WE(A,\otimes_{A}) \rightarrow WE(B,\otimes_{B})$ from the category of weak enrichments over $(A,\otimes_{A})$ to that of $(B,\otimes_{B})$ as follows.

\ref{TCatFun}.1. 
It sends an object $S=(s,h,\ci)$
of 
$WE(A,\otimes_{A})$ 
to the triple 
$F(S)=(s,h',\circ')$ 
where for 
$a,b,c\in s$,
$$
h'(a,b)= F(h(a,b)) \text{ and } \circ'(a,b,c) = F(\circ(a,b,c)) \circ \rho(h(b,c),h(a,b)).
$$
 

\ref{TCatFun}.2.
It sends an arrow 
$\phi: S = (s,h_{s},\circ_{s}) \rightarrow (t,h_{t},\circ_{t}) = T$ in $WE(A,\otimes_{A})$ 
(here $s^2\ni (a,b)\mm \phi(a,b)\in Arr(A)$)
to the arrow $F(\phi):F(S)\to F(T)$ that sends 
$
(a,b)\in s^2$ to  $F(\phi(a,b))) \in Arr(B)$.



\end{lem}

\sss{$\bold{Definition}$ of an Weakly Enriched Category}\label{WECat}
A weak enrichment of a category $C$ with respect to a tensor functor $(A, \otimes) \xrightarrow{(F,\rho)} (U-\mathfrak{Set},\times_{U-\mathfrak{Set}})$ is a quadruple $(C,h,\circ,\phi)$ such that $C \in Ob(U-\mathfrak{Cat})$ is a category, $h$ and $\circ$ define a weak enrichment of the set $Ob(C)$, and $\phi : Ob(C)^{2} \rightarrow Arr(U-\mathfrak{Set})$ is a function, such that

\ref{WECat}.1. For any $a,b \in Ob(C)$, $\phi(a,b) : F(h(a,b)) \rightarrow Hom_{C}(a,b)$ is an isomorphism;

\ref{WECat}.2. For any $a,b,c \in Ob(C)$, the composition 
$$
\circ_{C}(a,b,c) : Hom_{C}(b,c) \times_{U-\mathfrak{Set}} Hom_{C}(a,b) \rightarrow Hom_{C}(a,c)
$$
 
of hom sets in $C$ is given by the weak enrichment, i.e.
$$
\circ_{C}(a,b,c) = \phi(a,c)^{-1} \circ F(\circ(a,b,c)) \circ \rho(h(b,c),h(a,b)) \circ (\phi(b,c) \times_{U-\mathfrak{Set}} \phi(a,b))
$$

\sss{$\bold{Definition}$ of the Category of Weakly Enriched Categories}\label{CatWECat}

The category \newline
$WE_{\mathfrak{Cat}}(F,\rho)$ of categories weakly enriched over a tensor category 
$(A,\otimes)$ with respect to a tensor functor
$(F,\rho): (A,\otimes) \to (\USet , \times_{U-\mathfrak{Set}})$, has objects which are
categories $(C,h,\circ,\phi)$ weakly enriched over $(A, \otimes)$.
An arrow 
$f : (C,h_{C},\circ_{C},\phi) \rightarrow (D,h_{D},\circ_{D},\psi)$
consists  
of a functor
$(f_{0},f_{1}) : C \longrightarrow D$ 
and a function $f_{2} : Ob(C)^{2} \longrightarrow Arr(A)$, 
such that 

\ref{CatWECat}.1. $(f_{0},f_{2}) : (Ob(C),h_{C},\circ_{C}) \rightarrow (Ob(D),h_{D},\circ_{D})$ is an arrow of weak enrichments of sets; 

\ref{CatWECat}.2. For any $a,b \in Ob(C)$,
$$
F_{1}(f_{2}(a,b)) = \psi(f_{0}(a),f_{0}(b)) \circ F(f_{2}(a,b)) \circ \phi(a,b)^{-1}
$$ 

i.e. the functor agrees with that implied by the enrichment.

\sss{}\label{TCatFunC}
One can construct a functor from the category of tensor categories over the tensor category of sets $U-\mathfrak{TCat}_{/(U-\mathfrak{Set},\times_{U-\mathfrak{Set}})}$ to the category of categories, i.e. 
$$
WE_{\mathfrak{Cat}}( ) := (WE_{\mathfrak{Cat}0}( ), WE_{\mathfrak{Cat}1}( )) :U-\mathfrak{TCat}_{/(U-\mathfrak{Set},\times_{U-\mathfrak{Set}})} \longrightarrow U'-\mathfrak{Cat} 
$$

in analogue to the construction of Lemma \ref{TCatFun}, as follows. For any arrow $(\Phi,\rho) : (F,\rho_{F}) \longrightarrow (G,\rho_{G})$ of tensor categories $(F,\rho_{F}) : (A,\otimes_{A}) \longrightarrow (U-\mathfrak{Set},\times_{U-\mathfrak{Set}})$ and $(G,\rho_{G}) : (B,\otimes_{B}) \longrightarrow (U-\mathfrak{Set},\times_{U-\mathfrak{Set}})$ over $(\mathfrak{Set},\times_{U-\mathfrak{Set}} )$ define a functor \newline $WE_{\mathfrak{Cat}0}(F,\rho_{F}) \longrightarrow WE_{\mathfrak{Cat}0}(G,\rho_{G})$.

\ref{TCatFunC}.1. It is defined on an object $(C,h,\circ,\phi) \in Ob(WE_{\mathfrak{Cat}0}(F,\rho))$ by
$$
(C, h, \circ ,\phi ) \longmapsto (C, \Phi_{(0)} \circ h, ((a,b,c) \mapsto
\Phi_{(1)}(\circ (a,b,c)) \circ \rho (h(b,c),h(a,b)) )_{a,b,c \in Ob(\mathcal{C})},\phi ).
$$

\ref{TCatFunC}.2. It is deifined on arrows $(F,F_{2}) : (C,h_{C},\circ_{C},\phi) \rightarrow (D,h_{D},\circ_{D},\psi)$ by
$$
(F,F_{2}) \mapsto WE_{\mathfrak{Cat}1}(\Phi,\rho)(F,F_{2}) := (F,\Phi_{(1)} \circ F_{2}).
$$

\sss{$\bold{Definition}$ of Two Forgetful Functors}\label{DefForWE}

Define the following two functors.

\ref{DefForWE}.1. For any tensor functor $(F,\rho) : (A,\otimes) \longrightarrow (U-\mathfrak{Set},\times_{U-\mathfrak{Set}})$, the forgetful functor $For^{WE(F,\rho)}_{WE( dom(F, \rho) )} : WE_{\mathfrak{Cat}}(A, \otimes, F) \longrightarrow WE_{\mathfrak{Cat}}(A, \otimes )$ from the category of weakly enriched categories with respect to $(F,\rho)$ to weakly enriched sets with respect to $(A,\otimes)$ is the functor given by passing from a category $C$ to its set of objects $Ob(C)$. More precisely, it is defined on an object $(C,h,\circ,\phi) \in Ob(WE_{\mathfrak{Cat}}(F,\rho) )$ by 
$$
(C,h,\circ,\phi) \mapsto (Ob(C),h,\circ )
$$

and on an arrow $(f,f_{2}) \in Arr(WE_{\mathfrak{Cat}}(F,\rho))$ by
$$
(f,f_{2}) \mapsto (f_{(0)},f_{2})
$$

\ref{DefForWE}.2. The forgetful functor from the category of weakly enriched categories to the category of categories $For^{WE( F,\rho)}_{\mathfrak{Cat}} : WE_{\mathfrak{Cat}}(F,\rho) \longrightarrow U-\mathfrak{Cat}$ is the functor which forgets the enrichment structure, returning the underlying category. I.e. it sends a weakly enriched category $(C,h,\circ,\phi)$ to $C$.

\sss{$\bold{Definition}$ of the Category $WE_{(A,\otimes)(sk)}$} \label{DefWESk}
For any $sk: A \longrightarrow B \in Arr(\mathfrak{Cat})$, define the category $WE_{(A,\otimes)(sk)} \in Ob(\mathfrak{Cat})$ of $((A,\otimes),sk)$-enriched sets. 

\ref{DefWESk}.1. Its objects are sets enriched over $A$. 

\ref{DefWESk}.2. The hom sets
$$
Hom_{WE_{(A,\otimes)(sk)}}((S,h_{S},\circ),(T,h_{T},\circ_{T})) =
$$
are the pairs of maps of sets $(F_{0},F_{1}) \in Arr(\mathfrak{Set})^{2}$ such that $F_{0} : S \rightarrow T$ and $F_{1} : S^{2} \rightarrow Arr(A)$ and 

\ref{DefWESk}.2.1.
For any $a,b \in S, F_{1}(a,b) \in Hom_{A}(h_{S}(a,b),h_{T}(F_{0}(a),F_{0}(b))$.

\ref{DefWESk}.2.2. $F_{1}$ respects composition after applying $sk$.

\begin{rem}
Roughly speaking, $F_{0}$ is the map between objects of enriched sets, and $F_{1} : h_{S} \rightarrow h_{T}\circ F$ is the ``natural transformation of hom functors," (there are no non-trivial arrows in $S$). This means that applying the ``functor," $(F_{0},F_{1})$, then composing in $T$, versus composing in $S$ and  then applying the functor, gives two arrows in $A$, such that $sk$ of one arrow is equal to $sk$ of the other.
\end{rem}

\begin{lem}
$WE_{(A,\otimes)(sk)}$ is a category.
\end{lem}
\begin{proof}
The issue is composition. Given composible arrows
$$
((S,h_{S},\circ_{S}) \xrightarrow{(F_{0},F_{1})} (T,h_{T},\circ_{T})), ((T,h_{T},\circ_{T}) \xrightarrow{(G_{0},G_{1})} (U,h_{U},\circ_{U})) \in Arr(WE_{(A,\otimes)(sk)})
$$

Starting from the result of application of the functor $sk$ to the arrow which uses the composition $\circ_{U}$,
$$
sk(
$$
$$
\circ_{U}(G_{0}\circ F_{0}(a),G_{0}\circ F_{0}(b), G_{0}
\ \circ\ 
F_{0}(c))
$$
$$
\circ ((G_{1}(F_{0}(a),F_{0}(b))\circ F_{1}(a,b)) \otimes (G_{1}(F_{0}(b),F_{0}(c))\circ F_{1}(b,c)))
$$
$$
) =
$$
by functoriality of $\otimes$ 
$$
sk(
$$
$$
\circ_{U}(G_{0}\circ F_{0}(a),G_{0}\circ F_{0}(b)
, G_{0}\circ F_{0}(c)) \circ
$$
$$
(G_{1}(F_{0}(a),F_{0}(b)) 
\otimes G_{1}(F_{0}(b),F_{0}(c))) \circ
$$
$$
(F_{1}(a,b) \otimes F_{1}(b,c))
$$
$$
) =
$$
by functoriality of $sk$
$$
sk(\circ_{U}(G_{0}\circ F_{0}(a),G_{0}\circ F_{0}(b), G_{0}\circ F_{0}(c))) \circ 
$$
$$
sk((G_{1}(F_{0}(a),F_{0}(b)) \otimes G_{1}(F_{0}(b),F_{0}(c)))) \circ 
$$
$$
sk((F_{1}(a,b) \otimes F_{1}(b,c))) =
$$
by $(G_{0},G_{1}) \in Arr(WE_{(A,\otimes)(sk)})$,
$$
sk(G_{1}(F_{0}(a),F_{0}(c))) \circ sk(\circ_{T}(F_{0}(a),F_{0}(b),F_{0}(c))) \circ sk((F_{1}(a,b) \otimes F_{1}(b,c))) = 
$$
by $(F_{0},F_{1}) \in Arr(WE_{(A,\otimes)(sk)})$,
$$
sk(G_{1}(F_{0}(a),F_{0}(c))) \circ sk(F_{1}(a,c))\circ sk(\circ_{S}(a,b,c))
$$
\end{proof}


\begin{lem}
If $(A,\otimes)$ has products, then so does $WE_{(A,\otimes)(sk)}$. The product is functorial.
\end{lem}

\sss{$\bold{Definition}$ of $(sk)$-Associativity}
Consider an associative tensor category $(A, \otimes, \alpha)$ $\in Ob(U-\mathfrak{ATCat})$ (see \ref{ATCat}). A weakly $(A,\ten)$-enriched set $(S, h, \circ) \in Ob(WE (A, \otimes) )$, 
is said to be {\em ($sk,\alpha$)-associative} for a functor
$sk:A\to B$ (an arrow in $\UCat$)
if for any $a,b,c,d \in S$, 
$$
sk_{(1)} (\circ(a,b,d) \circ ( id_{h(a,b)} \otimes \circ(b,c,d) ) \circ \alpha(h(a,b), h(b,c), h(c,d)) ) =
$$
$$
sk_{(1)} (\circ(a,c,d) \circ (\circ (a,b,c) \otimes id_{h(c,d)} ) )
,$$ i.e. the standard self-consistency  diagram (pentagram) for the enriched composition $\circ$ is required to commute  after applying the functor $sk$
$$
\begin{CD}
h(c,d) \otimes ( h(b,c) \otimes h(a,b))
@>\alpha(h(c,d),h(b,c),h(a,b)) >>
(h(c,d) \otimes h(b,c)) \otimes h(a,b)
\\
@V{id_{h(c,d)} \otimes \circ(a,b,c)}VV
@V{\circ(b,c,d) \otimes id_{h(a,b)}}VV
\\
h(c,d) \otimes h(a,c)
@.
h(b,d) \otimes h(a,b)
\\
@V{\circ(a,c,d)}VV
@V{\circ(a,b,d)}VV
\\
h(a,d)
@>=>>
h(a,d).
\end{CD}
$$

If the associator $\alpha$ is understood, then we will write ``$(sk)$-associative".

\sss{$\bold{Definition}$ of $WE_{Ass(A,\otimes)(sk,\alpha)}$} Suppose that $(A,\otimes)$ has an associator $\alpha$ (see \ref{ATCat}). Define $WE_{Ass(A,\otimes)(sk,\alpha)}$ to be the full subcategory of $WE_{(A,\otimes)(sk)}$ generated by enriched sets $(S,h_{S},\circ_{S}) \in Ob(WE_{(A,\otimes)(sk)})$ which are $(sk,\alpha)$-associative. If the associator is understood, then we will denote this by ``$WE_{Ass(A,\otimes)(sk)}$".

\sus{Enrichment of $Hom_{WE_{(A,\otimes)(sk)}}(I,C)$}

Consider a tuple of functors $\{ p_{i} : I_{i} \longrightarrow A \}_{i=1}^{n}$. Suppose that for each $i \in \{1,...,n \}$, the colimit $colim \ p_{i} \in Ob(A)$ exists, with universal arrows $e_{i(x_{i})} : p_{i}(x_{i}) \rightarrow colim \ p_{i}$. Suppose that the colimit of the functor $\otimes_{i=1}^{n} p_{i} : \prod_{i=1}^{n} I_{i} \longrightarrow A$ defined by $(x_{i})_{i=1}^{n} \mapsto \otimes_{i=1}^{n}p_{i}(x_{i})$ is also an object in $A$. Consider the arrow $(colim \otimes_{i=1}^{n} p_{i} \rightarrow \otimes_{i=1}^{n} colim \ p_{i}) \in Arr(A)$ induced by $(x_{i})_{i=1}^{n} \mapsto \otimes_{i=1}^{n} e_{i(x_{i})}$; i.e. by tensoring the universal arrows together. The following lemma states that under certain conditions on the $p_{i}$, the above defines a natural transformation with respect to arrows of functors $\phi_{i} : p_{i} \rightarrow q_{i}$.

The $(A,\otimes)$-enrichment of the hom-sets in $WE_{(A,\otimes)(sk)}$ involves such colimits, and the definition of the composition requires that the above arrows should be isomorphisms. This means that the ``forward and backward composition functors" to be introduced in lemma \ref{PushPull} below are determined by the arrows between products $\prod h_{S}(...) \rightarrow \prod h_{T}(...)$.

\sss{Lemma on the Naturality of $\tau$}\label{LemNatTau}
Suppose that $\{ F_{i},G_{i} : I_{i} \longrightarrow A \}_{i=1}^{n}$ are functors, and $\{ (F_{i} \xrightarrow{\phi_{i}} G_{i}) \}_{i=1}^{n}$ are arrows of functors. Suppose that for each $i \in \{ 1,...,n\}$, $P_{F_{i}} \subseteq \downarrow_{(Hom^{(1)}_{U-\mathfrak{Cat}^{2}}(J_{i},A))}(\Delta_{(J_{i},A)},F_{i}\circ \varepsilon_{i} )$ and $P_{G_{i}} \subseteq \downarrow_{(Hom^{(1)}_{U-\mathfrak{Cat}^{2}}(J_{i},A))}(\Delta_{(J_{i},A)},G_{i}\circ \varepsilon_{i} )$ are subcategories, where $\varepsilon_{i} : J_{i} \subseteq I_{i}$ is the subcategory with only identity arrows.

\ref{LemNatTau}.1. Suppose that the functors $p_{F_{i}} : P_{F_{i}} \longrightarrow A$ and $p_{G_{i}} : P_{G_{i}} \longrightarrow A$ are as in the conditions of the limit inclusion lemma (i.e. $colim ( p_{F_{i}} )$ determines an object in $P_{F_{i}}$, with the analogue holding for $G_{i}$) 

\ref{LemNatTau}.2. Define an arrow of sets $\tau_{( \ )} : \prod_{i=1}^{n} Ob(Hom^{(1)}_{U-\mathfrak{Cat}^{2}}(I_{i},A)) \rightarrow Arr(A)$ so that for any $(H_{i})_{i=1}^{n} \in \prod_{i=1}^{n} Ob(Hom^{(1)}_{U-\mathfrak{Cat}^{2}}(I_{i},A))$, $\tau_{((H_{i})_{i=1}^{n}))} : colim (\bigotimes_{i=1}^{n} p_{H_{i}}) \rightarrow \bigotimes_{i=1}^{n} colim(p_{H_{i}})$ is the universal arrow for the colimit induced by the assignment (where $\lambda_{(i)}$ is the natural transformation defining the colimit of $p_{H_{i}}$)
$$
((a_{(i)},f_{(i)})_{i=1}^{n} \mapsto \otimes_{i=1}^{n} \lambda_{(i)}((a_{(i)},f_{(i)})) \ )_{(a_{(i)},f_{(i)})_{i=1}^{n} \in \prod_{i=1}^{n} P_{H_{i}}}
$$ 

\ref{LemNatTau}.3. If $u=u_{(p)} : colim (\bigotimes_{i=1}^{n} p_{F_{i}}) \rightarrow colim (\bigotimes_{i=1}^{n} p_{G_{i}})$ is the universal arrow for the colimit induced by the assignment $$
((a_{( \ )} , f_{( \ )}) \mapsto \lambda'_{G}((\otimes_{i=1}^{n}a_{(i)},\otimes_{i=1}^{n}\phi_{i}\circ f_{(i)} )) \ )_{(a_{( \ )} , f_{( \ )}) \in Ob(\prod_{i=1}^{n}P_{F_{i}})}
$$

then
$$
\otimes_{i=1}^{n} \lambda_{G_{j}}((a,\phi_{j}\circ f )) \circ \tau_{F_{( \ )}} = \tau_{G_{( \ )}} \circ u
$$

I.e., $\tau : colim (\bigotimes_{i=1}^{n} p_{i}) \rightarrow \bigotimes_{i=1}^{n} colim(p_{i})$ is ``natural at $(\phi_{i})_{i=1}^{n}$."

\begin{proof}
By the monic arrow condition, the arrows involved are situated above the products, $\prod_{a \in Ob(I_{i})} F_{i}(a)$, so that the arrows $\otimes_{i=1}^{n}a_{(i)} \rightarrow \otimes_{i=1}^{n} l_{G_{i}}$ are pure tensors respecting the arrows $\phi_{i}$.
\end{proof}

The following lemma defines a weak enrichment of the set $Hom_{WE_{(A,\otimes)(sk)}}(C,D)$. To any ``$(A,\otimes)$-functors" $\Phi,\Psi : C \rightarrow D$, one attaches a category $P$, and defines the hom object between $\Phi$ and $\Psi$ to be a colimit of a certain functor $P \longrightarrow A$. Roughly speaking, $P$ keeps track of all arrows into to the product $\prod_{x \in Ob(C)}h_{D}(\Phi(x),\Psi(x))$  which respect the composition with any arrows ``coming from some $h_{C}(x,y)$," after one applies $sk$. $P$ is a full sub-category of the category of arrows over $\prod_{x \in Ob(C)}h_{D}(\Phi(x),\Psi(x))$. The objects of $P$ are all arrows $(a \xrightarrow{\pi} \prod_{x \in Ob(C)}h_{D}(\Phi(x),\Psi(x))$, such that for any $x,y \in Ob(C)$, for any $(t_{0} \xrightarrow{t} h_{C}(x,y)) \in Arr(A)$, tensoring $\pi$ with $t$, projecting to the $y$-component $\prod_{x \in Ob(C)}h_{D}(\Phi(x),\Psi(x)) \rightarrow h_{D}(\Phi(y),\Psi(y))$, and composing in $D$ is $(sk)$-equal to tensoring $t$ with $\pi$, projecting to the $x$-component, and composing in $D$. One defines $p : P \longrightarrow A$ to be the functor which remembers the domain of a given arrow. One associates to $\Phi$ and $\Psi$ the object $colim(p) \in Ob(A)$ (assuming that the colimit exists).

One composes, i.e. defines, for all $(A,\otimes)$-functors $\Phi,\Psi,\Xi$, an arrow 
$$
(h(\Phi,\Psi) \otimes h(\Psi,\Xi) \xrightarrow{\circ} h(\Phi,\Xi)) \in Arr(A)
$$ 

by taking the inverse of the arrow $colim \ (P_{\Phi,\Psi} \otimes P_{\Psi,\Xi}) \rightarrow (colim \ P_{\Phi,\Psi}) \otimes (colim \ P_{\Psi,\Xi})$ (that this is an isomorphism is assumed), and recognizing $colim \ (P_{\Phi,\Psi} \otimes P_{\Psi,\Xi})$ as an object in $P_{\Phi,\Xi}$ by using the composition in $D$ and the projection for the products to define arrows $P_{\Phi,\Psi}(x) \otimes P_{\Psi,\Xi}(y) \rightarrow \prod_{x \in Ob(C)}h_{D}(\Phi(x),\Xi(x))$. As an object in $P_{\Phi,\Xi}$, $colim \ (P_{\Phi,\Psi} \otimes P_{\Psi,\Xi})$ has assigned to it an arrow into $colim \ P_{\Phi,\Xi}$, which is defined to be the hom object assigned to $\Phi$ and $\Xi$. One composes this colimit arrow with the inverse of the first arrow to define the composition arrow.

\sss{Lemma on the Enrichment of $Hom_{WE_{(A,\otimes)(sk)}}(C,D)$} \label{HomEnr}

Suppose that $(A,\otimes)$ has a symmetrizer and associator for the tensor. 

\ref{HomEnr}.1. For any $\Phi,\Psi \in Hom_{WE(A,\otimes)(sk)}(C,D)$, define 
$$
P \subseteq \downarrow_{( A )}(id_{A} , ob_{(A)}(\prod_{x \in Ob(C)} h_{D}(\Phi(x),\Psi(x))) )
$$

to be the full subcategory generated by objects (i.e. arrows $a \xrightarrow{\pi} \prod_{x \in Ob(C)} h_{D}(\Phi(x),\Psi(x))$ in $A$) such that for any $x,y \in Ob(C), (t_{0} \xrightarrow{t} h_{C}(x,y)) \in Arr(A)$, 
$$
sk(\circ_{D} \circ (id_{h_{D}(\Phi(y),\Psi(y))} \otimes \Psi(x,y)) \circ ((\pi \circ \pi_{y}) \otimes t)) =
$$
$$ 
sk(\circ_{D} \circ (\Phi (x,y) \otimes id_{h_{D}(\Phi(x),\Psi(x))}) \circ \sigma \circ ((\pi \circ \pi_{x}) \otimes t))
$$

Then define $\bar{h}_{WE(A,\otimes)(sk)1}(C,D)(\Phi,\Psi)$ to be the colimit of the domain object functor $p : P \longrightarrow A$ defined by $(a,f) \mapsto a$.

\ref{HomEnr}.2. Suppose that the compostion on $D$ is $(sk)$-associative, and the arrows $u = u_{(p)} : colim \otimes_{i=1}^{n} \ p_{i} \rightarrow \otimes_{i=1}^{n} colim \ p_{i}$ are isomorphisms, defined as in the previous lemma, and $p = \{ (1,p_{\Phi,\Psi}) \} \cup \{ (2,p_{\Psi,X} ) \} : \{1,2\} \rightarrow Arr(\mathfrak{Cat})$. Then define the composition $\circ_{WE(A,\otimes)(sk)}(C,D)(\Phi,\Psi,X) \in Arr(A)$ by taking it to be the composition of the colimit arrow $e : colim (-_{1} \otimes -_{2}) \rightarrow h_{WE(A,\otimes)(sk)}(C,D)(\Phi,X)$ associated to the object $ (colim (-_{1} \otimes -_{2}),\phi ) \in Ob(P_{\Phi,X})$ determined by the arrow
$$
\phi : colim (-_{1} \otimes -_{2}) \rightarrow \prod_{a \in Ob(C)} h_{D}(\Phi(a), X(a))
$$

induced by sending any given $((a,f),(b,g)) \in Ob(P_{\Phi,\Psi}\times_{\mathfrak{Cat}} P_{\Psi,X})$ to the product arrow given to the assignment
$$
a \mapsto \circ_{D}(\Phi(a),\Psi(a),X(a))\circ (\pi_{(\Phi,\Psi)a}\circ f \otimes \pi_{(\Psi,X)a}\circ g)
$$

with $u^{-1}$, I.e.
$$
\circ_{WE(A,\otimes)(sk)}(C,D)(\Phi,\Psi,X) := e \circ u^{-1}
$$

\ref{HomEnr}.3. Define $\bar{h}_{WE(A,\otimes)(sk)}(C,D) := $
$$
(Hom_{WE_{(A,\otimes)(sk)}}(C,D),\bar{h}_{WE(A,\otimes)(sk)1}(C,D)( \ , \ ) , \circ_{WE(A,\otimes)(sk)}(C,D)( \ , \ , \ ) ) \in Ob(WE(A,\otimes))
$$

i.e. part i. gives the hom objects and part ii. gives the composition.

\ref{HomEnr}.4. $\bar{h}_{WE(A,\otimes)(sk)}(C,D)$ is $(sk)$-associative. If $(A,\otimes)$ has a unit $I$ such that $\circ_{D}$ is $(Yo^{opp}_{(0)}(I))$-associative, then so does $\bar{h}_{WE(A,\otimes)(sk)}(C,D)$.

\begin{rem}
The enrichment on $Hom_{WE_{(A,\otimes)(sk)}}(C,D)$, i.e. the objects $h(\Phi,\Psi)$ defined in the previous lemma for $(A,\otimes)$-functors $\Phi$ and $\Psi$, were initially constructed as $(sk)$-equalizers. I believe that the present construction can also be realized as an $(sk)$-equalizer, but by use of a diagram containing arrows of the form $[ ( Homfun(A)\circ (( - \otimes J) \times id_{A}), \circ \circ (\pi \otimes id_{J})) ] \in Arr(\Omega)$, and with restrictions on $A$.
\end{rem}

\sss{$\bold{Definition}$ of the Enriched Arrows Functor}
If $(A,\otimes) \in Ob(\mathfrak{TCat})$ has coproducts, then define
$$
\bar{Arr}_{(A,\otimes)} : WE_{(A,\otimes)} \longrightarrow A
$$

by $(S,h,\circ) \mapsto \coprod_{s,t \in S} h(s,t)$ and $(F_{0},F_{1}) \mapsto \coprod_{s,t \in S} F_{1}(s,t)$.

\begin{rem}
The functor $sk$ is not referred to in this definition. $\bar{Arr}_{(A,\otimes)}$ is the ``enriched arrow functor."
\end{rem}

\begin{lem}
$\bar{Arr}_{(A,\otimes)}$ is faithful.
\end{lem}

The following lemma concerns the self-enrichment of the category $WE_{(A,\otimes)(sk)}$. The enriched hom set defined in the previous lemma is denoted by ``$\bar{h}_{WE(A,\otimes)(sk)}(B,C)$." Part (i) of the following lemma defines the ``forward composition/pushforward functor," \newline $\bar{h}_{WE(A,\otimes)(sk)}(B,C) \rightarrow \bar{h}_{WE(A,\otimes)(sk)}(B,D)$. Part (ii) defines the ``backward composition/ pullback functor," $\bar{h}_{WE(A,\otimes)(sk)}(C,D) \rightarrow \bar{h}_{WE(A,\otimes)(sk)}(B,D)$. Part (iii) states that one can use these to define an arrow $(\bar{h}_{WE(A,\otimes)(sk)}(B,C) \times \bar{h}_{WE(A,\otimes)(sk)}(C,D)$ $\rightarrow$ $\bar{h}_{WE(A,\otimes)(sk)}(B,D))$ $\in Arr(WE_{(A,\otimes)(sk)})$ which gives the enriched composition in $WE_{(A,\otimes)(sk)}$.

\sss{Lemma on Composition Functors} \label{PushPull}
Given $(sk) \in Arr(Cat)$, for any $(F : C \rightarrow D)$, $(G : B \rightarrow C ) \in Arr(WE_{(A,\otimes)(sk)})$,

\ref{PushPull}.1.
$$
(F_{*} : \bar{h}_{WE(A,\otimes)(sk)}(B,C) \rightarrow \bar{h}_{WE(A,\otimes)(sk)}(B,D)) \in Arr(WE_{(A,\otimes)(sk)})
$$

is induced by $\prod_{a \in Ob(B)} h_{C}(\Psi_{1}(a),\Psi_{2}(a)) \rightarrow \prod_{a\in Ob(B)} h_{D}(F\circ \Psi_{1}(a),F\circ\Psi_{2}(a))$, which induces a functor $P_{\bar{h}_{WE(A,\otimes)(sk)}(B,C)(\Psi_{1},\Psi_{2})} \longrightarrow P_{\bar{h}_{WE(A,\otimes)(sk)}(B,D)(F\circ\Psi_{1},F\circ\Psi)}$, so that an arrow is induced from the colimit of the first diagram ($p_{\bar{h}_{WE(A,\otimes)(sk)}(B,C)(\Psi,\Psi_{2})}$ to the colimit of the second $p_{\bar{h}_{WE(A,\otimes)(sk)}(B,D)(\F\circ\Psi_{1},F\circ\Psi)}$.

\ref{PushPull}.2.
$$
(G^{*} : \bar{h}_{WE(A,\otimes)(sk)}(C,D) \rightarrow \bar{h}_{WE(A,\otimes)(sk)}(B,D)) \in Arr(WE_{Set(sk)}(A,\otimes))
$$

is induced by $\prod_{a \in Ob(C)} h_{D}(\Phi_{1}(a),\Phi_{2}(a)) \rightarrow \prod_{a \in Ob(B)} h_{D}(\Phi_{1}\circ G(a), \Psi_{2} \circ G(a))$, which is the product map induced by the assignment $(a \mapsto \pi_{G(a)})$.

These are analogues to the usual forward and backward functors associated to composition on either end of a functor category $Hom(B,C)$.

\ref{PushPull}.3. From an arrow of functors $\alpha : \times_{A} \rightarrow \otimes$, the previous two constructions, and the product structure, construct an arrow in $WE_{Set(A,\otimes)}$

$$
\bar{h}_{WE(A,\otimes)(sk)}(B,C) \times_{WE_{(A,\otimes)(sk)}} \bar{h}_{WE(A,\otimes)(sk)}(C,D) \longrightarrow \bar{h}_{WE(A,\otimes)(sk)}(B,D)
$$

(Not unique. Corresponding to the choice of the path $F_{1}\circ G_{1} \rightarrow F_{1} \circ G_{2} \rightarrow F_{2} \times G_{2}$)

\ref{PushPull}.4. Defining $\bar{\circ} : Ob(WE_{Ass(A,\otimes)(sk)})^{3} \mapsto Arr(WE_{Ass(A,\otimes)(sk)})$ by sending $(B,C,D) \in Ob(WE_{Ass(A,\otimes)(sk)}$ to the arrow in (iii),
$$
(Ob(WE_{Assoc(sk)}(A,\otimes)),\bar{h}_{WE_{(sk)}(A,\otimes)},\bar{\circ})
$$

is an $(WE_{Ass(sk)}(A,\otimes),\times_{WE_{Assoc(sk)}}(A,\otimes)$-enriched set, whose composition is $(Ob)$-\newline associative and $(sk\circ\bar{Arr}_{(A,\otimes)})$-associative.

\begin{proof}
Parts i. and ii. consist only in checking for $(sk)$-commutativity so that the constructions can be made. Part iii., states that for any $C,D,E \in Ob(WE(A,\otimes))$, for any $\Phi_{1},\Phi_{2},\Phi_{3} \in Ob(\bar{h}_{WE(A,\otimes)(sk)}(C,D))$, for any $\Psi_{1},\Psi_{2},\Psi_{3} \in Ob(\bar{h}_{WE(A,\otimes)(sk)}(D,E)$,
$$
\circ_{\bar{h}_{(C,E)}}(\Psi_{1}\circ \Phi_{1} , \Psi_{2} \circ \Phi_{2}, \Psi_{3} \circ \Phi_{3}) \circ
$$
$$
( \circ_{\bar{h}_{(C,E)}}(\Psi_{1}\circ \Phi_{1} , \Psi_{1} \circ \Phi_{2}, \Psi_{2}\circ \Phi_{2}) \otimes \circ_{\bar{h}_{(C,E)}}(\Psi_{2}\circ\Phi_{2} ,\Psi_{2}\circ \Phi_{3}, \Psi_{3}\circ\Phi_{3})) \circ
$$
$$
((\Phi_{2}^{*}\otimes\Psi_{1*}) \otimes (\Phi_{3}^{*} \otimes \Psi_{2*})) \circ \sigma_{1*} =_{(sk)}
$$
$$
\circ_{\bar{h}_{(C,E)}}(\Psi_{1}\circ\Phi_{1} , \Psi_{1}\circ\Phi_{3} , \Psi_{3}\circ\Phi_{3}) \circ(\Phi_{3}^{*} \otimes \Psi_{1*}) \circ (\circ_{\bar{h}_{(D,E)}}(\Psi_{1},\Psi_{2},\Psi_{3})\otimes \circ_{\bar{h}_{(C,D)}}(\Phi_{1},\Phi_{2},\Phi_{3}) ))
$$

given that $colim(p) \in P$ with a monic arrow into the relevant product, and that $\forall f,g : x \rightarrow \prod_{i \in I}y_{i}$, $\forall i \in I, sk(\pi_{i}\circ f) = sk(\pi\circ g) \Longrightarrow sk(f) = sk(g)$.

All arrows between the objects $\bar{h}_{WE(A,\otimes)(sk)}(C,D) \rightarrow \bar{h}_{WE(A,\otimes)(sk)}(C',D')$ commute with monic arrows $\bar{h}_{WE(A,\otimes)(sk)}(C,D) \rightarrow \prod_{c \in Ob(C)} h_{D}(F(c),G(c))$.

After taking the inverse of the isomorphism $\otimes colim \ p_{i} \leftarrow colim \otimes \ p_{i}$ (that this is an isomorphism is assumed), these maps are determined by the arrows $\Psi_{i*}$ and $\Phi_{i}^{*}$. On the components of the product $\Phi_{i}^{*}$ come from identity arrows and $\Psi_{i*}$ from $\Psi(a,b)$.
$$
\text{Diagram with two arrows,}
$$
$$
\prod_{a \in Ob(D)} h_{E}(\Psi_{1}(a),\Psi_{2}(a)) \otimes \prod_{a \in Ob(D)} h_{E}(\Psi_{2}(a),\Psi_{3}(a)) \otimes
$$
$$
\prod_{a \in Ob(C)} h_{D}(\Phi_{1}(a),\Psi_{2}(a)) \otimes \prod_{a \in Ob(C)} h_{D}(\Phi_{2}(a),\Phi_{3}(a))
$$
$$
\rightarrow h_{E}(\Psi_{1}\circ\Phi_{1}(a),\Psi_{3}\circ\Phi_{3}(a))
$$

 (one side is $\Phi^{*}_{3} \otimes \Phi^{*}_{3} \otimes \Psi_{1*} \otimes \Psi_{1*}$ and the other is $\Phi_{2}^{*} \otimes \Phi_{3}^{*} \otimes \Psi_{1*} \otimes \Psi_{2*}$). The $\Phi^{*}_{3} \otimes \Phi^{*}_{3} \otimes \Psi_{1*} \otimes \Psi_{1*}$ side is
$$
\Pi \rightarrow
$$
$$
h_{E}(\Psi_{1}\circ\Phi_{3} (a) , \Psi_{2} \circ \Phi_{3}(a)) \otimes h_{E}(\Psi_{2},\Phi_{3}(a),\Psi_{3}\circ\Phi_{3}(a)) \otimes h_{D}(\Phi_{1}(a),\Phi_{2}(a))\otimes h_{D}(\Phi_{2}(a),\Phi_{3}(a))
$$
$$
\xrightarrow{id \otimes id \otimes \Psi_{1}(\Phi_{1}(a),\Phi_{2}(a)) \otimes \Psi_{1}(\Phi_{2}(a),\Phi_{3}(a))}
$$
$$
h_{E}(\Psi_{1}\circ\Phi_{3}(a),\Psi_{2}\circ\Phi_{3}(a))\otimes h_{E}(\Psi_{2}\circ \Phi_{3}(a),\Psi_{3}\circ\Phi_{3}(a)) \otimes
$$
$$
h_{E}(\Psi_{1}\circ\Phi_{1}(a),\Psi_{1}\circ\Phi_{2}(a))\otimes h_{E}(\Psi_{1}\circ\Phi_{2}(a) , \Psi_{1}\circ\Phi_{3}(a))
$$
$$
\xrightarrow{\circ_{E}} h_{E}(\Psi_{1}\circ\Phi_{1}(a),\Psi_{3}\circ\Phi_{3}(a))
$$

The $\Phi_{2}^{*} \otimes \Phi_{3}^{*} \otimes \Psi_{1*} \otimes \Psi_{2*}$ side is
$$
\Pi \rightarrow
$$
$$
h_{E}(\Psi_{1}\circ\Phi_{2}(a),\Psi_{2}\circ\Phi_{2}(a)) \otimes h_{E}(\Psi_{2},\Psi_{3}(a),\Psi_{3}\circ\Phi_{3}(a)) \otimes h_{D}(\Phi_{1}(a),\Phi_{2}(a)) \otimes h_{D}(\Phi_{2}(a),\Phi_{3}(a))
$$
$$
\xrightarrow{id \otimes id \otimes \Psi_{1}(\Phi_{1}(a),\Phi_{2}(a)) \otimes \Psi_{2}(\Phi_{2}(a),\Phi_{3}(a)) \circ \sigma}
$$
$$
h_{E}(\Psi_{2}\circ\Phi_{3}(a),\Psi_{3}\circ\Phi_{3}(a))\otimes h_{E}(\Psi_{1}\circ\Phi_{2}(a),\Psi_{2}\circ\Phi_{2}(a))\otimes 
$$
$$
h_{E}(\Psi_{2}\circ\Phi_{2}(a),\Psi_{2}\circ\Phi_{3}(a))\otimes h_{E}(\Psi_{1}\circ\Phi_{1}(a),\Psi_{1}\circ\Phi_{1}(a))
$$
$$
\xrightarrow{\circ_{E}}  h_{E}(\Psi_{1}\circ\Phi_{1}(a),\Psi_{3}\circ\Phi_{3}(a))
$$

By definition of $P_{\bar{h}_{(C,D)}(\Psi_{1},\Psi_{2})}$, in particular, ``commutativity" of the composition with any arrow going through a hom object of $C$, the two arrows are $(sk)$-equal.

\end{proof}
\begin{rem}
On underlying ``objects" this is the usual composition (e.g. 1-composition, of functors).
\end{rem}

\subsection{n-Categories}

An n-Category is defined inductively as an object in the category of (n-1)-enriched categories. 

The refutations of this approach (that it returns strict n-categories) which I've read referred only to enrichments associative in the strict sense. I therefore expect that requiring only ($sk$)-associativity (in a sense to be made precise below) should sidestep this. n-Categories with their basic structures are inductively defined, referring to each other (and therefore inseparable).

\sss{
The inductive  construction of $n$-categories
}
We define, inductively and simultaneously, the 
\ben
``forgetful functors" (``objects functors'') $F(n)$, 
\i
natural transformations $\rho (n)$, 
\i
the ``associators" $\alpha(n)$, 
\i
the ``product functors" $\times (n)$, 
\i
``symmetrizers" $\sigma (n)$, 
\i
``unit objects" $I(n)$, 
\i
right and left unit arrows $\rho_{u}(n)$, $\lambda_{u}(n)$, 
\i
$(n)-equivalence$ of
$(n)$-categories,
\i
$(n)$-equivalence of $(n)$-functors, 
\i
the $(n)$-skeleton functor $sk(n)$, and 
\i
the $U'$-category $U(n)-\mathfrak{Cat}$. 
\een
Here, for any $n\in\mathbb{N}$ the category of $(n)$-categories $U(n)-\mathfrak{Cat}$ 
is the category of sets that are weakly enriched over the category 
$U(n-1)-\mathfrak{Cat}$ of $(n-1)$-categories. 

\sss{$\bold{Definition}$ of n-Category}\label{nCat}
Assuming that we have defined these objects for all integers $\le n$ we define them for $n+1$.

\ref{nCat}.1.
The ``forgetful", or ``objects" functor
is defined on $n$-categories and it takes an $n$-category (an enriched set) to the underlying set 
$$
F(n+1) := Yo^{opp}_{((U,n+1)-\mathfrak{Cat}) (0)}(I(n+1)) : (U,n+1)-\mathfrak{Cat} \longrightarrow U-\mathfrak{Set}
$$

\ref{nCat}.2.
Define a natural transformation
$$
\rho(n) : \times_{U'-\mathfrak{Set}} \circ (F(n) \times_{U'-\mathfrak{Cat}} F(n)) \rightarrow F(n) \circ \times (n)
$$

by the identity maps. 

\ref{nCat}.3.
Define the $(n)$-associator
$$
\alpha (n+1) : \times(n+1) \circ (\times(n+1) \times_{U'-\mathfrak{Cat}}  id_{U'-\mathfrak{Cat}} ) \rightarrow \times(n+1) \circ (id_{(U,n+1)-\mathfrak{Cat}} \times_{U'-\mathfrak{Cat}} \times(n+1))
$$

as arrow of functors $((U,n+1)-\mathfrak{Cat})^{3} \longrightarrow (U,n+1)-\mathfrak{Cat}$, defined on objects by the associator and on hom objects by $\alpha (n)$. 

\ref{nCat}.4.
Define the $(n)$-product functor
$$
\times(n+1) : (U,n+1)-\mathfrak{Cat} \times_{U'-\mathfrak{Cat}} (U,n+1)-\mathfrak{Cat} \longrightarrow (U,n+1)-\mathfrak{Cat}
$$

on objects by the usual product functor (arrow in $U'-\mathfrak{Cat}$), defined on objects by the usual product functor and on hom objects by $\times (n)$, $\sigma (n)$, and $\alpha(n)$. 

\ref{nCat}.5.
Define the symmetrizing transformation
$$
\sigma (n+1) : \times(n+1) \rightarrow \times \sigma_{U'-\mathfrak{Cat}}
$$

on underlying objects by the usual symmetrizer and on hom objects by $\times(n)$ and $\sigma(n)$. 

\ref{nCat}.6.
Define the unit object
$$
I(n+1) \in Ob((U,n+1)-\mathfrak{Cat})
$$

having $\{ \emptyset \}$ as its underlying set, and $I(n)$ for the hom object. 

\ref{nCat}.7.
Define the left and right unit arrows
$$
\rho_{u}(n+1), \lambda_{u}(n+1) \in Arr(Hom_{U'-\mathfrak{Cat}^{2}}^{(1)}((U,n+1)-\mathfrak{Cat},(U,n+1)-\mathfrak{Cat}))
$$
$$
\rho_{u}(n+1) : - \times I(n+1) \rightarrow Id
$$
$$
\lambda_{u}(n+1) : I(n+1)\times - \rightarrow Id
$$

By the usual units on objects and $\rho_{u}(n)$ and $\lambda_{u}(n)$ on the hom objects. 

\ref{nCat}.8.
Define, for any $C,D \in Ob((U,n+1)-\mathfrak{Cat})$, the statement

$(C,D) \text{ are }(n+1)equivalent_{0} \Longleftrightarrow$

There exist $F : C \rightarrow D$, $G : D \rightarrow C$, such that for any $(c_{1},c_{2}) \in Ob(C)$, $(d_{1},d_{2}) \in Ob(D)$, $F_{(1)}(c_{1},c_{2})$ and $G_{(1)}(d_{1},d_{2})$ are $(n)$-equivalences, and  $(G\circ F,id_{C}), (F\circ G, id_{D})$ are $(n+1)$-equivalent$_{1}$.

\ref{nCat}.9.
Define, for any $C,D \in Ob((U,n+1)-\mathfrak{Cat})$, for any $F,G \in Hom_{(U,n+1)-\mathfrak{Cat})}(C,D)$, the statement 

$(F,G)$ are $(n+1)-equivalent_{1} \Longleftrightarrow$ 

There exist 
$$
\phi \in F(n)(\bar{h}_{WE((U,n)-\mathfrak{Cat},\times(n))(sk(n))}(C,D)(F,G))
$$
$$
\psi \in F(n)(\bar{h}_{WE((U,n)-\mathfrak{Cat},\times(n))(sk(n))}(C,D)(G,F))
$$, 

such that the various arrows
$$
\bar{h}_{WE((U,n)-\mathfrak{Cat},\times(n))(sk(n))}(C,D)(F,F) \rightarrow \bar{h}_{WE((U,n)-\mathfrak{Cat},\times(n))(sk(n))}(C,D)(F,G)
$$
$$
\bar{h}_{WE((U,n)-\mathfrak{Cat},\times(n))(sk(n))}(C,D)(G,G) \rightarrow \bar{h}_{WE((U,n)-\mathfrak{Cat},\times(n))(sk(n))}(C,D)(G,F)
$$

given by the composition of $\bar{\circ}$, the arrow $\bar{I(n)} \rightarrow \bar{h}_{WE((U,n)-\mathfrak{Cat},\times(n))(sk(n))}(C,D)(F,G)$ associated to $\phi$ or $\psi$ (see \ref{HomEnr} part(iv).) and a unit arrow ($\lambda_{u}(n)$ or $\rho_{u}(n)$), are $(n)$-equivalences$_{0}$ (Slightly loose usage. Adapt part (8).) (i.e. they $(sk(n))$-invert one another).

Roughly speaking there are $(sk(n))$-natural transformations between $F$ and $G$, which induce forward and backward compostion functors by the unit and enrichment lemma, which are $(n)$-equivalences, and such that $\phi \circ \psi$ and $\psi \circ \phi$ induce $(n)$-equivalent functors to the identities for the respective hom objects. 

\ref{nCat}.10.
Define the $(n)$-skeleton functor $sk(n)$ as a quotient functor
$$
sk(n) : (U,n)-\mathfrak{Cat} \longrightarrow Q 
$$

where $Q$ is the category defined by
$$
Ob(Q) = Ob((U,n)-\mathfrak{Cat})
$$
$$
Hom_{(Q)}(C,D) :_{t}= \{ [ F ]_{(n)eq} \in \bold{2}^{(Hom_{((U,n)-\mathfrak{Cat})}(C,D))} ; F \in Hom_{((U,n)-\mathfrak{Cat})}(C,D) \}
$$

where $[F]_{(n)eq} = [G]_{(n)eq}$ iff $(F,G)$ are $(n)$-equivalent. 

\ref{nCat}.11.
Define the category of $(n+1)$-categories
$$
(U,n+1)-\mathfrak{Cat} := WE_{Ass((U,n)-\mathfrak{Cat},\times(n)(sk(n),\alpha(n))}
$$

to be the category of sets $(sk(n))$-associatively enriched over over the category of $(n)$-categories

\sss{} Parts (ii) and (iii) of the following lemma give construction for limits and colimits in $WE(A,\otimes)$, to be applied to the (co)limits appearing in the construction of the enriched hom sets.

\sss{Lemma on Limits and Colimits in $WE(A,\otimes)$} \label{LemLimWE}
For any $(A,\otimes) \in Ob(\mathfrak{TCat})$,

\ref{LemLimWE}.1. $For : WE_{(sk)}(A,\otimes) \longrightarrow WE_{(term \circ sk)}(A,\otimes)$ is faithful, where $term :$ $codom(sk)$ $\longrightarrow$ $\star$ is the functor whose codomain is the terminal category. I.e. one forgets that one had had a composition requirement.

\ref{LemLimWE}.2. The limit of $F : I \longrightarrow WE(A,\otimes)$ can be constructed by the limit of the underlying sets and $(a_{i},b_{i})_{i\in I} \mapsto lim F'_{1}$, where $F_{1}' : I \longrightarrow A$ is defined on objects by
$$
F_{1(0)}' : j \mapsto h_{F_{(0)}(j)}(a_{j},b_{j}) )
$$

\ref{LemLimWE}.3. For any $F : I \longrightarrow WE_{(sk)}(A,\otimes)$, if $\tau : colim \circ \otimes \rightarrow \otimes \circ (colim \times _{\mathfrak{Cat}} colim )$ is an isomorphism where hom objects $h_{F_{(0)}(i)}(x,y)$ are concerned, then the colimit can be similarly constructed, by 
$$
([(a,i)],[(b,j)]) \mapsto colim F'_{1} \circ cob\downarrow_{(Hom^{(1)}_{\mathfrak{Cat}^{2}}((\{ 1,2 \} ,...), I)} (ob_{(I)}(i) \cup ob_{(I)}(j),\Delta_{ (\{ 1,2 \})} )
$$

i.e. taking the colimit of all hom objects below both $i$ and $j$. Define composition by the arrow induced by tensoring the colimit arrows assigned to $([(a,i)],[(b,j)])$ and $([(b,j)],[(c,k)])$, composed with the inverse of $\tau$.

\begin{rem} 
The explicit description of limits and (co)limits is applied to verify in the following lemma the isomorphism required for part (ii) of \ref{HomEnr}.
\end{rem}

\begin{lem}
$\forall n \in \mathbb{N}$, $colim \circ \times(n) \rightarrow \times(n) \circ (colim \times _{\mathfrak{Cat}} colim )$ is an isomorphism.
\end{lem}
\begin{proof}
On the level of sets, this is the isomorphism given by $[(a_{i},b_{j})] \mapsto ([a_{i}],[b_{j}])$. By the previous lemma the product of enriched sets is given by taking the products of their hom objects, so that $\tau_{n+1} : colim \circ \times(n+1) \rightarrow \times(n+1) \circ (colim \times _{\mathfrak{Cat}} colim )$ is determined by $\tau_{0}$ on underlying set and $\tau_{n}$ on hom objects. By induction, $\tau_{n}$ is for any $n$ an isomorphism.
\end{proof}

The ``meaning" of the following theorem consists in the special cases of parts (iii) and (iv) of \ref{PushPull}.

\sss{$\bold{Theorem}$ on $(U,n)-\mathfrak{Cat}$}
The category $(U,n)-\mathfrak{Cat}$ is weakly enriched over itself. I.e. $((U,n+1)-\mathfrak{Cat},\bar{h}_{WE((U,n)-\mathfrak{Cat},\times (n))(sk(n)},\bar{\circ} (n) ) \in Ob(WE((U,n+1)-\mathfrak{Cat},\times (n+1))$. The hom set agrees with that given by applying the objects functor $Ob = F(n)$ to the hom $n$-category, i.e.  $Ob\circ\bar{Hom}_{(U,n)-\mathfrak{Cat}} \cong Hom_{(U,n)-\mathfrak{Cat}}$.

\begin{proof}
One must check that the constructions of \ref{HomEnr} (see part(ii)) and \ref{PushPull} can be applied at each step.

$sk(n)$-associativity is part of the definition of $(U,n)-\mathfrak{Cat}$. The isomorphism of the previous lemma is the only other requirement.
\end{proof}

\begin{rem}
The restriction of $WE((U,n)-\mathfrak{Cat},\times (n))$ to the subcategory of $(sk(n))$-associative enrichments is necessary for the construction of the hom set enrichment, which is necessary for the definition of the next skeleton functor, $sk(n+1)$).
\end{rem}

\begin{rem}
That $(U,n+1)-\mathfrak{Cat}$ as an enriched set is $sk(n+1)$-associative (and therefore properly an $(n+2)$-category) was expected, but not yet clear to me. By part (iv) of \ref{PushPull} it is associative with respect to the objects functor and $sk(n)\circ \bar{Arr}_{((U.n)-\mathfrak{Cat},\times (n))}$, i.e. it is $sk(n)$-associative with respect to each hom object ($n$-category).  The difficulty seems to be in inferring, from the arrows giving the equivalences within the hom objects, arrows giving equivalences from without. I suspect that this should be easier to do for particular types of $n$-categories.
\end{rem}

\begin{ex}
$(2)-\mathfrak{Cat} \in Ob(WE((2)-\mathfrak{Cat},\times(2)))$. The skeleton is used at the level of the hom objects, so that only the usual skeleton, $sk(1)$, is seen in this case. The objects are enriched sets.
$$
O = Ob((2)-\mathfrak{Cat}) = \{ \bar{C} = (C,h,\circ ) \}
$$

where the composition is $(sk)$-associative, where $sk = sk(1) : \mathfrak{Cat} \longrightarrow Q$ is the quotient functor determined by identifying isomorphic arrows (functors). The arrows are arrows of enriched sets
$$
\Phi = (\Phi_{0},\Phi_{1}) : ( C,h_{C} ,\circ_{C}) \rightarrow (D,h_{D},\circ_{D})
$$

respecting composition after the application of $(sk)$.

By the Hom-enrichment construction one associates to any $C,D \in Ob((2)-\mathfrak{Cat})$, $\Phi, \Psi \in Hom_{((2)-\mathfrak{Cat})}(C,D)$, the category $P_{\Phi,\Psi}$ of all arrows $( x \xrightarrow{f} \prod_{c \in Ob(C)} h_{D}(\Phi(c),\Psi(c))$ satisfying the $(sk)$-commutativity requirement. $p : P_{\Phi,\Psi} \longrightarrow \mathfrak{Cat}$ is the functor defined by $(( x, f ) \mapsto x)_{(x,f) \in Ob(P_{\Phi,\Psi})}$. By definition $\bar{h}_{2-\mathfrak{Cat}}(a,b)(\Phi,\Psi) := colim P_{\Phi,\Psi}$

The description of the enrichment on $(2)-\mathfrak{Cat}$ requires, for any $(C,D,E) \in O$, an arrow
$$
(\bar{h}_{2-\mathfrak{Cat}} (C,D) \times \bar{h}_{2-\mathfrak{Cat}}(D,E) \xrightarrow{\circ} \bar{h}_{2-\mathfrak{Cat}}(C,E)) \in Arr((2)-\mathfrak{Cat})
$$

representing composition. That the above is an arrow in $(2)-\mathfrak{Cat}$, interpreted, means that for any $\Phi_{1},\Phi_{2},\Phi_{3} \in Hom_{((2)-\mathfrak{Cat})}(C,D), \Psi_{1},\Psi_{2},\Psi_{3} \in Hom_{((2)-\mathfrak{Cat})}(D,E)$, 
$$
F \cong G \in Hom_{\mathfrak{Cat}}(
$$
$$
\bar{h}_{(2)-\mathfrak{Cat}}(C,D)(\Psi_{1},\Psi_{2}) \times \bar{h}_{(2)-\mathfrak{Cat}}(C,D)(\Psi_{2},\Psi_{3})) \times (\bar{h}_{2-\mathfrak{Cat}}(D,E)(\Phi_{1},\Phi_{2}) \times \bar{h}_{2-\mathfrak{Cat}}(\Phi_{2},\Phi_{3})),
$$
$$
\bar{h}_{2-\mathfrak{Cat}}(C,E)(\Phi_{1}\circ\Psi_{1},\Phi_{3}\circ\Psi_{3}))
$$

where
$$
F :_{t}= \bar{\circ}(C,E) \circ (\Phi_{1*} \times \Psi_{3}^{*}) \circ (\bar{\circ}(C,D) \times \bar{\circ}(D,E))
$$
$$
G :_{t}= \bar{\circ}(C,E) \circ (\bar{\circ}(C,E) \times \bar{\circ}(C,E)) \circ ((\Phi_{1*} \times \Psi_{2}^{*}) \times (\Phi_{2*} \times \Psi_{3}^{*})) \circ \sigma 
$$

where $\bar{\circ}(C,D)$ denotes the enriched composition in $\bar{h}_{2-\mathfrak{Cat}}(C,D)$. I.e., there is a function (arrow of sets) $\alpha : Ob(dom(F)) = Ob(dom(G)) \rightarrow Arr(\mathfrak{Cat})$ defining a natural isomorphism between the functors $F$ and $G$.

\end{ex}

\begin{pro}
If the $P$-colimit inclusion condition is satisfied for $(U,n)-\mathfrak{Cat}$, regarding the construction of the hom enrichment, then it is satisfied for $(U,n+1)-\mathfrak{Cat}$ as well. I.e., the two arrows $colim \ p \otimes e_{0} \rightarrow \prod_{c \in Ob(C)} h_{D}(\Phi(c),\Psi(c))$, one from right composition and the other from left composition, are $(n+1)$-equivalent.
\end{pro}
\begin{proof}
The forgetful functor is at each step given by the objects functor. In this case, $P$ is given by all arrows $(a\xrightarrow{\pi} \prod_{c \in Ob(C)} h_{D}(\Phi(c),\Psi(c))) \in Arr((U,n)-\mathfrak{Cat})$, such that for any arrow $(e_{0} \xrightarrow{e} h_{C}(x,y)) \in Arr((U,n)-\mathfrak{Cat})$ into a hom object in $C$, the two arrows (if $\otimes= \times(n)$) 
$$
r_{(a)},l_{(a)} :a \otimes e_{0} \rightarrow \bar{h}_{D}(\Phi(x),\Psi(y))
$$

one given by composition with $e_{0}$ on one side and the other by composition on the other, are $sk(n)$-equivalent. Therefore a choice of an $(n+1)$-equivalence of $(n+1)$-functors is still a choice of
$$
\phi \in  F(n)(\bar{h}_{WE((U,n)-\mathfrak{Cat},\times(n))(sk(n))}(r,l))
$$
$$
\psi \in  F(n)(\bar{h}_{WE((U,n)-\mathfrak{Cat},\times(n))(sk(n))}(l,r))
$$

where $\bar{h}_{WE((U,n)-\mathfrak{Cat},\times(n))(sk(n))}(r,l)$ is itself by construction a colimit of the domain object functor
$$
p = dob\downarrow_{((U,n)-\mathfrak{Cat})} (id_{(U,n)-\mathfrak{Cat}},ob_{((U,n)-\mathfrak{Cat})}(\prod_{x \in Ob(t_{0})} h_{D}(r_{(0)}(x),l_{(0)}(x)))) \circ \varepsilon : 
$$
$$
P \longrightarrow (U,n)-\mathfrak{Cat}.
$$

By the inclusion condition for the $n$ case the hom object assigned to $r$ and $l$ has a monic arrow into the product of hom objects $h_{D}(r_{(0)}(x),l_{(0)}(x))$. By the isomorphism of the previous lemma and the construction of the colimit in $WE_{(A,\otimes)(sk)}$ in the lemma before that, an arrow of functors $\phi \in Ob(\bar{h}_{WE((U,n)-\mathfrak{Cat},\times(n))(sk(n))}(l_{colim \ p},r_{colim \ p})))$ is a map of sets
$$
\phi : Ob(colim \ p \times (n) e_{0}) \cong Ob(colim \ p) \times Ob(e_{0}) = 
$$
$$
\{ (a, \pi) ; a \in dom(\pi) \text{ and } (dom(\pi),\pi) \in Ob(P) \} \times Ob(e_{0})
$$
$$
\rightarrow \bigcup Ob(h_{D}(l_{colim \ p} ,r_{colim \ p} ))
$$

Claim - That a choice argument implies the existence of a natural isomorphism $\phi$ from the natural isomorphism $\phi_{i}$.
\end{proof}

\sus{Addresses}


We introduce the notion of an address, which is sequence of hom objects, each nested within the previous by the n-categorical enrichment. It is essentially a book-keeping tool, meant to record the ``location of a $k$-arrow within an $n$-category."

\sss{$\bold{Definition}$ of the Empty $n$-Category}
$\emptyset_{U(1)-\mathfrak{Cat}} := (\emptyset , \emptyset, \emptyset, \emptyset, \emptyset ) \in Ob(U-\mathfrak{Cat}) = Ob(U(1)-\mathfrak{Cat})$ is the empty category, and $\forall n \in \mathbb{N}, \emptyset_{U(n+2)-\mathfrak{Cat}} := (\emptyset_{U(1)-\mathfrak{Cat}}, \emptyset, \emptyset) \in Ob(U(n+2)-\mathfrak{Cat})$ is the empty (n+2)-category.

\sss{$\bold{Definition}$ of Addresses}\label{DefAdd}

We define two address functions, one for objects in $(U,n)-\mathfrak{Cat}$ and one for arrows.

\ref{DefAdd}.1. For any $n \in \mathbb{N}$, $fAdd_{U(n)0} : Ob(U(n)-\mathfrak{Cat}) \rightarrow U'$ is defined to be the function which sends an $n$-category $x \in Ob(U(n))-\mathfrak{Cat})$ to the set of functions $\alpha : \{ 1,...,j \} \rightarrow U'$ such that for any $k \in \{ 1,...,j \}$, where $j \in \{0,...,n\}$, 
$$
\alpha (k) = ( a(k),b(k),C(k),h(k),\circ (k)
$$
$$
h(k)(a(k),b(k)) = (C(k+1),h(k+1),\circ (k+1) )
$$
$$
a(k),b(k) \in Ob(C(k))
$$
$$
x = (C(0),h(0),\circ(0))
$$
For any $n \in \mathbb{N}$, $Add_{U(n)0} : Ob(U(n)-\mathfrak{Cat}) \rightarrow U'$ is the function which sends an $n$-category $x$ as above to the set of functions $\alpha : \{ 0,...,j \} \rightarrow U'$ such that there exist $a,b,C,h,\circ$ for which $\alpha = (a(k),b(k))_{k \in \{ 0,...,j \}}$ and $(a(k),b(k),C(k),h(k),\circ (k) ) \in fAdd_{U(n)0}(x)$.

These assign to each $n$-category its set of ``(full) addresses," being sequences \newline $(a(i),b(i),\mathcal{C}(i),h(i),\circ (i))$ such that $(a(i+1),b(i+1))$ is a pair of objects in the base category $\mathcal{C}(i)$ of the $(n-i-1)$-category associated to the previous pair $(a(i),b(i))$ by the enrichment. $fAdd$ refers to the former list and $Add$ to the truncated latter.

The ``length," $|\alpha | = | (a,b)|$, will denote its order as a set.

\ref{DefAdd}.2. For any $n \in Ob(\mathbb{N})$, $Add_{U(n)1} : Arr(U(n)-\mathfrak{Cat}) \rightarrow U'$ is defined to be the function which sends $\phi \in Arr(U(n)-\mathfrak{Cat})$ to a function
$$
S : Add_{U(n)0}(dom(\phi)) \rightarrow \bigcup_{k \in \mathbb{N}} Arr(U(k)-\mathfrak{Cat})
$$
defined inductively, by requiring that
$$
S : \emptyset \mapsto \phi
$$
and that for any $(a,b) \in Add_{U(n)0}(dom(\phi))$, for any $\bar{\phi} \in Arr(U(n-|(a,b)|)-\mathfrak{Cat})$, $S(a,b) := \bar{\phi}$ iff there exists $(a_{0},b_{0}) \in Add_{U(n)0}(dom(\phi))$ such that
$$
|(a_{0},b_{0})|+1 = |(a,b)| \text{ and } (a,b)|_{0,...,|(a_{0},b_{0})|-1} = (a_{0},b_{0})
$$

and there exists $\psi = ((f_{0},f_{1}),f_{2}) \in Arr(U(n-|(a,b)|+1)-\mathfrak{Cat})$, such that 
$$
\psi = S(a_{0},b_{0}) \text{ and } f_{2}(a(|(a,b)|),b(|(a,b)|)) = \bar{\phi}
$$

This associates to every arrow of $n$-categories a function which sends an address for the domain category to the arrow of ($n-k$)-categories assigned to it by the original arrow.

\begin{rem}
That the above definition consists of two maps, one for $n$-categories and the other for arrows of $n$-categories, suggests some functor giving an alternate description of $n$-categories.
\end{rem}

\sss{$\bold{Definition}$ of the Functors $Inc^{U(m)-\mathfrak{Cat}}_{U(n)-\mathfrak{Cat}}$ and $For^{U(m)-\mathfrak{Cat}}_{U(n)-\mathfrak{Cat}}$} \label{DefIncFor}
For any $n,m \in \mathbb{N} \backslash \{ 0 \}$ such that $n < m$, define functors $Inc^{U(m)-\mathfrak{Cat}}_{U(n)-\mathfrak{Cat}} : U(n)-\mathfrak{Cat} \longrightarrow U(m)-\mathfrak{Cat}$ and $For^{U(m)-\mathfrak{Cat}}_{U(n)-\mathfrak{Cat}} : U(m)-\mathfrak{Cat} \longrightarrow U(n)-\mathfrak{Cat}$ inductively, by the following.

\ref{DefIncFor}.1. For any $x = (C,h,\circ) \in Ob(U(n+1)-\mathfrak{Cat})$,
$$
Inc_{0U(n+1)-\mathfrak{Cat}}(x) :_{t}= (C,Inc_{0U(n)-\mathfrak{Cat}} \circ h, Inc_{1U(n)-\mathfrak{Cat}} \circ \circ )
$$
and for any $\phi = (\phi_{0},\phi_{2}) \in Arr(U(n+1)-\mathfrak{Cat})$,
$$
Inc_{1U(n+1)-\mathfrak{Cat}}(\phi) :_{t}= (\phi_{0},Inc_{U(n)-\mathfrak{Cat}}(\phi_{2}))
$$
so that $Inc_{U(n+1)-\mathfrak{Cat}} := (Inc_{0U(n+1)-\mathfrak{Cat}},Inc_{1U(n+1)-\mathfrak{Cat}}) : U(n+1)-\mathfrak{Cat} \longrightarrow U(n+2)-\mathfrak{Cat}$.

Now temporarily define $Inc_{U(1)-\mathfrak{Cat}} : U-\mathfrak{Cat} \longrightarrow U(2)-\mathfrak{Cat}$ to be the functor which sends a category $C$ to the 2-category with enrichment $h_{C}(a,b) := (Hom_{C}(a,b), \{ id_{f} ; f \in Hom_{C}(a,b) \} ,... )$ given by attaching only identity arrows. Define
$$
Inc^{U(m+1)-\mathfrak{Cat}}_{U(n)-\mathfrak{Cat}} := Inc_{U(m)-\mathfrak{Cat}} \circ Inc^{U(m)-\mathfrak{Cat}}_{U(n)-\mathfrak{Cat}}, \text{ and }
$$
$$
Inc^{U(2)-\mathfrak{Cat}}_{U(1)-\mathfrak{Cat}} := Inc_{U(1)-\mathfrak{Cat}}
$$

\ref{DefIncFor}.2. Similarly, for any $x = (C,h,\circ ) \in Ob(U(m+1)-\mathfrak{Cat})$,
$$
For_{0}^{U(m+1)-\mathfrak{Cat}}(x) :_{t}= (C,For_{0}^{U(m)-\mathfrak{Cat}} \circ h ,For_{1}^{U(m)-\mathfrak{Cat}} \circ \circ )
$$

and for any $\phi = (\phi_{0},\phi_{2}) \in Arr(U(m+1)-\mathfrak{Cat})$,
$$
For^{U(m+1)-\mathfrak{Cat}}_{1}(\phi) :_{t}= (\phi_{0},For^{U(m)-\mathfrak{Cat}}_{1}(\phi_{2}))
$$

so that $For^{U(m+1)-\mathfrak{Cat}} :_{t}= (For^{U(m+1)-\mathfrak{Cat}}_{0},For^{U(m+1)-\mathfrak{Cat}})_{1}) : U(m+1)-\mathfrak{Cat} \longrightarrow U(m)-\mathfrak{Cat}$.

Now temporarily define $For^{U(2)-\mathfrak{Cat}} : U(2)-\mathfrak{Cat} \longrightarrow U-\mathfrak{Cat}$ to be the functor which forgets the enrichment. Define
$$
For^{U(m+1)-\mathfrak{Cat}}_{U(n)-\mathfrak{Cat}} := For^{U(m)-\mathfrak{Cat}}_{U(n)-\mathfrak{Cat}} \circ For^{U(m+1)-\mathfrak{Cat}} \text{ and }
$$
$$
For^{U(n)-\mathfrak{Cat}}_{U(n)-\mathfrak{Cat}} := id_{U(n)-\mathfrak{Cat}}
$$

\begin{lem}
$\forall n \in \mathbb{N}$, $U(n+1)-\mathfrak{Cat}$ has products and coproducts.
\end{lem}
\begin{proof}
For products, by induction on $n$. At the base take the usual product category. For any tuple $(x_{i})_{i \in S}$, $(y_{i})_{i \in S}$, use the inductive step to take the product $\prod_{i \in S} h_{\mathcal{C}_{i}}(x_{i},y_{i})$.

For coproducts, at the base take the usual coproduct category (objects are the disjoint union. $Hom_{\coprod_{i \in S}}((a,j),(b,k))$ is $\empty$ for $j \neq k$, and $Hom_{\mathcal{C}_{j}}(a,b)$ for $j = k$). If $n \geq 1$, then for the enrichment, $h_{\coprod_{i \in S} \mathcal{C}_{i}}((a,j),(b,k))$  is $\emptyset_{U(n)-\mathfrak{Cat}}$ for $j \neq k$, and $h_{\mathcal{C}_{j}}(a,b)$ for $j = k$.
\end{proof}

\sss{$\bold{Definition}$ of Products and Coproducts}
$\prod _{U(n)-\mathfrak{Cat}}$ and $\coprod _{U(n)-\mathfrak{Cat}}$ will be functions \newline $\bigcup_{S \in U} Hom_{U'-\mathfrak{Set}}(S,Ob(U(n)-\mathfrak{Cat})) \longrightarrow Ob(U(n)-\mathfrak{Cat})$, the canonical constructions described in the previous lemma's proof.

\sss{$\bold{Definition}$ of the Restricted Simplicial Sets}
Define $\Delta \in Ob(U-mathfrak{Set})$ to to be the simplicial category, i.e. its objects are finite ordered sets and its arrows are order-preserving functions.

For any $n \in Ob(U-\mathfrak{Set})$, define the category $\Delta_{(n)} := \Delta_{\backslash (\{ j \in \mathbb{N} ; j \leq n-1 \} , \leq_{\mathbb{N}} )} = \downarrow_{(\Delta)}(ob_{(\Delta)}((\{ j \in \mathbb{N} ; j \leq n-1 \} , \leq_{\mathbb{N}} )),id_{(U-\mathfrak{Cat})(\Delta)} )$. This is the arrow category under the set with $n$ elements.

\sss{$\bold{Definition}$ of Primitive Arrows}
$\forall n \in \mathbb{N}$, $\forall f \in Arr(\Delta)$, $f$ is primitive iff $||dom(f)|-|codom(f)|| = 1$. $\forall \phi = (f, e, id_{\circ}) \in Arr(\Delta_{(n)})$, $\phi$ is primitive iff $f$ is primitive.

\begin{lem}
Any arrow in $\Delta$ or $\Delta_{(n)}$ is a composition of primitive arrows.
\end{lem}

\sss{Lemma on a Pseudo-Simplicial Structure on $(U.n)-\mathfrak{Cat}$}\label{LemSimp}
For any $n \in \mathbb{N}$, there exists a unique 
$$
\rho \in Hom_{(U''-\mathfrak{Cat})}(\Delta_{(n)}, \downarrow_{(U'-\mathfrak{Cat})}(ob_{(U'-\mathfrak{Cat})}(U(n)-\mathfrak{Cat}),id_{U'-\mathfrak{Cat}}))
$$ 

such that for any $\phi = (f,id_{(\{1,...,n\}, \leq )}) \in Arr(\Delta_{(n)})$, $f$ is primitive implies the following.

\ref{LemSimp}.1. If $f$ injective, then $\rho_{(1)}(\phi) : U(|dom(f)|)-\mathfrak{Cat} \longrightarrow U(|codom(f)|)-\mathfrak{Cat}$ is defined on objects by
$$
\rho_{(1)}(\phi)_{(0)} : (C,h,\circ) \mapsto (D,\bar{h},\bar{\circ})
$$

iff
$$
For^{U(|codom(f)|)-\mathfrak{Cat}}_{U(|dom(f)|-1)-\mathfrak{Cat}}(D) = For^{U(|dom(f)|)-\mathfrak{Cat}}_{U(|dom(f)|-1)-\mathfrak{Cat}}(C)
$$

and for any full address $\alpha = (a,b,C_{\alpha},h_{\alpha},\circ_{\alpha}) \in fAdd_{U(|dom(f)|)}((C,h,\circ ))$, for any $k \in \{ 1,...,|dom(f)| \}$, $f(k+1) = f(k) + 2$ implies
$$
Ob(h_{\alpha}(k)(a(k),b(k))) = \{ \emptyset \} \text{ and }
$$
$$
\alpha \in fAdd_{U(|codom(f)|)}((D,\bar{h},\bar{\circ})) \text{ and }
$$
$$
\forall \bar{\alpha} = (\bar{a},\bar{b},C_{\bar{\alpha}},h_{\bar{\alpha}},\circ_{\bar{\alpha}} ) \in fAdd_{U(|codom(f)|)}((D,\bar{h},\bar{\circ})),
$$
$$
\ulcorner |\bar{\alpha}| = |\alpha |+1 \text{ and } \bar{\alpha}_{\{0,...,k\} } = \alpha \urcorner  \Longrightarrow h_{\bar{\alpha}}(|\bar{\alpha}|)(\emptyset ,\emptyset ) = h_{\alpha}(|\alpha|)(a(k),b(k))
$$

The functor $\rho_{1}(\phi)$ is defined on arrows by
$$
\rho_{1}(\phi)_{(1)} : F = ((F_{0},F_{1}),F_{2}) \mapsto ((G_{0},G_{1}),G_{2}) = G
$$

iff
$$
For^{U(|codom(f)|)-\mathfrak{Cat}}_{U(|dom(f)|-1)-\mathfrak{Cat}}(G) = For^{U(|dom(f)|)-\mathfrak{Cat}}_{U(|dom(f)|-1)-\mathfrak{Cat}}(F)
$$

and for any address $\alpha = (a,b) \in Add_{U(|codom(f)|)}(dom(G))$, for any $k \in \{ 1,...,|dom(f)| \}$,
$$
\ulcorner f(k+1) = f(k)+2 \text{ and } |\alpha | = k+1 \urcorner \Longrightarrow 
$$
$$
Add_{U(|codom(f)|)1}(G)(\alpha) = Add_{U(|dom(f)|)1}(F)(\alpha|_{ \{ 0,...,k \} } )
$$

\ref{LemSimp}.2. If $f$ is surjective, then $\rho_{1}(\phi) : U(|dom(f)|)-\mathfrak{Cat} \longrightarrow U(|codom(f)|)-\mathfrak{Cat}$ is defined on objects by
$$
\rho_{1}(\phi)_{(0)} :(C,h,\circ) \mapsto (D,\bar{h},\bar{\circ})
$$

iff
$$
For^{U(|dom(f)|)-\mathfrak{Cat}}_{U(|codom(f)|-1)-\mathfrak{Cat}}(C) = For^{U(|codom(f)|)-\mathfrak{Cat}}_{U(|codom(f)|-1)-\mathfrak{Cat}}(D)
$$

and for any full address $\alpha = (a,b,C_{\alpha},h_{\alpha},\circ_{\alpha}) \in fAdd_{U(|codom(f)|)0}(D)$, for any $k \in \mathbb{N}$, $f(k+1) = f(k)$ implies
$$
(C_{\alpha}(k), h_{\alpha}(k), \circ_{\alpha}(k) ) = \coprod_{\bar{\alpha} \in S } (C_{\bar{\alpha}}(k+1),h_{\bar{\alpha}}(k+1), \circ_{\bar{\alpha}}(k+1))
$$

where 
$$
S = \{ \bar{\alpha} = (\bar{a},\bar{b},C_{\bar{\alpha}},h_{\bar{\alpha}},\circ_{\bar{\alpha}} ) \in fAdd_{U(|dom(f)|)0}(C) ;
$$
$$
(\bar{a},\bar{b})_{\{0,...,k-1\}} = (a,b)_{ \{ 0,...,k-1\} } \text{ and } |\bar{\alpha}| = k+1 \}.
$$

The functor $\rho_{1}(\phi)$ is defined on arrows by
$$
\rho_{1}(\phi) : F = ((F_{0},F_{1}),F_{2}) \mapsto ((G_{0},G_{1}),G_{2}) = G
$$

iff
$$
For^{U(|dom(f)|)-\mathfrak{Cat}}_{U(|codom(f)|-1)-\mathfrak{Cat}}(F) = For^{U(|codom(f)|)-\mathfrak{Cat}}_{U(|codom(f)|-1)-\mathfrak{Cat}}(G)
$$

and for any full address $\alpha = (a,b,C_{\alpha},h_{\alpha},\circ_{\alpha}) \in fAdd_{U(|codom(f)|)0}(C)$, for any $k \in \mathbb{N}$, $f(k+1) = f(k)$ and $|\alpha | = k+1$ imply
$$
Add_{U(|codom(f)|)1}(G)(\alpha) = \coprod_{\bar{\alpha} \in S} Add_{U(|dom(f)|)1}(F)(\bar{\alpha})
$$

I.e. if $f$ is injective, delete the k-th step, replacing it with the coproduct of all n-k-1 categories appearing in the enriched homs there. If surjective, add a step, a base category with only one object, leaving its enriched hom as that which had preceded it.

\sss{Lemma on Representing the k-Arrows Functor}
Adopt the notation of (\ref{LemSimp}). Then for any arrow $f \in Arr(\Delta_{(n)})$ if the functor $R : (U,n)-\mathfrak{Cat} \longrightarrow (U,|f|)-\mathfrak{Cat}$ given by requiring that $\rho(f) = ( \cdot , R , (U,|codom(f)|)-\mathfrak{Cat})$, then the functor
$$
F_{(U,n)-\mathfrak{Cat}} \circ \rho(f) : (U,n)-\mathfrak{Cat} \longrightarrow U'-\mathfrak{Set}
$$

is representable. 

\sss{Remark} I expect there to be some enriched version of this.

\begin{lem}
Adjunction of functors given to opposite pairs of primitive arrows by $\rho$.
\end{lem}

\sss{Conjecture on $(sk)$-associativity for a Subcategory of $n-Cat$}\label{assncat}

For any $n \in \mathbb{N}$, for any $\bar{B} = (B,h_{B},\bar{\circ}) \in Ob((WE(U,n)-\mathfrak{Cat},\times(n)))$, if there exists $C \in Ob(U-\mathfrak{Cat})$, and 
$$
(Add(\bar{B}) \xrightarrow{\Phi} Arr(U-\mathfrak{Cat})), (Add(\bar{B}) \xrightarrow{\varepsilon} Arr(U-\mathfrak{Cat}) \in Arr(U'-\mathfrak{Set})
$$

satisfying the following, properties, then $\bar{B}$ is $sk(n)$-associative.


\ref{assncat}.1. $C$ has colimits.

\ref{assncat}.2. For any address $\beta = (B_{i},h_{i},\bar{\circ}_{i},a_{i},b_{i})_{i \in \{ 1,...,k \} } \in Add(\bar{B})$, for some $c,d \in Ob(C)$, the functor
$$
\varepsilon(\beta) : E \longrightarrow C_{/c} = \downarrow_{(C)}(id_{(C)},ob_{(C)}(c))
$$

is faithful, and
$$
\Phi(\beta) : For^{(U,n-|\beta|)-\mathfrak{Cat}}_{U-\mathfrak{Cat}} (B_{|\beta|}) \longrightarrow Hom^{(1)}_{U-\mathfrak{Cat}^{2}}(E,C_{/d})
$$

is an equivalence of categories, where ``$|\beta|$" denotes the order of $\beta$ as a set of pairs, i.e. the number of categories or pairs of objects appearing in the sequence.

\ref{assncat}.3. The functors $\Phi(\beta)$ agree with the composition given by the Hom enrichment lemma, (\ref{PushPull}), up to natural isomorphism. Explanation follows.

\ref{assncat}.3.1. Let there be three addresses $\beta,\beta_{1},\beta_{2} \in Add(\bar{B})$, such that 
$$
|\beta_{1}| = |\beta_{2}| = |\beta_{3}| = |\beta|+1 \text{ and } \beta = \beta_{1} \cap \beta_{2} \cap \beta_{3}
$$
and 

$$
b_{1(|\beta|+1)} = a_{2(|\beta|+1)}
\text{ and }
a_{3(|\beta|+1)} = a_{1(|\beta|+1)}
\text{ and }
b_{3(|\beta|+1)} = b_{2(|\beta|+1)}
$$

i.e. the addresses $\beta_{1}$ and $\beta_{2}$ correspond to a triple $a_{1(k+1)},b_{1(k+1)}=a_{2(k+1)},b_{2(k+1)} \in Ob(C_{k})$ in the underlying category for one of the hom objects, composed to yield $\beta_{3}$.

\ref{assncat}.3.2. Then there is a natural isomorphism of functors
$$
\bar{\circ}_{\mathfrak{Cat}} \circ (\Phi(\beta_{1}) \times \Phi(\beta_{2})) \cong \Phi(\beta_{3})\circ For^{(U,n-|\beta|)-\mathfrak{Cat}}_{U-\mathfrak{Cat}}(\bar{\circ}),
$$

where $\bar{\circ}_{\mathfrak{Cat}}$ is that of (Enrichment, \ref{PushPull}) for $(U,2)-\mathfrak{Cat}$.








\se{The Category $\Omega$ of Pointed Correspondences and the Spec Functors}

In (3.1) we define a category, $\Omega$, intended to contain several ``categories of geometric objects''.\ftt{See [6]). Most of the $\mathbb{F}_{1}$-geometries there are based upon the incorporation of alternate notions of affine space (e.g. by monoids). How Borger's geometry might be satisfactorily included is less clear.} The arrows of $\Omega$ are partially given by cocorrespondences of categories $C^{opp} \times D \longrightarrow \mathfrak{Set}$. In particular, any two objects of $\Omega$ have arrows between them, and we hope in future work to use this to extend or compare existing geometries, by considering arrows between schemes of different characteristics or between objects ``originating from different geometries".

In (3.2) we define the notion of a spec datum, a tuple generally denoted by ``$\bar{sp}$", which is essentially a functor $\mathcal{R} \longrightarrow \mathcal{T}$, whose codomain is equipped with a Grothendieck topology. Under certain conditions, one can associate to a spec datum $\bar{sp}$ an ``$\Omega$-lift", which should be a functor $\mathcal{R} \longrightarrow \Omega$.  

In (3.3), we define, for any functor $F : I \longrightarrow U-\mathfrak{Cat}$, a subcategory $\Pi_{\Gamma}(F) \subset \Omega$. Given a spec datum $\bar{sp}$, we define a particular functor $F_{\bar{sp}}$ (see \ref{ExPiGTop}). From this we define, given an $\Omega$-lift $\tilde{sp}$ of a spec datum $\bar{sp}$, a subcategory $\Pi_{Sch}(\bar{sp},\tilde{sp}) \subseteq \Pi_{\Gamma}(F_{\bar{sp}})$, in analogy to the definition of the category of schemes by the usual affine chart construction.

\sus{Definitions Regarding the Category of Pointed Correspondences}

We denote by $\Omega$ a  category of pointed categories 
whose arrows are pointed correspondences. There is, for every category $C$, a canonical functor $\kappa_{C} : C \longrightarrow \Omega$. 
A functor $F:C'\to C''$
 between two categories induces a a natural transformation 
between the two canonical functors
$\kappa_{C'}\to \kappa_{C''}$. 

It is intended that $\Omega$ should serve as a common locale 
for several types of geometries.
Here, ``geometry'' is used in the sense of categories such that their
objects have some local structure 
(e.g. those of schemes or manifolds). This is formalized by requiring
any  ``geometry'' $\Pi$ 
to be a sub-category of $\Om$ 
generated in $\Omega$ from a 
subcategory interpreted as
``inclusion of the affine objects''. The objects in these ``geometric" subcategories of $\Omega$ are given by pairs $(Sh,\mathcal{O}_{x})$, where $Sh$ is the category of sheaves on some Grothendieck topology and $\mathcal{O}_{x} \in Ob(Sh)$.

It is intended that $\Omega$-arrows, are given by correspondences between the categories of sheaves associated to ``schemes". They therefore should allow for morphisms beside the usual ones (i.e. those constructed from functors beside the usual pushforward/pullback functors), such as algebraic correspondences.

\sss{$\bold{Definition}$ of $\Omega$} \label{omega}
For any two universes $U \in U'$,
the $omega$ $category$ for $(U,U')$
is a $U'$ category $\Om$
whose objects are pointed $U$-categories and arrows are
the isomorphism classes of pointed $U$-correspondences.


\ref{omega}.1. The set of objects is the set of categories with a distinguished object, i.e. an object is a pair $(C,x)$, where $C \in Ob(U-\mathfrak{Cat})$ and $x \in Ob(C)$.

\ref{omega}.2. Suppose that $C,D \in U-\mathfrak{Cat}$ are categories, with $x \in Ob(C)$, and $y \in Ob(D)$. Consider the set of all pairs $(F,\phi)$, where $F : C^{opp} \times D \longrightarrow U-\mathfrak{Set}$ is a functor and $\phi \in F(x,y)$.
Define an equivalence relation on this set by requiring that 
$$
(F,\phi) \equiv (G,\psi)
$$ 
iff there exists an isomorphism of functors $\Phi : F \longrightarrow G$, such that $\Phi(x,y)(\phi) = \psi$. 

Then the hom set $Hom_{\Omega}((C,x),(D,y))$ is defined to be the set of all equivalence classes $[(F,\phi)]$ of such pairs.\ftt{Formally, they should carry the information of their intended domain and codomain, so that the hom sets should be disjoint. This is a peculiarity of the definition here employed of a category, which requires that a set of arrows should be specified.}

\ref{omega}.3. $\textit{(The relation } R'(F,G,x,z)\textit{)}$. In order to define the composition in $\Omega$, we define, for any functors 
$$
F : C^{opp} \times D \longrightarrow U-\mathfrak{Set} \longleftarrow D^{opp} \times E : G
$$

and for any $x \in Ob(C)$ and $z \in Ob(E)$, an equivalence relation $R'(F,G,x,z)$ on the set 
$$
\coprod_{y \in Ob(D)} F(x,y) \times G(y,z)
$$

It is that which is generated by requiring that for any $y,y' \in Ob(D)$, for any $(\phi,\psi) \in F(x,y)\times G(y,z)$, for any $(\phi',\psi') \in F(x,y')\times G(y',z)$, 
$$
(\phi,\psi) \cong_{R'(F,G,x,z)} (\phi',\psi')
$$ 

if there exists an arrow $(y \xrightarrow{f} y' ) \in Arr(D)$ such that 
$$
G(f,id_{z})(\psi') = \psi \text{ and } F(id_{x},f)(\phi) = \phi'
$$

This is to say roughly that we require $(\phi,\psi)$ and $(\phi',\psi')$ to be equivalent if they can be related by an arrow $(y \xrightarrow{f} y')$ in the middle category $D$, in the sense that $\phi' = F(f)\circ \phi$ and $\psi = \psi'  \circ G(g)$.

For convenience the relation will be written $R(F,G,x,z) = R(F,G)$ when $x$ and $z$ are understood from the context.

\ref{omega}.4. For any composable arrows $(C,x)\xrightarrow{[(F,\phi)]} (D,y) \xrightarrow{[(G,\psi)]} (E,z)$ in $\Omega$, the composition is defined by 
$$
[(G,\psi)] \circ [(F,\phi)]:= [(F *_{\Omega} G,[(\phi,\psi)]_{R'(F,G,x,z)})]
$$

where the functor $F *_{\Omega} G : C^{opp} \times E \longrightarrow U-\mathfrak{Set}$ is defined to send a pair $(x',z') \in Ob(C^{opp} \times E)$ to the set of equivalence classes with respect to $R'(F,G,x',z')$ of pairs $(\psi',\phi') \in F(x',y') \times G(y',z')$ over varying $y'$, and the distinguished equivalence class of pairs is $[(\psi,\phi)]_{R'(F,G,x,z)}$.

\ref{omega}.5. For $(C,x) \in Ob(\Omega)$, we define the identity arrow $id_{(C,x)} : (C,x) \rightarrow (C,x)$ to be
$$
id_{(C,x)} := [(Hom_{C},id_{(C)(x)})]_{R}
$$

i.e. it is defined by the pair consisting of the hom functor $Hom_{C} : C^{opp} \times C \longrightarrow U-\mathfrak{Set}$ and $id_{x} \in Hom_{C}(x,x)$.

\begin{pro} \label{omegaincat}
Every universe within a universe $U \in U'$ has a unique category $\Omega \in Ob(U'-\mathfrak{Cat})$ as constructed above.
\end{pro}
\begin{proof}
The details of the proof are in the following subsections. As a set, $\Omega$ is determined by the definition. It remains to show that it is a category, in particular, we will show in the following that the composition defined in (\ref{omega}.4) is associative (see \ref{omegaincat}.2) and that the identity properties hold (see \ref{omegaincat}.3). Associativity comes from the usual natural isomorphisms associating two different product functors, $((-_{1} \times -_{2}) \times -_{3}) \rightarrow (-_{1} \times ( -_{2} \times -_{3}))$. The identity is given by $id_{\Omega((C,x))} := [(Hom_{C},id_{x})]$. The required natural isomorphism for the identity is found by applying (\ref{omega}.3) to any pair $(f,g) \in Hom_{C}(a,x) \times F(x,y)$ to show that the arrow of functors given by $(f,g) \mapsto F_{(1)}((f,id_{y}))(g)$ is injective. Surjectivity is obvious.

\ref{omegaincat}.1. Recall the functors $F *_{\Omega} G$ constructed in (\ref{omega}.4), with functors $F$ and $G$ as in that section.

\ref{omegaincat}.2. $(Associativity)$. For any composible arrows $((B,c) \xrightarrow{[(F,\phi)]} (C,c) \xrightarrow{[(G,\psi)]}$ \newline $(D,d) \xrightarrow{[(H,\chi)]} (E,e)) \in Arr(\Omega)$, define the associator $\alpha$, the isomorphism of functors required in (\ref{omega}.2), by sending each object $(b,e) \in Ob(B^{opp} \times E)$ to the map of sets
$$
\alpha(b,e) : ((F *_{\Omega} G) *_{\Omega} H)(b,e) \rightarrow (F *_{\Omega} (G *_{\Omega} H))(b,e)
$$

defined by
$$
\alpha(b,e) : [ ( \phi ', [ (\psi' , \chi' ) ]_{R'(F,G)} ) ]_{R'(F *_{\Omega}G,H)} \mapsto [ ( [( \phi ',  \psi') ]_{R'(G,H)}, \chi' ) ]_{R'(F,G*_{\Omega} H)} )
$$

Obviously $\alpha$ is an isomorphism.

\ref{omegaincat}.3. $(Identity)$.  For any $(C,x) \xrightarrow{[(F,\phi)]} (D,y)\in Arr(\Omega)$, one must construct natural isomorphisms $F*_{\Omega}Hom_{C} \rightarrow F$, for the left unit, and $Hom_{D} *_{\Omega} F \rightarrow F$, for the right unit (See \ref{omega}.4). 


Define the left unit isomorphism by sending an object $(x',y') \in Ob(C^{opp} \times D)$ to the map of sets
$$
u_{l}(x',y') : (Hom_{D}*_{\Omega}F)(x',y') \rightarrow  F(x',y')
$$

defined by
$$
u_{l}(x',y') :  [ ( \psi' , \phi') ]_{R'(Hom_{D},F)} \mapsto F( \psi',id_{(dom_{(C)}(\phi')}) (\phi') .
$$

The following gives the right unit isomorphism, mapping $(x',y') \in Ob(C^{opp} \times D)$ to the map of sets
$$
u_{r}(x',y') : (F*_{\Omega}Hom_{C})(x',y') \rightarrow  F(x',y')
$$

defined by
$$
u_{r}(x',y') :  [ (\psi', \phi' ) ]_{R'(F,Hom_{C})} \mapsto F( id_{(codom_{(D)}(\psi')} , \phi' )(\psi') .
$$

\ref{omegaincat}.3.1. One must show that $u_{l}$ and $u_{r}$ are isomorphisms of functors. The argument is symmetric, and so only the right side (involving $u_{r}$) will be explicitly written. First injectivity. 

Consider any objects 
$$
x' \in Ob(C), y',y'',y''' \in Ob(D),
$$

any arrows 
$$
\psi'' \in Hom_{D}(y'',y') \text{ and } \psi''' \in Hom_{D}(y''',y'),
$$

and any elements
$$
\phi'' \in F(x',y'') \text{ and } \phi''' \in F(x',y''').
$$


Suppose that  $u_{l}$ sends two elements (equivalence classes of pairs) to the same element: 
$$
u_{r}(x',y')( [ (\phi'',\psi'') ]_{R'(Hom_{D},F)} ) =   u_{r}(x',y')([ (\phi''',\psi''') ]_{R'(Hom_{D},F)} ).
$$

Then the two elements of $F(x',y')$ are the same:
$$
F(id_{x'},\psi'')(\phi'') = F(id_{x'},\psi''')(\phi''').
$$

Therefore the relation of (\ref{omega}.3) implies an equality:
$$
[ (\phi'',\psi'') ]_{R'(Hom_{D},F)} = [ (F(id_{x'},\psi'')(\phi'') , id_{y'} )]_{R'(Hom_{D},F)} =
$$
$$
[ (F(id_{x'},\psi''')(\phi''') , id_{y'} )]_{R'(Hom_{D},F)} = [ (\phi''',\psi''') ]_{R'(Hom_{D},F)}.
$$

\ref{omegaincat}.3.2. Surjectivity of the map $u_{l}$ is straightforward. Given any objects
$$
x' \in Ob(C)\text{ and } y' \in Ob(D),
$$

for any element $\phi' \in F_{(0)}(x',y')$,
$$
u_{r}(x',y')( [ (\phi' , id_{y'} )]_{R'(Hom_{D},F)} ) = \phi' .
$$

\ref{omegaincat}.3.3. So is the compatibility of the distinguished elements,

$$
u_{r}(x,y)( [ (\phi , id_{y} ) ]_{R'(Hom_{D},F)} ) = \phi \circ id_{y} = \phi.
$$

\ref{omegaincat}.3.4. A symmetric argument should be given for $u_{l}$ and the left identity)

This completes the argument for the identity.
\end{proof}

\sss{Convention}
If a universe $U$ is understood, then ``$\Omega$" should refer to the above category. If not, then ``$\Omega_{U}$" will.

\sss{Remark}
All hom sets in $\Omega$ are non-empty. To see this, let $\star$ be the terminal category, i.e. that having only one object, $\emptyset \in Ob(\star)$, and one arrow, $id_{\emptyset}$. Denote, for any objects $(C,x),(D,y) \in Ob(\Omega)$, their terminal functors by $C \xrightarrow{t_{C}} \star \xleftarrow{t_{D}} D$. Then $[(Hom_{\star}\circ (t_{C}^{opp} \times t_{D}), id_{\emptyset})] \in Hom_{\Omega}((C,x),(D,y))$.

\sss{$\bold{Definition}$ of the Canonical Functor $\kappa$} \label{Can}
For any category $C$, define the functor $\kappa_{C} : C \longrightarrow \Omega$ by sending objects $x$ to $(C,x)$, i.e. so that $x$ is distinguished within $C$, and arrows $\phi$ to $[(Hom_{C},\phi)]$, i.e. so that $\phi$ is distinguished within the hom set given by $C$.

\begin{rem}
That $\kappa_{C(0)}$ (the map on objects) is injective is trivial.
\end{rem}

\begin{lem}
For any $(I,i) \in Ob(\Omega)$, for any $C \in Ob(U-\mathfrak{Cat})$, the functor $Yo_{\Omega }(I,i) \circ \kappa_{C}^{opp} : C^{opp} \longrightarrow U'-\mathfrak{Set}$ (where $Yo_{\Omega} : \Omega \hookrightarrow Hom^{(1)}_{U''-\mathfrak{Cat}}(\Omega^{opp},U'-\mathfrak{Set})$ is the Yoneda functor), sends an object $c \in Ob(C)$ to the set of pairs consisting of isomorphism classes of functors $F : C^{opp} \longrightarrow Hom_{U-\mathfrak{Cat}^{2}}^{(1)}(I , U-\mathfrak{Set})$ with distinguished elements of the set $F(c)(i)$ given by $i$.
\end{lem}

\begin{lem}
For any category $C$, $\kappa_{C}$ is faithful iff the only automorphism of the identity functor $Id_{C} : C \longrightarrow C$ is the identity, i.e. $Aut_{End^{(1)}_{U-\mathfrak{Cat}^{2}} ( C )} ( Id_{C} ) = \{ id_{Id_{C}} \}$.
\end{lem}
\begin{proof}
The failure of $\kappa_{C}$ to be faithful would mean that there are distinct arrows $\phi, \psi : x \rightarrow y \in Arr(C)$ such that there is an isomorphism 
$$
(Hom_{C} \xrightarrow{\Phi} Hom_{C}) \in Arr( Hom^{(1)}_{U-\mathfrak{Cat}^{2}}( C^{opp} \times C , U - \mathfrak{Set}) )
$$

such that $\Phi(x,y)(\phi) = \psi$ . Consider the map of sets $f : Ob(C) \rightarrow Arr(C)$ given by $f: c \mapsto \Phi (c,c) (id_{c})$. By a Yoneda-type argument, $f$ is a non-trivial automorphism of the functor $Id_{C}$.

Now the reverse. If $g : Id_{C} \rightarrow Id_{C}$ is an natural isomorphism of functors, then define an arrow of functors, $\Phi : Hom_{C} \rightarrow Hom_{C}$, by requiring that for any $x,y \in Ob(C)$, for any $\phi \in Hom_{C}(x,y)$,
$$
\Phi(x,y) : \phi \mapsto \phi \circ g(x) = g(y) \circ \phi
$$

Naturalness on either side follows from the naturalness of $g$. Since $g$ is invertible, $\Phi$ can be inverted by
$$
\Phi^{-1}(x,y) : \phi \mapsto \phi \circ g(x)^{-1} = g(y)^{-1} \circ \phi,
$$

which is similarly also a natural transformation, since $g^{-1}$ is natural.
\end{proof}

\sss{Lemma} \label{gen} Suppose that a category $C$ is generated by objects $G \subseteq Ob(C)$, i.e. that $\prod_{g \in G} Yo^{opp}(g) : C \longrightarrow U-\mathfrak{Set}$ is faithful.

\ref{gen}.1. The functor $\kappa_{C}$ is faithful iff the set of automorphisms of each generator $g \in G$ is trivial.

\ref{gen}.2. If a category $C$ has an initial object $e$, then for any index category $I$ the category of functors from $I$ to $C$ has a set of generators, given by the left Kan extensions of the diagonal functors $\Delta_{(I_{i/},C)(0)}(g)$ for $g \in G$, along the inclusions in $I$ 
$$
i : I_{i/} = \downarrow_{(I)}( ob_{(I)}(i),id_{(I)}) \xrightarrow{cob} I
$$

of the arrow categories under varying objects in $I$, where $cob$ is the functor which remembers the codomain object of an arrow.

\begin{cor}
For any universes $U \in U' \in U''$, for any $D \in Ob(U-\mathfrak{Cat})$, $\kappa_{C}$ for $U''$ is faithful if $C$ is either $Hom_{U'-\mathfrak{Cat}^{2}}^{(1)}(D,U-\mathfrak{Set})$, $Hom_{U'-\mathfrak{Cat}^{2}}^{(1)}(D,U-\mathfrak{Cat})$ or some category of diagrams in either.
\end{cor}

\begin{cor}
Every $U-\mathfrak{Cat}$ object has a faithful functor into the omega of $U''$.
\end{cor}

\sss{$\bold{Definition}$ of Natural Transformations $\kappa_{tw,F}$ and $\kappa_{cotw}$}\label{kappatwist}
For any $(C \xrightarrow{F} D) \in Arr(U-\mathfrak{Cat})$, define the $\kappa-twist$ of $F$ to be the function $\kappa_{tw,F} : Ob(C) \rightarrow Arr(\Omega)$ given by composing the Hom functor $Hom_{D}$ with $F^{opp}$ for one of the arguments, with the identity arrow $id_{F(c)}$ in the codomain of $F$,
$$
\kappa_{tw,F} := (c \mapsto [(Hom_{D} \circ (F^{opp} \times_{U-\mathfrak{Cat}} id_{D} ) , id_{F(c)} ) ] )_{c \in Ob(C)}.
$$

The $\kappa-cotwist$ of $F$ is similarly given by $F$ and $id_{F(c)}$, being a function $\kappa_{cotw,F} : Ob(C) \rightarrow Arr(\Omega)$, so that
$$
\kappa_{cotw,F} := (c \mapsto [(Hom_{D} \circ ((id_{D}^{opp} \times_{U-\mathfrak{Cat}} F ) , id_{F(c)} ) ] )_{c \in Ob(C)}.
$$

\begin{lem}
Given a functor $F$ as above, the $\kappa$-twist and $\kappa$-cotwist of $F$ are natural transformations of functors,
$$
\kappa_{C} \xrightarrow{\kappa_{tw,F}} \kappa_{D} \circ F ,
$$
$$
\kappa_{D} \circ F \xrightarrow{\kappa_{cotw,F}} \kappa_{C},
$$

i.e. arrows in the category $Hom_{U'-\mathfrak{Cat}^{2}}^{(1)}(C,\Omega)$ of functors.
\end{lem}

\sus{$\Omega$-Lifts}

An $\Omega$-lift is essentially the construction of a structure sheaf associated to the topological spec functor $sp : \mathcal{R} \longrightarrow \mathcal{T}$ between categories, when the codomain of the concerned functor has been given a Grothendieck topology, and a functorial notion of an open immersion. 

An admissibility structure is defined to be some sub-functor of a ``fibre functor," which essentially assignes to each object in $\mathcal{T}$ a ``category of open subsets." For any functor $\mathcal{O} : \mathcal{R} \longrightarrow \mathcal{S}$ one constructs a sheaf on the the category assigned to each object $sp(x) \in Ob(\mathcal{T})$, by sheafifying the Kan extension of $\mathcal{O}$ restricted to the pre-image category of the functor which sends the arrow category in $\mathcal{R}$ over $x$ to the arrow category in $\mathcal{T}$ over $sp(x)$.

\sss{$\bold{Definition} \textit{ of a Fibre Functor}$} \label{fibre}
For any universes $U \in U'$, for any category $C \in Ob(U-\mathfrak{Cat})$ a functor $(C^{opp} \xrightarrow{Fib} U-\mathfrak{Cat}) 
$ is a fibre functor 
if the following, (\ref{fibre}.1) and (\ref{fibre}.2) hold.

\ref{fibre}.1. For any $c \in Ob(C)$, the category $Fib(c)$ is the 
category $C_{/c}$ of objects in $C$ that lie over $c$\ie of morphisms
$x\to c$.

\ref{fibre}.2. For any $(c_{1} \xrightarrow{f} c_{2}) \in Arr(C)$, the functor $Fib(f) : C_{/c_{2}} \longrightarrow C_{/c_{1}}$ is  the pullback functor, i.e., it
sends objects $x \xrightarrow{a} c_2$ to their fibred products by $f$, i.e. 
$$
Fib(f) : (x \xrightarrow{a} c_{2}) \mapsto (c_{1} \times_{c_{2}} x \xrightarrow{b} c_{1})
;$$
and it  sends an arrow $\phi :  (x\xrightarrow{a}c_{2}) \rightarrow (y\xrightarrow{b} c_{2})$ (i.e. $\phi:x\to y$ with $b\circ \phi = a$), to its pullback 
$\phi\times_{c_2} 1_{c_1} : x \times_{c_{2}} c_{1} \rightarrow y \times_{c_{2}} c_{1}$. 

\begin{lem} \label{Fibadj}
Consider any arrow $x \xrightarrow{\phi} y \in Arr(C)$.

\ref{Fibadj}.1. Define a category $L_{\phi} \in Ob(U-\mathfrak{Cat})$ by the following.

\ref{Fibadj}.1.1. Its objects are diagrams consisting of triples of arrows $(u,f,v) \in Arr(C)^{3}$ such that $v \circ f = \phi \circ u$

\ref{Fibadj}.1.2. Its arrows are natural transformations (given by certain pairs of arrows $(a,b) \in Arr(C)$).

\ref{Fibadj}.2. Define the functor $For_{c} : L_{\phi} \longrightarrow Fib(y)$, on objects by $(u,f,v) \mapsto v$, and on arrows by $(a,b) \mapsto b$.

\ref{Fibadj}.3. Define the functor $For_{d} : L_{\phi} \longrightarrow Fib(x)$, on objects by $(u,f,v) \mapsto u$, and on arrows by $(a,b) \mapsto a$.

\ref{Fibadj}.4. Suppose that $G$ is a right adjoint functor to $For_{c}$.

\ref{Fibadj}.5. Then $Fib(\phi) \cong For_{d} \circ G$.
\end{lem}

\begin{proof}
For any $v' \in Ob(Fib(y))$, $(u,f,v) \in Ob(L_{\phi})$, $u',f' \in Arr(C)$, 
$$
(u',f') fibres (\phi,v')  \Longrightarrow
$$
$$ 
Hom_{L_{\phi}}((u,f,v),G(v')) \cong Hom_{Fib(y)}(For_{c}((u,f,v)),v') =
$$
$$
Hom_{Fib(y)}(v,v') \cong Hom_{L_{\phi}}((u,f,v),(u',f',v')),
$$

where the last isomorphism $F_{(u,f,v)} : Hom_{Fib(y)}(v,v') \rightarrow Hom_{L_{\phi}}((u,f,v),(u',f',v'))$ is given, for any $b \in Hom_{Fib(y)}(v,v')$, by
$$
F_{(u,f,v)} : b \mapsto (a,b) \Longleftrightarrow f'\circ a = b \circ f \text{ and } u' \circ a = u,
$$

so that
$$
F_{(u,f,v)}^{-1}  : (a,b) \mapsto b.
$$

By the Yoneda lemma the right adjoint functor is determined to be the functor which sends an arrow over $y$ to its fibred product with $\phi$.
\end{proof}

\sss{Remark} A fibre functor is a functor determined by fibre diagrams. The previous lemma (\ref{Fibadj}) expresses the notion that the data of the arrows $f$ opposite to the arrows $\phi$ in the fibre diagrams by which one constructs $Fib$ are implicitly retained. The arrows $f$ are to be used in the following as ``localizations" of $\phi$.

\sss{$\bold{Definition}$ of an Admissibility Structure.}\label{admiss}
An admissibility structure $\varepsilon$ is a subfunctor of a fibre functor, with an identity object. I.e. it is a natural transformation
$$
\varepsilon : E \rightarrow Fib
$$

of $\mathfrak{Cat}$-valued functors, so that the following hold.

\ref{admiss}.1. $Fib$ is a fibre functor$(C)$. 

\ref{admiss}.2. For any $c \in Ob(C)$, $\varepsilon(c) : E(c) \hookrightarrow Fib(c)$ is faithful.

\ref{admiss}.3. For any $c \in Ob(C)$, there exists some $e \in Ob(E(c))$ such that $e$ is terminal, and $\varepsilon(c)(e) \cong (c,id_{c},\circ )$ (i.e. $e$ is terminsl, and sent to the terminal object in the arrow category, given by the identity arrow).

\sss{$\bold{Definition}$ of a Spec Datum.}
For any universes $U \in U'$, a spec datum$(U,U')$ is a tuple $(sp, \mathcal{O} , \tau, \varepsilon)$, consisting of two functors $(\mathcal{S} \xleftarrow{\mathcal{O}} \mathcal{R} \xrightarrow{sp} \mathcal{T}) \in Arr(U-\mathfrak{Cat})$ with the same domain, a Grothendieck topology $\tau$ on $codom(sp)$, and an admissibility structure $\varepsilon \in Arr(Hom^{(1)}_{U'-\mathfrak{Cat}^{2}}( codom(sp) , U-\mathfrak{Cat} ) )$ on $\mathcal{T}$.

\begin{rem}
I believe that there ought to be a definition for a category, whose object set is the set of spec data. I'd imagined at first the category of functors from the category with two arrows with the same domain to the category of categories, with the codomain of one functor having a topology and admissibility structure. But I am uncertain of which functors should be allowed between the categories with the topology and admissibility, and the ``correct" direction of functors on the common domain category. I might imagine a path category generated by all possibilities.
\end{rem}

\sss{$\bold{Definition}$ of $\Omega$-lifts of arrows in $\mathcal{R}$} \label{omegalift}

For any spec datum $(sp : \mathcal{R} \rightarrow \mathcal{T}, \mathcal{O} : \mathcal{R} \rightarrow \mathcal{S},\tau,\varepsilon )$, for any $(\phi : x \longrightarrow y) \in Arr(\mathcal{R})$, we say that an arrow $\tilde{\phi} \in Arr(\Omega)$ is an $\Omega$-lift of $\phi$ with respect to $(sp,\mathcal{O},\tau,\varepsilon)$ iff $\tilde{\phi}$ is constructed from the sheafification construction of a left Kan extension of $\mathcal{O}$ along a ``localization" of $sp$, i.e. iff $\tilde{\phi}$ can be constructed by the following procedure. 

To an object $x \in Ob(\mathcal{R})$ we associate an object $\tilde{x}$ in $\Omega$. The category component is the product of the categories $\mathcal{T}$ and the category of $\mathcal{S}$-valued sheaves on the category $E(sp(x))$ assigned by the admissibility to $sp(x)$,
$$
\tilde{x} = (\mathcal{T} \times Sh((E(sp(x)),\tau_{x}),\mathcal{S}^{opp})^{opp}, (sp(x),\mathcal{O}_{x})) \in Ob(\Omega).
$$

Consider the functor between arrow categories, $\sigma : \mathcal{R}_{/x} \longrightarrow \mathcal{T}_{/sp(x)}$ induced by $sp$, and the functor given by the admissibility $\varepsilon(sp(x)) : E(sp(x)) \longrightarrow \mathcal{T}_{/sp(x)}$. The pre-image (under $sp$) category $\mathcal{R}_{/x} \times_{\mathcal{T}_{/sp(x)}} E(sp(x))$ is the domain of the fibred product of the functors $\sigma$ and $\varepsilon(sp(x))$, and the codomain is the category $\mathcal{T}_{/sp(x)}$. The distinguished object $(sp(x),\mathcal{O}_{x})$ is roughly the sheafification of the left Kan extension of the functor $\mathcal{O}_{0} : \mathcal{R}_{/x} \times_{\mathcal{T}_{/sp(x)}} E(sp(x)) \longrightarrow \mathcal{S}$ given by $\mathcal{O}$. 

The functor $\mathcal{O}_{0}$ is a restriction of $\mathcal{O}$, which sends an object $u : U \rightarrow x$ over $x$, to $\mathcal{O}(U) \in Ob(\mathcal{S})$. 

Let the functor $\tilde{\mathcal{O}_{0}} : E(sp(x)) \longrightarrow \mathcal{S}$ be the left Kan extension of $\mathcal{O}_{0}$ along the pullback, $\tilde{\sigma} : \mathcal{R}_{/x} \times_{\mathcal{T}_{/sp(x)}} E(sp(x)) \longrightarrow E(sp(x))$, of the functor $\sigma$. Let the functor $\tilde{\mathcal{O}_{0}}^{opp} : E(sp(x))^{opp} \longrightarrow \mathcal{S}^{opp}$, be the opposite of $\tilde{\mathcal{O}_{0}}$. 

The functor $\mathcal{O}_{x} : E(sp(x))^{opp} \longrightarrow \mathcal{S}^{opp}$ is the sheafification of $\tilde{\mathcal{O}_{0}}^{opp}$ with respect to the pullback topology $\tau_{x}$ on $E(sp(x))$ of the topology $\tau$ on $\mathcal{T}$.

Generally, an arrow $\tilde{\phi} = [(F,\phi')] \in Arr(\Omega)$ consists of a functor $F$ and some $\phi' \in F(x',y')$, where $x'$ and $y'$ are the distiguished objects in the domain and codomain respectively. For any arrow $\phi : x \rightarrow y$ in $\mathcal{R}$, any $\Omega$-lift 
$$
(\mathcal{T} \times_{U-\mathfrak{Cat}} Sh((E(sp(x)),\tau_{x}),\mathcal{S}^{opp})^{opp}, (sp(x),\mathcal{O}_{x})) \xrightarrow{\tilde{\phi}}
$$
$$
(\mathcal{T} \times_{U-\mathfrak{Cat}} Sh((E(sp(y)),\tau_{y}),\mathcal{S}^{opp})^{opp}, (sp(y),\mathcal{O}_{y}))
$$

has the corresponding functor 
$$
F : (\mathcal{T} \times Sh((E(sp(x)),\tau_{x}),\mathcal{S}^{opp})^{opp})^{opp} \times (\mathcal{T} \times Sh((E(sp(y)),\tau_{y}),\mathcal{S}^{opp})^{opp}) \longrightarrow U-\mathfrak{Set}
$$

(so that it is determined by $\phi$). It is the composition of the hom functor of the category $\mathcal{T} \times Sh((E(sp(y)),\tau_{y}),\mathcal{S}^{opp})^{opp}$, with the pushforward functor
$$
id_{\mathcal{T}} \times sp(\phi)_{*} : \mathcal{T} \times Sh((E(sp(x)),\tau_{x}),\mathcal{S}^{opp})^{opp} \longrightarrow \mathcal{T} \times Sh((E(sp(y)),\tau_{y}),\mathcal{S}^{opp})^{opp}
$$

given by the admissibility structure, i.e. the functor $sp(\phi)_{*}$ sends a sheaf $\mathcal{F}$ to the composition $\mathcal{F} \circ E(sp(\phi))$), so that
$$
F = Hom_{\mathcal{T} \times Sh((E(sp(y)),\tau_{y}),\mathcal{S}^{opp})^{opp}} \circ ((id_{\mathcal{T}} \times sp(\phi)_{*}^{opp})^{opp} \times (id_{\mathcal{T} \times Sh((E(sp(y)),\tau_{y}),\mathcal{S}^{opp})^{opp}})).
$$

The second part of the arrow $\tilde{\phi}$, the distinguished element $(sp(\phi),\phi_{\sharp})$, is a natural transformation $\phi_{\sharp} : \mathcal{O}_{y} \rightarrow sp(\phi)_{*}(\mathcal{O}_{x})$ in the category $Sh(E(sp(y)),\tau_{y}),\mathcal{S}^{opp})$. We view it as an arrow in the opposite direction $sp(\phi)_{*}(\mathcal{O}_{x}) \xrightarrow{\phi_{\sharp}} \mathcal{O}_{y}$ in $Sh(E(sp(y)),\tau_{y}),\mathcal{S}^{opp})^{opp}$. We require that the arrow $\phi_{\sharp}$ is locally given by a choice of fibres of the arrow $\mathcal{O}(\phi)$ in $\mathcal{S}$. This is to say, that for any $u \in Ob(E(sp(y)))$, there exists a cover $\Gamma \subseteq Arr(E(sp(y)))$ of $u$ such that for any arrow $ (u' \xrightarrow{v} u) \in \Gamma$, the arrow $\phi_{\sharp}(u') \in Arr(\mathcal{S})$, assigned to $u'$ by the natural transformation $\phi_{\sharp}$, is the fibre along $t \circ \mathcal{O}_{y}(u')$, of the arrow $(\mathcal{O}(x) \xrightarrow{\mathcal{O}(\phi)} \mathcal{O}(y)) \in Arr(\mathcal{S})$, where $t$ is the composition of the universal arrows associated to the Kan extension, from applying the isomorphism
$$
Hom_{Hom_{U-\mathfrak{Cat}^{2}}^{(1)}(\mathcal{R}_{/y} \times_{\mathcal{T}_{/sp(y)}} E(sp(y)),\mathcal{S}^{opp})}(Hom_{U-\mathfrak{Cat}^{2}}^{(1)}(Id_{\mathcal{S}^{opp}},i) (\tilde{\mathcal{O}_{y0}}), \mathcal{O}_{y0})) \cong
$$
$$
Hom_{Hom_{U-\mathfrak{Cat}^{2}}^{(1)}(E(sp(y)),\mathcal{S}^{opp})}(\tilde{\mathcal{O}_{y0}},\tilde{\mathcal{O}_{y0}})
$$

to $id_{\tilde{\mathcal{O}_{y0}}}$, and sheafification, from applying the isomorphism
$$
Hom_{Sh((E(sp(y)),\tau_{y}),\mathcal{S}^{opp})}(\tilde{\tilde{\mathcal{O}_{y0}}},\tilde{\tilde{\mathcal{O}_{y0}}}) \cong
$$
$$
Hom_{Hom_{U-\mathfrak{Cat}^{2}}^{(1)}(E(sp(y)),\mathcal{S}^{opp})}(\tilde{\mathcal{O}_{y0}},For^{Sh((E(sp(y)),\tau_{y}),\mathcal{S}^{opp})}_{Hom_{U-\mathfrak{Cat}^{2}}^{(1)}(E(sp(y)),\mathcal{S}^{opp})}(\tilde{\tilde{\mathcal{O}_{y0}}})).
$$

to $id_{\tilde{\tilde{\mathcal{O}_{y0}}}}$.

All together, one can write $\tilde{\phi} =$
$$
[(Hom_{\mathcal{T} \times Sh((E(sp(y)),\tau_{y}),\mathcal{S}^{opp})^{opp}} \circ ((id_{\mathcal{T}} \times sp_{\phi*}^{opp})^{opp} \times (id_{\mathcal{T} \times Sh((E(sp(y)),\tau_{y}),\mathcal{S}^{opp})^{opp}})), (sp(\phi),\phi_{\sharp}))].
$$

\ref{omegalift}.1. Define a functor
$$
\sigma_{d} : \mathcal{R}_{/x} \longrightarrow \mathcal{T}_{/sp(x)}
$$

between arrow categories over objects corresponding to the domain of $\phi$, induced by $sp$, by $(U \xrightarrow{u} x) \mapsto (sp(U) \xrightarrow{sp(u)} sp(x))$.

\ref{omegalift}.2. The functor $\sigma_{c}$ on arrow categories over objects corresponding to the codomain of $\phi$ is constructed in a similar fashion.

\ref{omegalift}.3. Functors $p_{d},p_{c},q_{d},q_{c}$ are chosen so as to be fibred products of the above functors $\sigma_{d}$ and $\sigma_{c}$ with the sub-category arrows 
$$
\varepsilon(sp(x)) : E(sp(x)) \longrightarrow Fib(sp(x))
$$

and 
$$
\varepsilon(sp(y)) : E(sp(y)) \longrightarrow Fib(sp(y))
$$ 

given to the respective objects by the admissibility structure. I.e. \newline $(p_{d},q_{d})fibres(\sigma_{d},\varepsilon(sp(x)))$ and $(p_{c},q_{c})fibres(\sigma_{c},\varepsilon(sp(y)))$.

\ref{omegalift}.4. Denote by
$$
dob_{d} = dob\downarrow_{(\mathcal{R})}(id_{(U-\mathfrak{Cat})(\mathcal{R})},ob_{(\mathcal{R})}(x)) : \mathcal{R}_{/x} \longrightarrow \mathcal{R}
$$

and 
$$
dob_{c} = dob\downarrow_{(\mathcal{R})}(id_{(U-\mathfrak{Cat})(\mathcal{R})},ob_{(\mathcal{R})}(y)) : \mathcal{R}_{/y} \longrightarrow \mathcal{R}
$$

the domain object functors, defined by sending an arrow to its domain object. Recall that, for any category $C \in Ob(\mathfrak{Cat})$, we denote by 
$$
Yo^{opp}_{C} : C^{opp} \hookrightarrow Hom^{(1)}_{\mathfrak{Cat}^{2}}(C,\mathfrak{Set})
$$

the Yoneda embedding. Choose morphisms between $\mathfrak{Set}$-valued functors 
$$
\Phi_{d} : Yo^{opp}_{Hom^{(1)}_{U-\mathfrak{Cat}^{2}}(E(sp(x)),\mathcal{S})}(\tilde{\mathcal{O}}_{d}) \rightarrow Yo^{opp}_{Hom^{(1)}_{U-\mathfrak{Cat}^{2}}(dom(p_{d}),\mathcal{S})}(\mathcal{O}\circ dob_{d} \circ p_{d}) \circ Hom^{(1)}_{U-\mathfrak{Cat}^{2}}(q_{d},id_{\mathcal{S}})
$$ 

and 
$$
\Phi_{c} : Yo^{opp}_{Hom^{(1)}_{U-\mathfrak{Cat}^{2}}(E(sp(y)),\mathcal{S})}(\tilde{\mathcal{O}}_{c}) \rightarrow Yo^{opp}_{Hom^{(1)}_{U-\mathfrak{Cat}^{2}}(dom(p_{c}),\mathcal{S})}(\mathcal{O} \circ dob_{c} \circ p_{c}) \circ Hom^{(1)}_{U-\mathfrak{Cat}^{2}}(q_{c},id_{\mathcal{S}})
$$

so that they should be isomorphisms of hom sets, determining the left Kan extensions $\tilde{\mathcal{O}}_{d}$ of $\mathcal{O} \circ dob_{d} \circ p_{d}$ along $q_{d}$ and $\tilde{\mathcal{O}}_{c}$ of $\mathcal{O} \circ dob_{c} \circ p_{c}$ along $q_{c}$, respectively.

\ref{omegalift}.5. Choose arrows of functors
$(s_{d} : \tilde{\tilde{\mathcal{O}}}_{d} \rightarrow \tilde{\mathcal{O}}_{d})$ and $(s_{c} : \tilde{\tilde{\mathcal{O}}}_{c} \rightarrow \tilde{\mathcal{O}}_{c})$ so that they are universal arrows going from the sheafification of the Kan extension to the Kan extension in the opposites of their respective categories of sheaves (originating from the adjunction to the forgetful functor from sheaves to presheaves, the arrows having been reversed by the self-adjunction of $^{opp}$ so as to agree with the direction in $\mathcal{R}$).

\ref{omegalift}.6. Define the Grothendieck topology $\tau_{y} \subseteq \bold{2}^{Arr(E(sp(y)))}$ to be that induced from the topology $\tau$ on $\mathcal{T}$ given by the spec datum. I.e., for any set of arrows $\Gamma \in \bold{2}^{Arr(E(sp(y)))}$, $\Gamma \in \tau_{y}$ iff the image of $\Gamma$ in $\mathcal{T}$ by the domain object functor composed with the admissibility functor is a cover in $\mathcal{T}$, i.e. iff
$$
\{ dob \circ \varepsilon(sp(y))(c) \in Arr(\mathcal{T}) ; c \in \Gamma \} \in \tau
$$

where
$$
dob = dob\downarrow_{(\mathcal{T})}(id_{\mathcal{T}},ob_{\mathcal{T}}(sp(y))) : \mathcal{T}_{/sp(y)} \longrightarrow \mathcal{T}
$$

is the domain object functor.

\ref{omegalift}.7. Suppose that of the maps of sets
$$
e_{c}: Ob(E(sp(y))) \rightarrow Arr(E(sp(y)))
$$ 

and 
$$
e_{d}: Ob(E(sp(x))) \rightarrow Arr(E(sp(x)))
$$ 

each sends an object $u$ in its respective category to its terminal arrow $u \rightarrow t_{d},t_{c}$ for some fixed objects $t_{d} \in Ob(E(sp(x)))$ and $t_{c} \in Ob(E(sp(y)))$ satisfying (\ref{admiss}.3).

\ref{omegalift}.8. Composition with the sub-category functor given by the admissibility structure, \newline $E(sp(\phi))$, gives a functor from the category of sheaves over $x$ to that of sheaves over $y$ (the ``pushforward"). Let 
$$
\phi_{\sharp} : sp(\phi)_{*}(\tilde{\tilde{\mathcal{O}_{d}}}) \rightarrow \tilde{\tilde{\mathcal{O}_{c}}}
$$

be any arrow which satisfies the requirement, that for every object $u \in Ob(E(sp(y)))$ in the admissibility structure on the codomain, there is a cover $\Gamma \in \tau_{y}$ of $u$ such that for each $v \in \Gamma$, the arrows in $\mathcal{S}^{opp}$ given by $\phi$ and the terminal arrows from each domain of an arrow of the cover, composed with $s_{d}$ and $s_{c}$, form with $\phi_{\sharp}$ a fibre diagram in $\mathcal{S}$, i.e. 
$$
( \phi_{\sharp}(v) ,\Phi_{d}(\tilde{\mathcal{O}}_{x})(id_{\tilde{\mathcal{O}}_{x}}) \circ s_{d}(t_{d}) \circ \tilde{\tilde{O}}_{x}(e_{d}(E(sp(\phi))(v))) )
$$
$$
fibres_{(\mathcal{S})} 
$$
$$
( \Phi_{c}(\tilde{\mathcal{O}}_{y})(id_{\tilde{\mathcal{O}}_{y}}) \circ s_{c}(t_{c}) \circ \tilde{\tilde{O}}_{y}(e_{c}(v)) , \mathcal{O}(\phi))
$$

(which uses (\ref{admiss}.3), to determine arrows from the Kan extensions on the terminal object to the original objects $\mathcal{O}(x)$, $\mathcal{O}(y)$).

\ref{omegalift}.9. If the functor $F$ is given by the product of the hom functor of $\mathcal{T}$ with the hom functor on the category of sheaves over the codomain composed with the product of the the pushforward with the identity functor for the category of sheaves over the codomain, i.e. by
$$
F \cong Hom_{\mathcal{T} \times Sh((E(sp(y)),\tau_{y}),\mathcal{S}) } \circ (id_{\mathcal{T}} \times sp(\phi)_{*})
$$

then $\tilde{\phi} = [(F,(sp(\phi),\phi_{\sharp}))] \in Arr(\Omega)$.

\ref{omegalift}.10. The object markers (i.e. the domain and codomain) are $(sp(x),\tilde{\tilde{\mathcal{O}_{x}}})$ and $(sp(y),\tilde{\tilde{\mathcal{O}_{y}}})$.

\begin{rem}
The arrows $\tilde{sp}$ being determined by a product of functors and categories, two forgetful functors suggest themselves, one being to the image of $\kappa_{\mathcal{T}}$, and the other being to a subcategory of $\Omega$ containing objects $(Sh (... ) , sp(x))$, for the various topologies on categories given to objects of $\mathcal{R}$ by the admissibility structure.
\end{rem}

\begin{lem}
For any universes $U \in U'$, for any spec datum$(U,U')$, $(sp,\mathcal{O},\tau,\varepsilon)$, for a given $\tilde{\tilde{\mathcal{O}}}_{d}$ and $\tilde{\tilde{\mathcal{O}}}_{c}$.as above, if the cover $\Gamma$ of the previous definition an be chosen so that for all $v \in \Gamma$ with associated terminal arrow $e_{v}$, the arrow $\tilde{\tilde{\mathcal{O}}}_{c}(e_{v})$ is monic, then $\tilde{\phi}$ is unique.
\end{lem}

\begin{ex}
Localization of rings is epic, and so monic in the opposite category.
\end{ex}

\sss{$\bold{Definition}$ of an $\Omega$-lift Functor}
For any universes $U \in U'$, for any spec datum, $(sp,\mathcal{O},\tau,\varepsilon)$, we say that a functor $\tilde{sp} : \mathcal{R} \longrightarrow \Omega$, is an $\Omega$-lift of $(sp,\mathcal{O},\tau,\varepsilon )$ iff for any arrow $\phi \in Arr(\mathcal{R})$, the arrow $\tilde{sp}(\phi)$ is an $\Omega$-lift of $\phi$ with respect to $(sp,\mathcal{O},\tau,\varepsilon)$.

\begin{pro}
There is an isomorphism of functors between any two $\Omega$-lifts having the same spec datum.
\end{pro}

\begin{rem}
One might define a sheaf of categories, by the $\Omega$ lift of the spec datum differing from the original only in replacing $\mathcal{O}$ with a fibre functor on $\mathcal{S}$ composed with $\mathcal{O}$. If each object in the restrictions of such a sheaf over a given object $sp(x)$ gives a different manifestation of the resulting sheafification, then by the preceding proposition I'd imagine that the functor category of $\Omega$-lifts should correspond to a group (if the preceding lemma holds, and a fibre functor on $\mathcal{S}$ taken to be the neutral element, then the elements of the group should be given by the objects), functorially determined on $\mathcal{R}$, which for each $x \in Ob(\mathcal{R})$ has elements in a one-to-one correspondence with the set of $\Omega$-lifts of $sp$ restricted to the arrow category over $x$.
\end{rem}

\begin{lem}
Given a spec datum $\bar{sp} = (sp,\mathcal{O} : \mathcal{R} \rightarrow \mathcal{S},\tau,\varepsilon)$ an functor $F:\mathcal{S} \longrightarrow \mathcal{S}'$ induces an arrow of functors, one from the omega-lift of $\bar{sp}' = (sp,F\circ\mathcal{O},\tau,\varepsilon)$ to that of the original
$$
x \longmapsto \kappa_{tw}(i_{Hom_{U-\mathfrak{Cat}}^{(1)}(dom(\varepsilon)(sp(x),\mathcal{S})}^{Sh(...,\mathcal{S})} ) \circ \kappa_{cotw}(Hom^{(1)}_{U-\mathfrak{Cat}^{2}}(id_{\mathcal{T}} \times Hom^{(1)}_{U-\mathfrak{Cat}^{(2)}}(F ,id_{\mathcal{R}}) )) \circ
$$
$$
\kappa_{Hom_{U-\mathfrak{Cat}}^{(1)}(dom(\varepsilon)(sp(x),\mathcal{S}')}(u) \circ
 \kappa_{cotw}(i_{Hom_{U-\mathfrak{Cat}}^{(1)}(dom(\varepsilon)(sp(x),\mathcal{S}')}^{Sh(...,\mathcal{S}')} ).
$$




Where $u$ is the universal arrow from the sheafification.\ftt{Some smaller version might be desirable here}

\end{lem}


\sus{$End_{End^{(1)}_{(U-\mathfrak{Cat}^{2})} ( C )} ( Id_{C} )$ and $\Pi_{Sp}$}

In the present section we define a notion of the global sections functor $\Omega \supseteq \Pi \longrightarrow \mathcal{S}$, appropriate to such sub-categories of $\Omega$ as contain the images of a particular $\Omega$-lift. It is defined with respect to a compatible collection of functors, whose domains are the various categories of sheaves which appear in the objects defining the sub-category. We first define, for every functor $F : I \longrightarrow U-\mathfrak{Cat}$, a subcategory $\Pi_{\Gamma}(F) \subseteq \Omega$, whose arrows are those in $\Omega$ generated by the compositions of hom functors of various categories (the codomains of the functors $F(i)$) with the functors $F(i)$. 

\sss{$\bold{Definition}$ of $\Pi_{\Gamma}(F)$} \label{PiG}
Let $F : I \longrightarrow U-\mathfrak{Cat}$ be a functor.

\ref{PiG}.1. Define the ``category generated by $F$," $\Pi_{\Gamma}(F) \subseteq \Omega$ to be the subcategory of $\Omega$ generated by arrows whose functors factor through the hom functor of some category and some functor in the diagram $F$, i.e. are given by arrows of the form
$$
[(Hom_{codom(F(i))} \circ ( F(i)^{opp} \times id_{codom(F(i))}) , \phi )] \in Arr(\Omega),
$$

where $i \in Arr(I)$ and $\phi \in Arr(codom(F(i)))$.

\ref{PiG}.2. Dually, define the ``category co-generated by $F$," $\Pi_{co\Gamma(F)} \subseteq \Omega$ to be the subcategory generated by arrows of the form
$$
[(Hom_{codom(F(i))} \circ (id_{codom(F(i))}^{opp} \times F(i)) , \phi )] \in Arr(\Omega),
$$

where $i \in Arr(I)$ and $\phi \in Arr(codom(F(i)))$.


\sss{Example} For any category $C \in Ob(U-\mathfrak{Cat})$, if $F$ is the constant functor \newline $\Delta_{(I,U-\mathfrak{Cat})}(C) : I \longrightarrow U-\mathfrak{Cat}$, which sends every $i \in Arr(I)$ to the identity functor $Id_{C}$, then the category $\Pi_{\Gamma}(F)$ is the image in $\Omega$ of the functor $\kappa_{C} : C \longrightarrow \Omega$.

\sss{Example}\label{ExPiGTop} For any $\mathcal{S} \in Ob(U-\mathfrak{Cat})$, for any category $\mathcal{T} \in Ob(U-\mathfrak{Cat})$ with an admissibility structure (inclusion of $U-\mathfrak{Cat}$-valued functors) $\varepsilon : E \hookrightarrow Fib$ on $\mathcal{T}$, define a functor $F : \mathcal{T} \longrightarrow U-\mathfrak{Cat}$, on objects by sending an object $x \in Ob(\mathcal{T})$ to the category of $\mathcal{S}$-valued sheaves on the category $E(x)$, assigned by the admissibility structure to $x$.
$$
x \mapsto Sh((E(x),\tau_{X}),\mathcal{S}),
$$

and on arrows by the usual pushforward,
$$
f \mapsto Hom_{U-\mathfrak{Cat}^{2}}^{(1)}(E(f)^{opp},Id_{\mathcal{S}}) \circ e,
$$

being defined by the composition of any given sheaf with the arrow $E(f)$ given by the subfunctor given by the admissibility structure, where where $e$ is the inclusion of the category of sheaves into the category of presheaves. In this case $\Pi_{\Gamma}(F)$ can be thought of as the category of spaces with $\mathcal{S}$-valued sheaves, i.e. the objects are given by pairs $(x,\mathcal{O}_{x})$, where $x \in Ob(\mathcal{T})$ and $\mathcal{O}_{x} : E(x)^{opp} \longrightarrow \mathcal{S}$ is a sheaf on $E(x)$.

\sss{$\bold{Definition}$ of a Global Sections Datum}\label{GlobD}
For any category $I \in Ob(U-\mathfrak{Cat})$, let $I_{0} \in Ob(U-\mathfrak{Cat})$ be the subcategory which consists of the identity arrows of $I$, and $\varepsilon : I_{0} \longrightarrow I$ be the inclusion functor. A $global \ sections \ datum(I)$ is defined to be an element of the set $Ob(\Gamma)$, where $\Gamma \in Ob(U-\mathfrak{Cat})$ is defined to be $dom(p) = dom(q)$, where $p,q \in Arr(U-\mathfrak{Cat})$ and $(p,q) fibres (\varepsilon^{*},sk_{*})$, where
$$
\varepsilon^{*} : \downarrow_{(Hom_{U-\mathfrak{Cat}^{2}}^{(1)}(I,U-\mathfrak{SCat}))}(id_{(...)},id_{(...)}) \longrightarrow \downarrow_{(Hom_{U-\mathfrak{Cat}^{2}}^{(1)}(I_{0},U-\mathfrak{SCat}))}(id_{(...)},id_{(...)})
$$

and 
$$
sk_{*} : \downarrow_{(Hom_{U-\mathfrak{Cat}^{2}}^{(1)}(I_{0},U-\mathfrak{SCat}))}(id_{(...)},id_{(...)})\longrightarrow \downarrow_{(Hom_{U-\mathfrak{Cat}^{2}}^{(1)}(I_{0},U-\mathfrak{SCat}))}(id_{(...)},id_{(...)})
$$

where the functor $sk : U-\mathfrak{Cat} \longrightarrow U-\mathfrak{SCat}$ is the quotient functor defined by equating naturally isomorphic functors (arrows of categories).

I.e. it is a $sk$-natural transformation of $(U-\mathfrak{Cat})$-valued functors.

%
%


\sss{$\bold{Definition}$ of a Global Sections Functor, Transformation} \label{GlobFT}

Adopt the notation of (\ref{GlobD}). For any $sk$-natural transformation $(F \xrightarrow{t} G) \in Ob(P)$ for which $F$ is injective on objects (we essentially require the $Hom_{U-\mathfrak{Cat}^{2}}^{(1)}(I_{0},U-\mathfrak{Cat})$ data from the fibred product), we make the following constructions. 

\ref{GlobFT}.1. Define a function on objects
$$
\gamma_{t,0} : Ob(\Pi_{\Gamma}(F)) \rightarrow Ob(\Pi_{\Gamma}(G))
$$

by the $\kappa$-twist, i.e. for any $i \in Ob(I)$, for any $x \in Ob(F(i))$,
$$
\gamma_{t,0} : (F(i),x) \mapsto (G(i),t(i)(x)) = codom(\kappa_{tw,t(i)}(x)).
$$

Define a subset $\gamma_{t,1} \subseteq Arr(\Pi_{\Gamma}(F)) \times Arr(\Pi_{\Gamma}(G))$, by
$$
\gamma_{t,1} = \{
$$
$$
([(Hom_{codom(F(f))} \circ (F(f)^{opp} \times id_{codom(F(f))}) , \phi )],
$$
$$
[(Hom_{codom(G(f))} \circ (G(f)^{opp} \times id_{codom(G(f))}) , t(codom(f))(\phi) )])
$$
$$
\in  Arr(\Pi_{\Gamma}(F)) \times Arr(\Pi_{\Gamma}(G)) ; f \in Arr(I) \}.
$$

\ref{GlobFT}.2. Let $\Pi_{\Gamma}(F) \xrightarrow{\varepsilon_{F}} \Omega$ and $\Pi_{\Gamma}(G) \xrightarrow{\varepsilon_{G}} \Omega$ be the inclusion functors. Define a map of sets $\kappa_{\Pi-tw(t)} : Ob(\Pi_{\Gamma}(F)) \rightarrow Arr(\Omega)$ by the various $\kappa$-twists, i.e. so that for any $i \in Ob(I)$, for any $x \in Ob(F(i))$, $\kappa_{\Pi-tw(t)} : (F(i),x) \mapsto \kappa_{tw,t(i)}(x)$.

\sss{Proposition}\label{GlobFTf}
Suppose that for any $i \in Ob(dom(F))$, for any automorphism $\phi \in \newline Aut_{End_{U-\mathfrak{Cat}^{2}}(F(i))}(Id_{F(i)})$, there exists an automorphism $\tilde{\phi} \in Aut_{End_{U-\mathfrak{Cat}^{2}}(G(i))}(Id_{G(i)})$ such that
$$
Hom_{U-\mathfrak{Cat}^{2}}^{(1)}(t(i),Id_{G(i)})(\tilde{\phi}) \cong Hom_{U-\mathfrak{Cat}^{2}}^{(1)}(Id_{F(i)},t(i))(\phi),
$$

i.e. the automorphisms of $Id_{F(i)}$ ``lift" to automorphisms of $Id_{G(i)}$.\ftt{One might slightly weaken this, by requiring only that those automorphisms should lift which might equate two different arrows under the $\kappa_{F(i)}$ functor.
} Then the following hold.

\ref{GlobFTf}.1. The pair of maps $\gamma_{t} = (\gamma_{t,0},\gamma_{t,1})$ of $(\ref{GlobFT}.1)$ defines a functor,
$$
\gamma_{t} : \Pi_{\Gamma}(F) \longrightarrow \Pi_{\Gamma}(G),
$$

so that for any $f \in Arr(I)$,
$$
\gamma_{t,1} : [(Hom_{codom(F(f))} \circ (F(f)^{opp} \times id_{codom(F(f))}) , \phi )] \mapsto
$$
$$
[(Hom_{codom(G(f))} \circ (G(f)^{opp} \times id_{codom(G(f))}) , t(codom(f))(\phi) )].
$$

\ref{GlobFTf}.2. The map $(\ref{GlobFT}.2)$ defines a natural transformation,
$$
\kappa_{\Pi-tw(t)} : \varepsilon_{F} \rightarrow \varepsilon_{G} \circ \gamma_{t}.
$$

\sss{} Suppose that $F : \mathcal{T} \longrightarrow U-\mathfrak{Cat}$ is as in (\ref{ExPiGTop}). Consider, for any $x \in Ob(\mathcal{T})$, the ``global sections functor,"
$$
t(x) : Sh((E(x),\varepsilon(x)^{*}(\tau)),\mathcal{S}) \longrightarrow \mathcal{S}
$$

defined by $\mathcal{O} \mapsto \mathcal{O}(e)$, where $e \in Ob(E(x))$ is the terminal object. Then the map of sets $t : Ob(\mathcal{T}) \rightarrow Arr(U-\mathfrak{Cat})$ is an $sk$-natural transformation of functors, $t : F \rightarrow \Delta_{(\mathcal{T},U-\mathfrak{Cat})}(\mathcal{S})$.

\sss{} If the canonical functor $\kappa_{\mathcal{S}}$ is faithful, then there exists a functor $\gamma' : \Pi_{\Gamma}(F) \rightarrow \mathcal{S}$, unique such that $\kappa_{\mathcal{S}} \circ \gamma' = \gamma_{t}$.

\sss{Proposition on Diagrams in $\Pi_{\Gamma}(F)$} \label{PiDiag}
For any functor $F : I \longrightarrow U-\mathfrak{Cat}$ and any category $J \in Ob(U-\mathfrak{Cat}$, there exists a functor $\bar{F} : \bar{I} \longrightarrow U-\mathfrak{Cat}$ such that there is an equivalence of categories
$$
Hom^{(1)}_{U-\mathfrak{Cat}^{2}}(J,\Pi_{\Gamma}(F)) \cong \Pi_{\Gamma}(\bar{F})
$$

given by the following.

\ref{PiDiag}.1. Suppose that
$$
e : J \longrightarrow U-\mathfrak{Cat}
$$

is the functor defined on objects by sending $j \in Ob(J)$ to the category of arrows over it,
$$
e : j \mapsto J_{/j} = \downarrow_{(J)}(id_{J},ob_{(J)}(j))
$$

and on arrows by composition, so that for any arrow $(j \xrightarrow{f} k) \in Arr(J)$, the functor
$$
e(f) : J_{/j} \longrightarrow J_{/k}
$$

is defined on objects, so that for any $g \in Ob(J_{/j})$, map $(g \mapsto f \circ g)$. The map on arrows is essentially the expected identity map, i.e. one sends a triangle $g_{2} \circ \phi = g_{1}$ to the triangle $f \circ g_{2} \circ \phi = f \circ g_{1}$.

\ref{PiDiag}.2. Define a function 
$$
p : Hom_{U'-\mathfrak{Cat}}(J,\Pi_{\Gamma}(F)) \rightarrow Hom_{U'-\mathfrak{Cat}}(J,U-\mathfrak{Cat})
$$ 

by sending a functor $X : J \longrightarrow \Pi_{\Gamma}(F)$ to the functor $p(X) : J \longrightarrow U-\mathfrak{Cat}$ defined by sending an arrow $X(f)$ to the functor $F(\tilde{f})$ associated to it, i.e. it is defined by requiring that, for any $f \in Arr(J)$, and for any $\tilde{f} \in Arr(I)$,
$$
p(X) : f \mapsto F(\tilde{f})
$$

iff there exists $\phi \in Arr(codom(F(f)))$, such that
$$
X(f) = [(Hom_{codom(F(\tilde{f}))} \circ (F(\tilde{f})^{opp} \times_{U-\mathfrak{Cat}} Id_{codom(F(\tilde{f}))}) , \phi )].
$$

Define a function
$$
p_{1} : Arr(Hom^{(1)}_{U'-\mathfrak{Cat}^{2}}(J,\Pi_{\Gamma}(F))) \rightarrow Arr(Hom^{(1)}_{U'-\mathfrak{Cat}^{2}}(J,U-\mathfrak{Cat}))
$$

similarly, by sending a natural transformation of $\mathfrak{Cat}$-valued functors $t : X \xrightarrow{t} Y$ to the tuple of functors associated to the given tuple of arrows in $\Omega$, i.e. for any $j \in Ob(J)$, and for any $\tilde{f} \in I$, require that $p_{1}(t)(j) = F(\tilde{f})$ iff there exists $\phi \in Arr(codom(F(f)))$, such that
$$
t(j) = [(Hom_{codom(F(\tilde{f}))} \circ (F(\tilde{f})^{opp} \times_{U-\mathfrak{Cat}} Id_{codom(F(tilde{f}))}) , \phi )].
$$

\ref{PiDiag}.3. Define the functor
$$
\bar{F} : \bar{I} := Hom^{(1)}_{U-\mathfrak{Cat}^{2}}(J,\Pi_{\Gamma}(F)) \longrightarrow U-\mathfrak{Cat}
$$

on objects by sending a functor $X : J \longrightarrow \Pi_{\Gamma}(F)$ to the category 
$$
\bar{h}_{WE(U-\mathfrak{Cat},\times_{U-\mathfrak{Cat}})} (L(J),U-\mathfrak{Cat})(e,p(X))
$$

where $L(J)$ is the category $J$ with the trivial $U-\mathfrak{Cat}$ enrichment, where a given hom category is the category whose objects are given by the hom set in $J$, with only identity arrows. Define $\bar{F}$ on arrows by sending a natural transformation $(X \xrightarrow{t} Y) \in Arr(Hom^{(1)}_{U'-\mathfrak{Cat}^{2}}(J,\Pi_{\Gamma}(F)))$ to the functor given by the enriched composition of the diagram category, i.e.
$$
\bar{h}_{WE(U'-\mathfrak{Cat},\times_{U'-\mathfrak{Cat}})} (L(J),U-\mathfrak{Cat})(e,p(X))
$$
$$
\xrightarrow{\rho} \circ \xrightarrow{Id \times_{U'-\mathfrak{Cat}} p_{1}(t)} \circ \xrightarrow{\bar{\circ}_{WE(U-\mathfrak{Cat},\times_{U-\mathfrak{Cat}})}} 
$$
$$
\bar{h}_{WE(U'-\mathfrak{Cat},\times_{U'-\mathfrak{Cat}})} (L(J),U-\mathfrak{Cat})(e,p(Y)),
$$

where $\rho$ is the right unit isomorphism $Id \rightarrow Id \times_{U'-\mathfrak{Cat}} I'$, the functor $Id \times_{U'-\mathfrak{Cat}} p_{1}(t)$ is given by the representation by the unit $I'$ of the objects functor, and $\bar{\circ}_{WE(U-\mathfrak{Cat},\times_{U-\mathfrak{Cat}})}$ is the enriched composition.

\sss{Conjecture on Enrichments}
For any functor $F : I \longrightarrow (A,\otimes)-\mathfrak{Cat}$ and any tensor functor $(For,\rho) : (A,\otimes) \longrightarrow (U-\mathfrak{Set},\times_{U-\mathfrak{Set}})$, the category $\Pi_{\Gamma}(For^{(For,\rho)-\mathfrak{Cat}}_{U-\mathfrak{Cat}}(F))$ carries a natural enrichment over $(A,\otimes)$.

\sss{Remark}
Whether the functor $\kappa_{C}$ is left or right exact in its image seems to depend upon the compatibility of automorphisms of the identity over various objects in $C$. In particular, given a functor $F : I \longrightarrow C$, with a limit $(l,\lambda) \in Ob(C) \times Arr(Hom_{U'-\mathfrak{Cat}}^{(1)}(Hom_{U-\mathfrak{Cat}^{2}}^{(1)}(I,C)^{opp},U-\mathfrak{Set}))$, one might associate to any arrow of functors $\phi : \Delta_{(I,Im(\kappa_{C}))}(a) \rightarrow \kappa_{C} \circ F$ an arrow of functors $\tilde{\phi} : \Delta_{(I,C)}(a) \rightarrow \kappa_{C}$ by an explicit choice of arrows in $C$ representing the arrows appearing in $\phi$, hoping that the image by $\kappa_{C}$ of the limit arrow $\lambda(\bar{\phi})$ should be the limit arrow associated to $\phi$. But such a choice seems difficult, since each arrow in $I$ could a priori be associated to its own automorphism $Id_{C} \rightarrow Id_{C}$, whereas one would like to associate to each object in $I$ such an automorphism. Furthermore, once the candidate limit arrow is chosen, uniqueness seems to rely upon the idea that an assignment of objects in $I$ to automorphisms of $Id_{C}$, such as should relate the component arrows, would determine an automorphism of $Id_{C}$ suitable for relating the limit arrows.

\sss{$\bold{Definition}$ of $\Pi_{Sch}$} \label{PiSch}
For any given spec datum $\bar{sp} = (\mathcal{T} \xleftarrow{sp}  \mathcal{R} \xrightarrow{\mathcal{O}} \mathcal{S} , \tau , \varepsilon  )$, for any $\Omega-lift(U,\bar{sp} )$, $\tilde{sp} : \mathcal{R} \longrightarrow \Omega$, let $F : \mathcal{T} \longrightarrow U-\mathfrak{Cat}$ be as in (\ref{ExPiGTop}), and let 
$$
\pi : \Delta_{(\mathcal{T},U-\mathfrak{Cat})}(\mathcal{T}) \times F \rightarrow \Delta_{(\mathcal{T},U-\mathfrak{Cat})}(\mathcal{T})
$$

be the projection. Define a subcategory $\Pi_{Sch}(\bar{sp}, \tilde{sp}) \subseteq  \Pi_{\Gamma}(\Delta_{(\mathcal{T},U-\mathfrak{Cat})}(\mathcal{T}) \times F) \subset \Omega$ so that its arrows are those arrows $(x \xrightarrow{\phi} y) \in Arr(\Pi_{\Gamma}(\Delta_{(\mathcal{T},U-\mathfrak{Cat})}(\mathcal{T})\times F))$, such that subsets $C_{x}, C_{y} \subset Arr(\Pi_{\Gamma}(\Delta_{(\mathcal{T},U-\mathfrak{Cat})}(\mathcal{T}) \times F))$ exist such that the following hold.

\ref{PiSch}.1. The image of $C_{x}$ under $\gamma_{\pi(x)}$ is a cover of $\gamma_{\pi(x)}(x)$, and the image of $C_{y}$ under $\gamma_{\pi(y)}$ is a cover of $\gamma_{\pi(y)}$, in $(Im(\kappa_{\mathcal{T}}),\tau_{Im(\kappa_{\mathcal{T}})})$, where $\tau_{Im(\kappa_{\mathcal{T}})}$ is the topology generated by all the images of all elements of $\tau$ (i.e. covers in $(\mathcal{T},\tau)$).

\ref{PiSch}.2. Each of the arrows in $C_{x}$, when composed with $\phi$, yields an arrow factoring through an omega lift of some arrow in $\mathcal{R}$, i.e., for any $u \in C_{x}$, there exists $\phi ' \in Arr(\mathcal{R})$, such that there exists $v \in C_{y}$, such that $v \circ \tilde{sp}(\phi ') = \phi \circ u$.

\begin{rem}
The functor $^{opp} : \mathfrak{Cat} \longrightarrow \mathfrak{Cat}$ is its own adjoint, implying a fortiori that it preserves limits, and therefore that $(\mathcal{T} \times Sh( (dom(\varepsilon)(sp(y)) , (\varepsilon (sp(y)))^{*} \circ$

$(dob\downarrow_{(\mathcal{T})} ( id_{\mathcal{T}} , ob_{(\mathcal{T})}(sp(y)) ) )^{*} (\tau) ) , \mathcal{R} ) )^{opp} \cong \mathcal{T}^{opp} \times ( Sh( (dom(\varepsilon)(sp(y)) ,$

$(\varepsilon (sp(y)))^{*} \circ (dob\downarrow_{(\mathcal{T})} ( id_{\mathcal{T}} , ob_{(\mathcal{T})}(sp(y)) ) )^{*} (\tau) ) , \mathcal{R} ) )^{opp}$ . Thus isomorphisms of functors which identify representations of arrows in $\Omega$ might be thought of in a piecewise fashion.
\end{rem}

\sss{Proposition on Finite Diagrams} \label{PropFD}
For any spec datum $\bar{sp} = ...$ with an $\Omega$-lift $\tilde{sp}$, for any category $I \in Ob(U-\mathfrak{Cat})$ such that the set $Arr(I)$ is finite, there exists a spec datum 
$$
\bar{sp}' = ( \circ \xleftarrow{Hom_{U-\mathfrak{Cat}^{2}}^{(1)}(Id_{I},\mathcal{O})} Hom^{(1)}_{U-\mathfrak{Cat}^{2}}(I,\mathcal{R}) \xrightarrow{Hom_{U-\mathfrak{Cat}^{2}}^{(1)}(Id_{I},sp)} \circ , ... )
$$

with an $\Omega$-lift $\tilde{sp}'$, such that $Hom_{U'-\mathfrak{Cat}^{2}}^{(1)}(I,\Pi_{Sch}(\bar{sp},\tilde{sp})) \cong \Pi_{Sch}(\bar{sp}',\tilde{sp}')$.

\begin{lem}
For any ring $k$, $Aut_{End^{(1)}_{(U-\mathfrak{Cat}^{2})} ( k-\mathfrak{Alg} )} ( Id_{k-\mathfrak{Alg}} ) = \{ id_{Id_{k-\mathfrak{Alg} }} \}$
\end{lem}
\begin{proof}
If $k = \{ 0 \}$, then the category of $k$-algebras is the category with one arrow, and the result is trivial. If $k$ has at least two elements, then consider the free ring in one variable $k[x]$. But if there were a non-trivial automorphism of $Id_{k-\mathfrak{Alg}}$ as in the preceding lemma, then there would exist $A \in Ob(k-\mathfrak{Alg})$ such that the arrow $\alpha(A) : A \rightarrow A$ associated to $A$ would be other than the identity. Therefore there would exist $a \in A$ such that $\alpha (A)(a) \neq a$, implying that, if $f : k[x] \rightarrow A$ were the arrow given by $x \mapsto a$, $\alpha(k[x]) = id_{k[x]}$ would imply that $a = f (x) = f \circ \alpha(k[x])(x) = \alpha(A) \circ f (x) = \alpha(A) (a) \neq a$, a contradiction. Therefore $\alpha(k[x]) \neq id_{k[x]}$. Therefore the arrow assigned by the natural transformation to $k[x]$ would be other than the identity, and by naturality, it would have to commute with every other arrow $k[x] \mapsto k[x]$. The only automorphisms are $x \mapsto ax + b$, composed with automorphisms of $k$. But if $k$ has at least two elements, then a map $x \mapsto cx + d$ can be found which should make the required commuting square impossible.
\end{proof}

\begin{rem}
We've a more general notion of a ``free object" $x$ over an object $k$, with respect to a ``forgetful" functor $F$, being an object $x$ such that $Yo^{opp}(x) \cong Yo^{opp}(k) \times F$. We hope that that the notion of freedom, and the previous lemma, might be generalized and made precise, so far as possible, the inducement being thoughts of expressing ``finite type" in a recursive fashion, and examining the relationship between it (an ``algebraic" notion) and compactness, a property referring only to the topology. $codom(F)$ seems somewhat flexible, assuming that the category in question is enriched over $codom(F)$, so that the functor category should inherit the enrichment.
\end{rem}

\begin{cor}
For any ring $k$, $\kappa_{k - \mathfrak{Alg} }$ is faithful.
\end{cor}

\sss{Corollary}\label{PiSt}
If $\bar{spec} = ( \mathfrak{Top} \xleftarrow{spec} \mathfrak{Ring} \xrightarrow{Id_{\mathfrak{Ring}}} \mathfrak{Ring} ,  )$ is the usual topological spectrum data and $\tilde{spec}$ is its $\Omega$-lift, then the following hold.

\ref{PiSt}.1. The functor $\tilde{spec}$ is faithful.

\ref{PiSt}.2. The left Kan extension of the usual inclusion functor $\mathfrak{Ring} \longrightarrow \mathfrak{Sch}$ along the natural factor of the $\Omega$-lift $\tilde{spec}' : \mathfrak{Ring} \longrightarrow \Pi(\bar{spec},\tilde{spec})$ through its image (i.e. the functor such that $\varepsilon_{\Pi(\bar{spec},\tilde{spec})} \circ \tilde{spec}' = \tilde{spec}$), is an equivalence of categories $\Pi(\bar{spec},\tilde{spec}) \longrightarrow \mathfrak{Sch}$.

\begin{proof}
In this case $\mathcal{O} = Id_{U-k-\mathfrak{Alg}}$, so that the global sections functor exists, and by the usual arguments $\Gamma \circ \tilde{spec} \cong Id_{k-\mathfrak{Alg}}$. An arrow of schemes is uniquely determined by its open affine fibres.
\end{proof}



\sus{Separateness and Extensions of $\Pi(sp,\mathcal{O},...)$}

We interpret separateness as ``the sufficiency of the sheaf category component of $\tilde{x} = (\mathcal{T} \times Sh(X),(x,\mathcal{O}_{x}))$ to determine points," by the following conjecture.

\sss{$\bold{Definition}$ of Separateness}
For any $sk$-natural transformation of functors $F,G : I \longrightarrow \Omega$, $t : F \rightarrow G$, for any $x \in Ob(\Pi_{\Gamma}(F))$, we say that $x$ is $separated(t)$ iff for any $y \in Ob(\Pi_{\Gamma}(F))$, the function $\gamma_{t}(y,x) : Hom_{\Pi_{\Gamma}(F)}(y,x) \rightarrow Hom_{\Pi_{\Gamma(G)}}(\gamma_{t}(y),\gamma_{t}(x))$ is injective.


%

\sss{Remark} By (\ref{PiDiag}) and (\ref{PropFD}) there are at least two ways in which a category of diagrams in some $\Pi_{Sch}(\bar{sp},\tilde{sp})$ can be realized as a subcategory of some $\Pi_{\Gamma}(\bar{F})$. The former is intrinsic to $\Pi_{\Gamma}$, and essentially involves tuples of functors $e(j) \longrightarrow X(j)$ (in this case $F$ of that construction would be the product of the functor in (\ref{ExPiGTop}) with $Id_{\mathcal{T}}$). The latter essentially restricts the spec functor to appropriate diagrams in $\mathcal{R}$, and uses the product of the functor of (\ref{ExPiGTop}) with $Id_{Hom_{U-\mathfrak{Cat}^{2}}^{(1)}(I,\mathcal{T})}$. In either case, I expect that the natural transformation determining separateness on the diagram category in $\Pi_{Sch}(\bar{sp},\tilde{sp})$ should be given by the ``projection which forgets the constant $\mathcal{T}$ component."

\sss{Conjecture on Separateness}
Adopt the conditions and notation of (\ref{ExPiGTop}). Let
$$
\pi : \Delta_{(\mathcal{T},U-\mathfrak{Cat})}(\mathcal{T}) \times F \rightarrow F
$$

be the projection. For any $x \in Ob(\Pi(\bar{sp},\tilde{sp}))$, $x$ is $separated(\bar{sp},\tilde{sp})$ iff for any $y \in Ob(\Pi(\bar{sp},\tilde{sp}))$, $\gamma_{\pi}(y,x) \text{ is injective}$.

Suppose that $( X, X \xrightarrow{f} S, \emptyset ) \in Ob(\mathfrak{Sch})$ is a scheme over $Y$. Then the notions of separateness coincide, i.e. $X \xrightarrow{\Delta} X \times_{S} X \text{ is closed, an immersion iff } X \text{ is separated}(\bar{sp},\tilde{sp})$.

\sss{Remarks on Correspondences}

If $\tilde{sp}$ is an $\Omega$-lift of $\bar{sp}$, $S$ is separated($\bar{sp},\tilde{sp}$), and if $COLIM$ denotes a functor $Hom_{U-\mathfrak{Cat}^{2}}^{(1)}(I,A) \longrightarrow A$ which sends a functor to its colimit (in this case, $I = A$), and $\bar{Yo_{opp}}$ denotes a  $\mathcal{O}_{X \times_{S} Y}-\mathfrak{Mod}$-enriched hom functor (i.e. for any $\mathcal{M} \in Ob(\mathcal{O}_{X \times_{S} Y}-\mathfrak{Mod})$, $\bar{Yo_{opp}}(\mathcal{M}) : a \mapsto Hom^{\sharp}(M,a)$, so that $Yo = (A \xrightarrow{\bar{Yo}} Hom_{U-\mathfrak{Cat}^{2}}^{(1)}(A,A) \xrightarrow{Hom_{U-\mathfrak{Cat}^{2}(1)}^{(1)}(id_{A},For) } \circ$), then it is expected that the set of arrows of the form

$$
[ ( For^{COLIM(\mathcal{O}_{X \times_{S} Y})-\mathfrak{Mod}}_{U-\mathfrak{Set}} \circ COLIM \circ COLIM \circ \bar{Yo}
$$
$$
\circ For^{Sh((X \times_{S} Y , \tau_{X \times_{S} Y} ),\mathfrak{Ring})}_{\mathcal{O}_{X \times_{S} Y}-\mathfrak{Mod}} \circ \otimes_{Sh((X \times_{S} Y , \tau_{X \times_{S} Y} ),\mathfrak{Ring})} \circ ((p^{*})^{opp} \times_{\mathfrak{Cat}} q^{*} ) , \phi )] \in Arr(\Omega),
$$

where $(p,q) fibres_{(\mathfrak{Sch})} ( X \xrightarrow{f} S, Y \xrightarrow{g} S )$, should be isomorphic to the set of correspondences from $X$ to $Y$. \ftt{
The forgetful functor, from sheaves of rings over (in opp) the pullback of the base structure sheaf to modules over the structure sheaf of the fibre, seems suspect. I believe that correspondences are usually in the literature restricted to smooth schemes over a field (and projective, I imagine so that $X \times_{S} Y \times_{S} Z \xrightarrow{\pi} X \times_{S} Z$ should be closed for the sake of composition, but I debate whether this is necessary in the above formulation), and I wonder therefore whether one should replace $For$ with $Diff : Sh((X \times_{S} Y , \tau_{X \times_{S} Y} ),\mathfrak{Ring})_{ \longrightarrow \mathcal{O}_{X \times_{S} Y}-\mathfrak{Mod}}$.
}

We also expect that restricting the ``sub-object diagram" by replacing $COLIM \circ \bar{Yo}$ with $COLIM \circ Hom_{U-\mathfrak{Cat}^{2}}^{(1)}(\varepsilon_{i},id_{\mathcal{O}_{X \times_{S} Y}-\mathfrak{Mod}}) \circ \bar{Yo}$ in the above arrow in $\Omega$, where $\varepsilon_{i} \hookrightarrow \mathcal{O}_{X \times_{S} Y}-\mathfrak{Mod}$ is the inclusion of modules with an epic map $\bigoplus_{j=1}^{i} \mathcal{O}_{X \times_{S} Y} \longrightarrow \mathcal{M}$, yields the set of correspondences given by subschemes $Z \hookrightarrow X \times_{S} Y$ of codimension $i$.

However they would appear here as arrows $(Sh(X)^{opp},\mathcal{O}_{X}) \rightarrow (Sh(Y),\mathcal{O}_{Y})$ in $\Omega$, rather than having $(Sh(X),\mathcal{O}_{X})$ as the domain, since the tensor functor $\otimes$ is covariant in each variable.

\sus{Addenda}

\begin{rem}
For any category $I$, consider a limit $(i,l)$ of some diagram $D : J \longrightarrow I$, i.e. $i \in Ob(I)$ and $l: \Delta_{(J)}(i) \rightarrow D$ is a universal arrow of functors, where $\Delta_{(J)} : C \longrightarrow Hom^{(1)}_{U-\mathfrak{Cat}^{2}}(J,C)$ sends an object to its constant functor, so  that any $f : \Delta_{(J)}(i') \rightarrow D$ factors uniquely through some $\Delta_{(J)}(\tilde{f}) : \Delta_{(J)}(i') \rightarrow \Delta_{(J)}(i)$. Then for any arrow $[(F,\phi)] \in Arr(\Omega)$, the distinguished element of the set, $\phi \in F(c,i)$, determines compatible elements $F(id_{c},l(j))(\phi) \in F(c,D(j))$, where $j \in Ob(J)$. This is to say that one can also distinguish collections of elements in various sets by using a limit object $i \in Ob(I)$ for the distinguished element in the object $(I,i)$ which determines a representable functor in $\Omega$.
\end{rem}

\begin{rem}
I expect that the representatives of many sorts of functors $\Omega \supseteq \Pi \longrightarrow \mathcal{S}$, in $\Pi$ may be sought by the colimits of the relevant $dob\downarrow_{(\Omega)}$ functors from arrow categories in $\Omega$ over $(I,i)$, i.e. by constructing right adjoints to inclusion functors $\Pi \hookrightarrow \bar{\Pi}$, where $\bar{\Pi}$ is $\Pi$ with a single extra object.
\end{rem}

\se{(Weak) $U-\mathfrak{Cat}$-valued sheaves and Steps Toward Homotopy Groups, and (Co)-Homology}

\sss{
Constructions of admissibility structures
}
In (\ref{SusAugAdm}), we develop two methods by which one can construct new admissibility structures from old ones.

In the first construction (see \ref{LemAdmGlu}), we work with a given pair of admissibility structures $\varepsilon_{i} : E_{i} \hookrightarrow Fib_{i}$ attached to a pair of categories $\mathcal{T}_{i}$, for $i \in \{ 1,.2 \}$. Given a functor $\Lambda : \mathcal{T}_{1} \longrightarrow \mathcal{T}_{2}$, between such categories, we use the $\kappa$-cotwists (see \ref{kappatwist}) to define $\mathfrak{Cat}$-valued functors on the image of $\mathcal{T}_{1}$ in $\Omega$ under the canonical functor $\kappa_{\mathcal{T}_{1}}$ (\ref{Can}). 

Such a functor would send an object $x \in \mathcal{T}_{1}$ to the category of arrows over $(\mathcal{T}_{1},x) \in Ob(\Omega)$ of the form $u \circ \kappa_{cotw,\Lambda}(dom(u)) \circ v$, where $u$ and $v$ correspond to objects in $E_{1}(x)$ and $E_{2}(\Lambda (x))$ respectively. When the canonical functors $\kappa_{\mathcal{T}_{i}}$ are faithful, this functor determines an admissibility structure on the subcategory $\bar{\mathcal{T}}$ of $\Omega$ generated by the images of $\kappa_{\mathcal{T}_{i}}$ and the $\kappa$-cotwists $\kappa_{cotw,\Lambda}$. 

In the second construction (see \ref{LemAdmPath}), we consider a category $\mathcal{T}$ with an admissibility structure $\varepsilon : E \hookrightarrow Fib$. In a fashion similar to that by which admissibility strictures $\varepsilon_{i}$ are glued, we use the $\kappa$-twist of the functor $P : Path(For^{\mathfrak{Cat}}_{\mathfrak{PreCat}}(\mathcal{T})) \longrightarrow \mathcal{T}$ (which comes from the adjunction, sending a string of arrows to their composition) to define a $\mathfrak{Cat}$-valued functor, $\mathcal{T} \longrightarrow \mathfrak{Cat}$.  We associate to $x \in Ob(\mathcal{T})$ the set of arrows of the form $[(Hom_{\mathcal{T}} , u)] \circ \kappa_{tw,P} \circ [(Hom_{Path(\mathcal{T})},v)] \in Arr(\Omega)$, where $u \in Ob(E(x))$ and $v$ is a string of arrows corresponding to objects in $E(x)$. 

\subsubsection{
Motivation
}
Recall that in (\ref{ExPiGTop}) we have associated to any category $\mathcal{T}$ with an admissibility structure and Grothendieck topology the $\mathfrak{Cat}$-valued functor $F : \mathcal{T} \longrightarrow \mathfrak{Cat}$. The category of schemes $\Pi_{(\bar{sp},\tilde{sp})} \subseteq \Pi_{\Gamma}(F) \subseteq \Omega$ was constructed in (3.2) as a sub-category of $\Pi_{\Gamma}(F)$ under the guiding principle that ``an object in a geometric category is a sheaf, i.e. a distinguished object within some category of sheaves." 

We expect, in future work, to encounter situations in which we should consider more general geometric objects, given by ``weak sheaves" $\mathcal{O}_{x} \in Ob(WeSh(x))$. The ``category of weak sheaves on $x$", $WeSh(x)$ is to be interpreted according to one's needs, or notion of ``weakness". We have in mind the example in which the objects are maps $Arr(E(x)) \rightarrow Arr(\mathfrak{Cat})$ which respect composition only up to natural isomorphism. This is to say, we would consider objects of the form $\tilde{x}' = (WeSh(x),\mathcal{O}_{x}) \in Ob(\Omega)$, rather than of the form $\tilde{x} = (Sh(x),\mathcal{O}_{x}) \in Ob(\Omega)$. 

We define the category of ``pre-categories", so that its objects are pairs $(S,h)$ such that $S$ is a set and $h$ is a function (``hom set") which assigns to each $s,t \in S$ a set $h(s,t)$ (there is no composition law). A pre-functor is a pair of maps $(F_{0} : S \rightarrow T,F_{1} : S^{2} \rightarrow Arr(\mathfrak{Set}))$ such that $F_{1}(a,b) : h_{S}(a,b) \rightarrow h_{T}(F_{0}(s),F_{0}(t))$. A pre-functor would contain the data for the assignment of individual arrows, while having no composition requirement. The ``path category" $Path(S)$ of a pre-category $S$ has for its objects the underlying set of $S$. The elements of its hom sets are strings of elements of hom sets $(f_{i})_{i \in \{ 1,...,n\}}$, where $f_{i} \in h(s_{i},s_{i+1})$, for $i \in \{ 1,...,n\}$, with composition given by concatenation. 

We loosely think of a ``weak functor" as a pre-functor with some further requirements, e.g. some associativity or identity condition. The data of a weak $C$-valued functor for some category $C$ would then be contained in the data of a functor $Path(For^{\mathfrak{Cat}}_{\mathfrak{PreCat}}(\mathcal{T})) \longrightarrow C$, by the adjunction $Hom_{\mathfrak{PreCat}}(-,For^{\mathfrak{Cat}}_{\mathfrak{PreCat}}(-)) \cong Hom_{\mathfrak{Cat}}(Path(-),-)$.  We think of this as the extension of the admissibility structure on $\mathcal{T}$ to one which assigns to an object $x \in \mathcal{T}$ a category $E'(x)$ such that the set $Hom^{(1)}_{U-\mathfrak{Cat}^{2}}(E'(x),C)$ of $C$-valued functors is isomorphic to the set of weakened functors $E(x) \longrightarrow C$, whose domains would be given by the original admissibility structure.

We intend, in future work, to derive at least ``the" homotopy groups from admissibility structures, by imitating the description of $\pi_{1}$ in terms of covering spaces and adopting the notion that $\pi_{n+1}(x) \cong \pi_{n}(\Lambda(x))$ for some endofunctor $\Lambda : \Pi \longrightarrow \Pi$. It is intended that (\ref{assncat}) should be applied to such admissibility structures as appear in the lemma (\ref{LemAdmGlu}) below, them being in the place of the $\varepsilon(\beta)$ of (\ref{assncat}.2), to define functors $\Pi \longrightarrow n-\mathfrak{Cat}$, so that ``the $n$-categery attached to an object should contain its homotopy data." To this end we define, in (\ref{InvCat}), a certain functor $\mathfrak{Cat} \longrightarrow \mathfrak{PreGroup}$.

\sus{Augmented Admissibility Structures}\label{SusAugAdm}

Given a functor $\Lambda : \mathcal{T}_{1} \longrightarrow \mathcal{T}_{2}$ between categories with admissibility structures $\varepsilon_{i} : E_{i} \hookrightarrow Fib$, we define in (\ref{LemAdmGlu}) a functor $\bar{\mathcal{T}}^{opp} \longrightarrow \mathfrak{Cat}$, where $\bar{\mathcal{T}} \subseteq \Omega$ is generated from arrows in the images of $\kappa_{\mathcal{T}_{i}}$ and arrows of the form $\kappa_{cotw,\Lambda}(x)$. If the canonical functors $\kappa_{\mathcal{T}_{i}}$ are faithful, then this functor is an admissibility structure.

It is expected that one should desire to consider $\mathfrak{Cat}$-valued pseudo-functors attached to a space, such that composition might be respected only up to isomorphism. We use the $\kappa$-twists attached to the concatenation functors $Path(\mathcal{T}) \longrightarrow \mathcal{T}$ to construct, from an admissibility structure on $\mathcal{T}$, an $\mathfrak{Cat}$-valued functor on a category $\bar{\mathcal{T}}^{opp}$, where $\mathcal{T} \subseteq \bar{\mathcal{T}} \subseteq \Omega$. If $\kappa_{\mathcal{T}}$ is faithful, then the restriction of this functor to $\mathcal{T}^{opp}$ is a subfunctor of a fibre functor, so that the categories of functors with domains given the new ``admissibility" should be equivalent to categories of pseudo-functors with domains given by the old admissibility.

\sss{$\bold{Definition}$ of the Category of Pre-Categories} \label{DefPreCat}

Define the category of pre-categories to be the category of sets with hom set assignments, but no composition structure. Explanation follows.

\ref{DefPreCat}.1. For any universe $U$, define the category $U-\mathfrak{PreCat} \in Ob(U'-\mathfrak{Cat})$ to be the category of $U$-small pre-categories. Its set of objects is the set of triples $(O,A,h)$, where $O$ and $A$ are sets and $h : O^{2} \rightarrow 2^{A}$ assigns to every pair of ``objects" a ``hom set," so that distinct objects have disjoint hom sets.
$$
\{ (O,A,h) \in U^{3} ; \ulcorner\ulcorner h \in Hom_{(U-\mathfrak{Set})}(O^{2},2^{A}) \urcorner \text{ and } 
$$
$$
\ulcorner \forall a,b,c,d \in O, \ulcorner\ulcorner (a,b) \neq (c,d) \urcorner \Longrightarrow \ulcorner h((a,b)) \cap \h((c,d)) = \emptyset \urcorner\urcorner\urcorner\urcorner \}
$$

Its set of arrows is the set of ``pre-functors," pairs $F = (F_{0},F_{1})$ between the object sets and the arrow sets which are compatible with the hom sets, i.e. for any two pre-categories $(O_{C},A_{C},h_{C}),(O_{D},A_{D},h_{D}) \in Ob(U-\mathfrak{PreCat})$, define
$$
Hom_{U-\mathfrak{PreCat}}((O_{C},A_{C},h_{C}),(O_{D},A_{D},h_{D})) := 
$$
$$
\{ (F_{0},F_{1}) \in U; F_{0} : O_{C} \rightarrow O_{D} \text{ and } F_{1} : A_{C} \rightarrow A_{D} \text{ and }
$$
$$
\forall c_{1},c_{2} \in O_{C}, Im(F_{1} |_{h_{C}(c_{1},c_{2})}) \subseteq h_{C}(F_{0}(c_{1}),F_{0}(c_{2})) \}
$$

\ref{DefPreCat}.2. Define the functor $For^{U-\mathfrak{Cat}}_{U-\mathfrak{PreCat}} : U-\mathfrak{Cat} \longrightarrow U-\mathfrak{PreCat}$ to be the expected forgetful functor, with the arrow map given by the identity.

\sss{$\bold{Definition}$ of the Functor ``$Path$"}
Define the functor $Path : U-\mathfrak{Cat} \longrightarrow U-\mathfrak{Cat}$ as follows.

For any category $C \in Ob(U-\mathfrak{Cat})$, define the category $Path(C)$ to be the path category of $C$, so that its objects $Ob(Path(C)) = Ob(C)$, are the same, and arrows are composible sequences of arrows in $C$, 
$$
Hom_{Path(C)}(x,y) = \{ \emptyset \} \cup \coprod_{n \in \mathbb{N}^{+}} \coprod_{z : \{ 1,...,n \} \rightarrow Ob(C) ; z(1) = x \text{ and } z(n) = y} \prod_{i=1}^{n-1} Hom_{C}(z(i),z(i+1)),
$$

with composition given by concatenation.

For any functor $F : C \longrightarrow D$, define the functor $Path(F) : Path(C) \longrightarrow Path(D)$ so that for any objects $x,y \in Ob(C)$,
$$
Path(F)|_{Hom_{Path(C)}(x,y)} = id_{\{ \emptyset \}} \cup \coprod_{n \in \mathbb{N}^{+}} \coprod_{z : \{ 1,...,n \} \rightarrow Ob(C) ; z(1) = x \text{ and } z(n) = y} \prod_{i=1}^{n-1} F_{Hom_{C}(z(i),z(i+1))}
$$

\begin{lem}
$For^{U-\mathfrak{Cat}}_{U-\mathfrak{PreCat}}$ has a left adjoint $L$, which when composed with the forgetful functor yields $Path \cong L \circ For^{U-\mathfrak{Cat}}_{U-\mathfrak{PreCat}}$.
\end{lem}

\begin{lem}
For any category $C \in Ob(U-\mathfrak{Cat}), \kappa_{Path(C)} \text{ is faithful}$.
\end{lem}

\sss{Remark on a Generalization}
There should be a version of this for $WE_{(A,\otimes)}$ for any associative tensor category $((A,\otimes ),\alpha,\sigma)$, which satisfies a pentagonal requirement for $\alpha$ and $\sigma$, i.e. that the associator is canonical for larger tensors.

\begin{pro}
For any $C \in Ob(U'-\mathfrak{Cat})$, $\kappa_{Hom_{U'-\mathfrak{Cat}^{2}}^{(1)}(Path(C),U-\mathfrak{Cat})}$ is faithful.
\end{pro}

\begin{lem}\label{LemAdmPath}
Suppose that $T \in Ob(U-\mathfrak{Cat})$ is a category, such that $\kappa_{T}$ is faithful. If $P : Path_{(0)}(T) \longrightarrow T$ is the concatenation functor (universal arrow from adjunction), then for any admissibility structure $\varepsilon$ on any $T \in Ob(U-\mathfrak{Cat})$, there is a unique admissibility structure $\varepsilon'$ on the subcategory of $\Omega$ generated by the image of $\kappa_{T}$ such that for any $x \in Ob(T)$, the category $dom(\varepsilon')(x) \subseteq \Omega_{/\kappa_{T}(x)}$ is that whose objects are the set
$$
\{ [(Hom_{T} \circ (P^{opp} \times_{U-\mathfrak{Cat}} Id_{T}), u )] \in Arr(\Omega) ; u \in Im(\varepsilon(x)) \} \cup
$$
$$
\{ [(Hom_{T},u )] \in Arr(\Omega); u \in Im(\varepsilon(x) \}
$$

and whose arrows are the set
$$
\{ [(Hom_{Path(T)} , u )] \in Arr(\Omega) ; P(u) \in Im(\varepsilon(x)) \} \cup
$$
$$
\{ [(Hom_{T} \circ (P^{opp} \times_{U-\mathfrak{Cat}} Id_{T}), u )] \in Arr(\Omega) ; u \in Im(\varepsilon(x)) \} \cup
$$
$$
\{ [(Hom_{T},u )] \in Arr(\Omega); u \in Im(\varepsilon(x) \}
$$
\end{lem}

\sss{Corollary on Iterated ``$Path$ Augmentation"}
Suppose that $\Pi \in Ob(U-\mathfrak{Cat})$ has an admissibility structure $\varepsilon_{0}$. Then there is an admissibility structure $\varepsilon$ on $Im(\kappa_{\Omega(\Pi)}) \subseteq \Omega$ such that for every $n \in \mathbb{N} \cup \{ \infty \}$,  there exists $\varepsilon' \subseteq \varepsilon$, such that for any $x \in Im(\kappa_{\Omega(\Pi)}), dom(\varepsilon')(x) \cong Path^{n}(dom(\varepsilon_{0})(x))$


\begin{rem}
In light of the fact that $\kappa_{Path(C)}$ is faithful, this domain for $\mathfrak{Cat}$-valued functors is naturally attached in $\Omega$ by arrows to a domain for $\mathfrak{Cat}$-valued pseudo-functors. In particular, the set of objects of the category of functors $Hom_{U'-\mathfrak{Cat}^{2}}^{(1)}(Path(C),U-\mathfrak{Cat})$ is by the adjunction isomorphic to the set $Hom_{U'-\mathfrak{PreCat}}(For^{U-\mathfrak{Cat}}_{U-\mathfrak{PreCat}}(C),For^{U'-\mathfrak{Cat}}_{U'-\mathfrak{PreCat}}(U-\mathfrak{Cat}))$ of pseudo-functors . 
\end{rem}

\begin{rem}
For any spec datum $(sp : R \rightarrow T,\mathcal{O},\tau,\varepsilon )$, if $\kappa_{\Omega(codom_{(U-\mathfrak{Cat})}(sp))}$ were faithful then it would be an isomorphism onto its image, and the latter could be given structures so that $(\kappa_{\Omega (T)} \circ sp, \mathcal{O},\tau',\varepsilon' )$ should be a spec datum as well.
\end{rem}

\begin{rem}
Since the subcategory denoted by $\Pi$ associated to $\Omega$-lifts has objects of the form $(\mathcal{T} \times_{U-\mathfrak{Cat}  (0)}(Sh((X,\tau_{X}),\mathcal{R}))^{opp} , (X,\mathcal{O}_{X}))$ the previous lemma can be used to extend the admissibility structure on $\Pi$ in a component-wise fashion. 
\end{rem}

\begin{lem}
Suppose that $F: I \longrightarrow U-\mathfrak{Cat}$ and $(l,\lambda)$ is a $(\mathfrak{Skel})-limit(F)$. 

Suppose that $c : l \longrightarrow \prod_{i \in Ob(I)} F(i)$ is the map induced by the definition of the $(sk)-limit$ as a colimit (as in lemma on $sk$-limits having unique maps into them when the arrow from the $(sk)$-limit to the product is monic) and that $c$ is monic (faithful). Then $\forall G \subseteq Ob(l), \ulcorner \forall i \in Ob(I),\ulcorner  \{ \pi_{i} \circ c(g) ; g \in G \} \text{ generates }F(i) \urcorner\urcorner \Longrightarrow \ulcorner G \text{ generates }l \urcorner$

 $x \in Ob(dom_{(U-\mathfrak{Cat})}(\lambda_{i}))$ from those of the components, such that $\forall i \in Ob(I), \lambda_{(i)}(x) \cong g_{i}$, $g_{i}$ being some generator of $F(i)$
\end{lem}

\begin{cor}
$\kappa_{\Omega(lim_{i\in \mathbb{N}} Path^{(i)}(C))}$ is faithful.
\end{cor}

\sss{Lemma on Gluing Admissibility Structures}\label{LemAdmGlu}
If categories $\Pi_{1},\Pi_{2} \in Ob(U-\mathfrak{Cat})$, such that $\kappa_{\Pi_{1}}$ and $\kappa_{\Pi_{2}}$ are faithful, have admissibility structures $\varepsilon_{1},\varepsilon_{2}$ respectively, then a functor $(\Pi_{1} \xrightarrow{\Lambda} \Pi_{2}) \in Arr(U-\mathfrak{Cat})$ determines an admissibility structure on the subcategory of $\Omega$ generated by arrows of the form $\kappa_{cotw(\Lambda)}(x)$, where $x \in Ob(\Lambda_{1})$, or contained in the image of either $\kappa_{\Pi_{1}}$ or $\kappa_{\Pi_{2}}$. Explanation follows.

\ref{LemAdmGlu}.1. Let $\Lambda : \Pi_{1} \longrightarrow \Pi_{2}$ be a functor between $U$-categories.

\ref{LemAdmGlu}.2. Suppose that $\kappa_{\Pi_{i}}$ is faithful, for $i \in \{1,2\}$.

\ref{LemAdmGlu}.3. Let $\varepsilon_{i}$ be an admissibility structure on $\Pi_{i}$, for $i \in \{1,2\}$.

\ref{LemAdmGlu}.4. Let $\Pi \subseteq \Omega$ be the subcategory generated by the union of the images of the $\kappa$-functors, with the set of arrows appearing in $\kappa$-cotwists, i.e. the set 
$$
Im(\kappa_{\Pi_{1} 1}) \cup Im(\kappa_{\Pi_{2} 1}) \cup \{ \kappa_{cotw(\Lambda)}(x) \in Arr(\Omega) ; x \in Ob(\Pi_{1}) \}
$$.

\ref{LemAdmGlu}.5. Then there is a unique admissibility structure $\varepsilon \in Arr(Hom_{U'-\mathfrak{Cat}^{2}}^{(1)}(\Pi^{opp},U-\mathfrak{Cat}))$ formed from gluing the arrows appearing in the original admissibility structures by the $\kappa$-twist arrows, i.e. such that the following hold.

\ref{LemAdmGlu}.5.1. For any object $x \in Ob(\Pi_{1})$, the objects of the category assigned to $(\Pi_{1},x) \in Ob(\Pi)$ are compositions of arrows appearing in the original admissibility structures, i.e.
$$
Ob(dom(\varepsilon)(x)) = \{ [(Hom_{\Pi_{2}} \circ (Id_{\Pi_{2}}^{opp} \times_{U-\mathfrak{Cat}} \Lambda) , u )] \circ [(Hom_{\Pi_{1}}, v )] \in Arr(\Omega) ;
$$
$$
u \in Im(\varepsilon_{1}(x)) \text{ and } v \in Im(\varepsilon_{2}(\Lambda(dom(u)))) \} \cup
$$
$$
\{ [(Hom_{T},u )] \in Arr(\Omega); u \in Im(\varepsilon_{1}(x) \},
$$

and the set of arrows of the category assigned to $(\Pi_{1},x) \in Ob(\Pi)$ is the pre-image of the original admissibility structures, i.e. arrows of the form $[(Hom_{\Pi_{2}},v)]$, where $v \in Im(\varepsilon_{2}(\Lambda(x)))$ or the form $[Hom_{\Pi_{1}},u)]$, where $u \in Im(\varepsilon_{1}(x))$.

\ref{LemAdmGlu}.5.2. For any object $x \in Ob(\Pi_{2})$, define the category assigned to the object $(\Pi_{2},x) \in Ob(\Omega)$ so as to agree with the original admissibility structure, i.e.
$$
dom(\varepsilon)((\Pi_{2},x)) \cong dom(\varepsilon_{2})(x)
$$

\sus{Remarks toward Invariants from Categories}\label{InvCat}

We define a functor $U-\mathfrak{Cat} \longrightarrow \mathfrak{PreGrp}$, intended to send the category of functors $Hom_{\mathfrak{Cat}^{2}}(E(x),E(x))$ to a pre-group (without inversion) containing the fundamental group $\pi_{1}(x)$.

\begin{rem}
If $dom_{(U'-\mathfrak{Cat})}(P)$ has $U$-small limits, then $P$ has a left adjoint and if it has $U$-small colimits, and $P$ preserves them, then it has a right adjoint. This is a general categorical construction. I think of the standard example of an admissibility structure, the open subsets in a topological space, when considering the feasibility of the existence of the right adjoint.
\end{rem}

\sss{$\bold{Definition}$ of Sweep}
Define the ``sweep functor" by sending a category $C$ to the set of equivalence classes of ordered lists of arrows in $C$, where the equivalence is that generated by allowing the compositions of consecutive pairs of arrows in the list. Composition is given by concatenation.


\begin{lem}
Sweep is functorial.
\end{lem}

\begin{rem}
We hope, in future work, to associate to each spec datum $\bar{sp} = (...,\varepsilon : E \hookrightarrow Fib)$ a $\mathfrak{Cat}$ valued functor, which would attach, to each $x \in Ob(\mathcal{T})$, the category of endofunctors of $E(x)$. The expectation is that any two $\Omega$-lifts of the same arrow in $\mathcal{R}$ should be related by automorphisms of the fibres used to determine the $\Omega$-lifts, which would determine an isomorphism in the endofunctor category.
\end{rem}


\se{Appendix}

\sss{An Older Definition of $n$-Categories} This construction is simpler than that which appears in the main text, but does not involve defining a sequence of skeleton functors $sk(n) : n-\mathfrak{Cat} \longrightarrow \cdot$ to weaken the required equalities (e.g. associativity). 

The following five definitions
$$
For^{U(n)-\mathfrak{Cat}}_{U-\mathfrak{Cat}} : U(n)-\mathfrak{Cat} \longrightarrow U-\mathfrak{Cat}
$$
$$
\alpha_{U(n)-\mathfrak{Cat}} Ob(U(n)-\mathfrak{Cat})^{3} \rightarrow Arr(U(n)-\mathfrak{Cat})
$$ 
$$
\times_{U(n+1)-\mathfrak{Cat}} : U(n+1)-\mathfrak{Cat} \times_{U'-\mathfrak{Cat}} U(n+1)-\mathfrak{Cat} \longrightarrow U(n+1)-\mathfrak{Cat}
$$
$$
\sigma_{U(n+1)-\mathfrak{Cat}} : Ob(U(n+1)-\mathfrak{Cat})^{2} \rightarrow Arr(U(n+1)-\mathfrak{Cat})
$$
$$
U(n)-\mathfrak{Cat} \in Ob(U'-\mathfrak{Cat})
$$ 

are simultaneously made, for any $n \in \mathbb{N} \backslash \{ 0 \}$,\ftt{The expected case of sets, for $n = 0$, is excluded, since a set does not seem naturally to be an enriched set $(S,h,\circ)$, unless it takes the enrichment over $\mathfrak{Set}$ given by the trivial category functor $\mathfrak{Set} \longrightarrow \mathfrak{Cat}$.} so that the category of (n+1)-categories (the fifth of these definitions) should be defined as the category of weakly ($n-\mathfrak{Cat}$)-enriched categories. \ftt{The definition applies the unary map of the comprehension schema, attached to the open statement described below, to the universe $U'$. It makes use of the absolute ``$=$" sign, a relation on the one type, to compare the sets $S_{i}$ below, a priori existing, though uncharacterized, to certain sets constructed from them.}

For any pair of universes $U \in U'$, for any tuple $S = (S_{1},S_{2},S_{3},S_{4},S_{5}) \in Ob(U'-\mathfrak{Set})$, we say that $S$ defines higher categories in $U \in U'$ iff





and 

$$
S_{1} = \{ (n , ( F_{1} = \{ ((C,h,\circ) , C ) \in U' ; \exists c \in Ob(U'-\mathfrak{Cat}), (C,h,\circ) \in Ob(c) \text{ and } (n,c) \in S_{5} \},
$$
$$
F_{2} = \{ ((f_{0},f_{1}),f_{2}) , (f_{0},f_{1})) \in U' ;
$$
$$
\exists c \in Ob(U'-\mathfrak{Cat}), ((f_{0},f_{1}),f_{2}) \in Arr(c) \text{ and } (n,c) \in S_{5} \} ) )
$$
$$
\in (\mathbb{N} \backslash \{ 0 \} ) \times U' \}
$$ 

For any $n \in \mathbb{N} \backslash \{0\}$, $For^{U(n)-\mathfrak{Cat}}_{U-\mathfrak{Cat}} : U(n)-\mathfrak{Cat} \longrightarrow U-\mathfrak{Cat}$ is given by $(C,h,\circ) \mapsto C$, the underlying category functor or the $n$-forgetful functor i.e. a pair $For^{U(n)-\mathfrak{Cat}}_{U-\mathfrak{Cat}} = For^{WE(F,id)}_{\mathfrak{Cat}}$. The underlying objects functor for $n$, $F_{U(n)-\mathfrak{Cat}} : U(n)-\mathfrak{Cat} \longrightarrow U-\mathfrak{Set}$ is given by $(C,h,\circ) \mapsto Ob(C)$.

(ii) The set $S_{2}$ consists of all pairs of the form $(n+1,As)$, where $n \in \mathbb{N}$ and $As$ is a $U'$-function (set of pairs), consisting of all pairs of the form $((a,b,c),(F,F_{2}))$, where there exists some $U'$-category $C$ such that $(n+1,C) \in S_{5}$ (i.e. $C$ can be identified with the category of $n+1$-categories), $a = (a_{0},h_{a},\circ_{a}),b = (b_{0},h_{b},\circ_{b}),c = (c_{0},h_{c},\circ_{c}) \in Ob(C)$, and if $a',b',c' \in Ob(U-\mathfrak{Cat})$ are such that $(a,a'),(b,b'),(c,c') \in s$ and $(n+1,s) \in S_{1}$ for some $s \in U'$, (i.e. $s$ provides an $(n+1)$-forgetful functor), then $F_{2}$ is a function which sends a pair of triples $((x_{a},x_{b}),x_{c}),((y_{a},y_{b}),y_{c}) \in Ob((a'\times b') \times c')$ to $\alpha''$ iff there exists $\alpha' \in U'$ such that $(n,\alpha') \in S_{2}$ (i.e. $\alpha'$ provides an $n$-associator), and $\alpha''$ is assigned by $\alpha'$ to the triple $(h_{a}(x_{a},y_{a}),h_{b}(x_{b},y_{b}),h_{c}(x_{c},y_{c}))$.
$$
S_{2} = \{ (n+1 , \{ ((a,b,c), (F,F_{2})) \in U' ;
$$
$$
\exists C \in Ob(U'-\mathfrak{Cat}), (n+1,C) \in S_{5} \text{ and } a,b,c \in Ob(C) \text{ and }
$$
$$
\exists a',b',c' \in Ob(U-\mathfrak{Cat}), \exists s,
$$
$$
(n+1,s) \in S_{1} \text{ and } (a,a'),(b,b'),(c,c') \in s \text{ and } F = \alpha(a',b',c') \text{ and }
$$
$$
\exists \alpha' \in U', (n,\alpha') \in S_{2} \text{ and } \forall x_{a},y_{a} \in Ob(a'), \forall x_{b},y_{b} \in Ob(b'), \forall x_{c},y_{c} \in Ob(c'), \forall \alpha'' \in U',
$$
$$
((((x_{a},x_{b}),x_{c}),((y_{a},y_{b}),y_{c})),\alpha'') \in F_{2} \text{ iff } 
$$
$$
((h_{a}(x_{a},y_{a}),h_{b}(x_{b},y_{b}),h_{c}(x_{c},y_{c})),\alpha'') \in \alpha' \} ) \in U' \}
$$

The associator $\alpha_{U(n)-\mathfrak{Cat}} : Ob(U(n)-\mathfrak{Cat})^{3} \rightarrow Arr(U(n)-\mathfrak{Cat})$ is inductively defined so that for any $n \in \mathbb{N}$ for which $\alpha_{U(n)-\mathfrak{Cat}}$ and $\times_{U(n+1)-\mathfrak{Cat}}$ are defined, for any $a,b,c \in Ob(U(n+1)-\mathfrak{Cat})$
$$
\alpha_{U(n+1)-\mathfrak{Cat}}(a,b,c) : (a \times_{U(n+1)-\mathfrak{Cat}} b) \times_{U(n+1)-\mathfrak{Cat}} c \longrightarrow a \times_{U(n+1)-\mathfrak{Cat}} (b \times_{U(n+1)-\mathfrak{Cat}} c)
$$

is given by the usual associator $\alpha_{U-\mathfrak{Cat}}$ on the underlying category 
$$
For^{U(n+1)-\mathfrak{Cat}}_{U-\mathfrak{Cat}}((a\times_{U(n+1)-\mathfrak{Cat}}b) \times_{U(n+1)-\mathfrak{Cat}} c) =
$$
$$
(For^{U(n+1)-\mathfrak{Cat}}_{U-\mathfrak{Cat}}(a) \times_{U-\mathfrak{Cat}} For^{U(n+1)-\mathfrak{Cat}}_{U-\mathfrak{Cat}}(b)) \times_{U-\mathfrak{Cat}} For^{U(n+1)-\mathfrak{Cat}}_{U-\mathfrak{Cat}}(c)
$$

and 
$$
\alpha_{U(n)-\mathfrak{Cat}}(h_{a}(x_{a},y_{a}),h_{b}(x_{b},y_{b}),h_{c}(x_{c},y_{c}))
$$

on each of the hom $n$-categories.

(iii) The set $S_{3}$ is defined to consist of pairs $(n+1,(F_{0},F_{1}))$, so that $F_{0}$ and $F_{1}$ are functions, defined so that (a). $F_{0}$ consists of pairs associating to the pair $((\bar{C} = (C,h_{C},\circ_{C}),\bar{D} = (D,h_{D},\circ_{D}))$, where there exists $Ca \in Ob(U'-\mathfrak{Cat})$, $(n+1,Ca) \in S_{5}$ (i.e. $Ca$ is identified with the category of $n+1$-categories) such that $\bar{C},\bar{D} \in Ob(Ca)$, the product category $C \times D$, enriched by taking the $n$-products of each hom object, and (b). $F_{1}$ associates to a pair $\Phi = (\phi,\phi_{2}),\Psi = (\psi,\psi_{2}) \in Arr(Ca)$ the product functors on the underlying categories, paired with the function which sends a pair of pairs $((x_{c},x_{d}),(y_{c},y_{d}))$ to the $n$-product of $\phi_{2}$ and $\psi_{2}$.
$$
S_{3} = \{ (n+1, ( \{ ((\bar{C} = (C,h_{C},\circ_{C}), \bar{D} = (D,h_{D},\circ_{D})),
$$
$$
(C\times_{U-\mathfrak{Cat}} D, \{ (((a,b),(c,d)), \times (h_{C}(a,c),h_{D}(b,d))) ; \exists \times \in U', (n,\times) \in S_{3} \}, \circ )) ;
$$
$$
\exists Ca \in Ob(U'-\mathfrak{Cat}), (n+1,Ca) \in S_{5} \text{ and } \bar{C},\bar{D} \in Ob(Ca) \} , \{ Arrows \} )) \in U' \}
$$

The product $\times_{U(n+1)-\mathfrak{Cat}} : U(n+1)-\mathfrak{Cat} \times_{U'-\mathfrak{Cat}} U(n+1)-\mathfrak{Cat} \longrightarrow U(n+1)-\mathfrak{Cat}$ is defined, for any $n \in \mathbb{N}$ such that $\times_{U(n)-\mathfrak{Cat}}$, $\alpha_{U(n)-\mathfrak{Cat}}$, and $\sigma_{U(n)-\mathfrak{Cat}} : Ob(U(n)-\mathfrak{Cat})^{2} \rightarrow Arr(U(n)-\mathfrak{Cat})$ are already defined so as to be given by
$$
((a,h_{a},\circ_{a}),(b,h_{b},\circ_{b})) \mapsto (a \times_{U-\mathfrak{Cat}} b, (((x_{a},x_{b}),(y_{a},y_{b})) \mapsto h_{a}(x_{a},y_{a}) \times_{U(n)-\mathfrak{Cat}} h_{b}(x_{b},y_{b})),...)
$$

with the composition defined component-wise (using the $n$-symmetrizer and associator maps).

(iv). The set $S_{4}$ is defined so as to consist of pairs of the form $(n+1,Sym)$, where $Sym$ is a set consisting of pairs $((\bar{C},\bar{D}),(F,F_{2}))$, where for some $Ca \in Ob(U'-\mathfrak{Cat})$ such that $(n+1,Ca) \in S_{5}$, $\bar{C},\bar{D} \in Ob(Ca)$, $F$ is the usual symmetrizer, and $F_{2}$ sends a pair of pairs $((x_{c},x_{d}),(y_{c},y_{d}))$ to the $n$-symmetrizer.

$$
S_{4} = \{ (n+1, \{ ((\bar{C} = (C,h_{C},\circ_{C}),\bar{D} = (D,h_{D},\circ_{D})), (F,F_{2})) ; F = \sigma (C,D) \text{ and }
$$
$$
F_{2} = \{ ( ((x_{c},x_{d}),(y_{c},y_{d})) ,\sigma') \in U' ; \exists \sigma'' \in U', (n,\sigma'') \in S_{4} \text{ and }
$$
$$
((h_{(C)}(x_{c},y_{c}),h_{D}(x_{d},y_{d})) ,\sigma') \in \sigma'' \} \}) \in U'; \exists Ca \in Ob(U'-\mathfrak{Cat}), (n+1,Ca) \in S_{5} \text{ and } 
$$
$$
\bar{C},\bar{D} \in Ob(Ca) \}
$$

The symmetrizer $\sigma_{U(n+1)-\mathfrak{Cat}} : Ob(U(n+1)-\mathfrak{Cat})^{2} \rightarrow Arr(U(n+1)-\mathfrak{Cat})$ is defined for any $n \in \mathbb{N}$ for which $\sigma_{U(n)-\mathfrak{Cat}}$ and $\times_{U(n+1)-\mathfrak{Cat}}$ are defined, using the usual associator on the underlying categories and $\sigma_{U(n)-\mathfrak{Cat}}$ on the hom $n$-categories, so that 
$$
\sigma_{U(n+1)-\mathfrak{Cat}}(a,b) : a \times_{U(n+1)-\mathfrak{Cat}} b \longrightarrow b \times_{U(n+1)-\mathfrak{Cat}} a
$$

for any $a,b \in Ob(U(n+1)-\mathfrak{Cat})$.

(v). The set $S_{5}$ consists of all pairs of the for $(n+1,D)$, where there exist $D' \in Ob(U'-\mathfrak{Cat})$, $F,\times \in U'$, for which $(n,D') \in S_{5}$, $F = For_{U(n)-\mathfrak{Cat}}$ is the underlying objects functor for $n$, $\times$ is the $n$-product (i.e. $(n,\times) \in S_{3}$), $\rho : (F \circ \times ) \rightarrow \times_{U'-\mathfrak{Set}} \circ (F \times_{U'-\mathfrak{Cat}} F)$ is an isomorphism of functors, given by the identity on each set, and $D$ is the category of weakly $(D',F,\rho,\times)$-enriched categories.
$$
S_{5} = \{ (n+1, WE_{\mathfrak{Cat}}((Ca, F),\times)) \in U' ; (n,Ca) \in S_{5} \text{ and } (n,F) \in S_{1} \text{ and } (n,\times) \in S_{3} \} \cup
$$
$$
\{ (1, U-\mathfrak{Cat} ) \}
$$

For any $n \in \mathbb{N}$ for which $U(n)-\mathfrak{Cat}$ and $\times_{U(n)-\mathfrak{Cat}}$ are defined, $U(n+1)-\mathfrak{Cat}$ is defined to be the category of all categories weakly enriched over $U(n)-\mathfrak{Cat}$, i.e. $U(n+1)-\mathfrak{Cat} := WE((U(n)-\mathfrak{Cat},\times_{U(n)-\mathfrak{Cat}}), F_{U(n)-\mathfrak{Cat}})$.

\begin{lem}
$F_{U( )-\mathfrak{Cat}}$, $\alpha_{U( )-\mathfrak{Cat}}$, $\times_{U( )-\mathfrak{Cat}}$, $\sigma_{U( )-\mathfrak{Cat}}$, and $U( )-\mathfrak{Cat}$ are well defined functions $\mathbb{N} \longrightarrow U'$.
\end{lem}
\begin{proof}
If $U(n+1)-\mathfrak{Cat}$ is uniquely determined, then $\alpha_{U(n+1)-\mathfrak{Cat}}$, $\times_{U(n+1)-\mathfrak{Cat}}$, and
\newline $\sigma_{U(n+1)-\mathfrak{Cat}}$ are defined directly from $\alpha_{U(n)-\mathfrak{Cat}}$, $\times_{U(n)-\mathfrak{Cat}}$, and $\sigma_{U(n)-\mathfrak{Cat}}$. $U(n+1)-\mathfrak{Cat}$ is defined from $\alpha_{U(n)-\mathfrak{Cat}}$, $\times_{U(n)-\mathfrak{Cat}}$, $\sigma_{U(n)-\mathfrak{Cat}}$, and $U(n)-\mathfrak{Cat}$. Apply induction.  $F_{U(n)-\mathfrak{Cat}}$ is straightforward, using only $U(n)-\mathfrak{Cat}$ and the usual object functor.
\end{proof}

\sss{Inducing Enrichments of $Hom_{\mathfrak{Cat}^{2}}^{(1)}(I,C)$ from Enrichments of $C$} The following is an older version of the hom set enrichment which appears in the main body. The exact relation should be worked out.

\begin{lem}
(Hom-Cat Enrichment)

Suppose that the amnesia $F : A \rightarrow \mathfrak{Set}$ of a weak $(sk : A \rightarrow B)$-associative enrichment $(C,h,\circ)$ over $((A,\otimes ), \alpha )$ is co-represented by a unit object
$$
\Phi : F \longrightarrow Yo^{opp}_{(A)(0)}(I)
$$
with unit arrows (natural transformations)
$$
u_{l} : Id_{A} \longrightarrow I \otimes - \text{ and } u_{r} : Id_{A} \longrightarrow - \otimes I
$$

Then for any category $J$ one can construct a very weak enrichment of $Hom^{(1)}_{U-\mathfrak{Cat}^{2}}(J,C)$ over $(A,\otimes )$ by the following.

(i). Choose for each pair $G,H \in Ob(Hom^{(1)}_{U-\mathfrak{Cat}^{2}}(J,C))$ of functors, products $Dia :_{t}= \prod_{\phi \in Arr(J)} h(G(dom(\phi )),H(codom(\phi)))$ and $Hor :_{t}= \prod_{a \in Ob(J)} h(G(a),H(a))$.

(ii). Choose, for each $G,H \in Ob(Hom^{(1)}_{U-\mathfrak{Cat}^{2}}(J,C))$ arrows $f,g : Hor \rightarrow Dia$ in $A$, determined by the construction of $Dia$ as a product, so that $f$ is given by associating $\phi \in Arr(J)$ with
$$
Hor \xrightarrow{\pi} h(G(dom(\phi)),H(dom(\phi))) \xrightarrow{u_{l}} I \otimes h(G(dom(\phi)),H(dom(\phi))) \xrightarrow{id \otimes \Phi (H(\phi))}
$$
$$
h(H(dom(\phi)),H(codom(\phi))) \otimes h(G(dom(\phi)),H(dom(\phi))) \xrightarrow{\circ}
$$
$$
h(G(dom(\phi)),H(codom(\phi)))
$$
and $g$ is given by associating $\phi \in Arr(J)$ with
$$
Hom \xrightarrow{\pi} h(G(codom(\phi)),H(codom(\phi))) \xrightarrow{u_{r}} h(G(codom(\phi)),H(codom(\phi))) \otimes I \xrightarrow{\Phi(G(\phi)) \otimes id}
$$
$$
h(G(codom(\phi)),H(codom(\phi))) \otimes h(G(dom(\phi)),G(codom(\phi))) \xrightarrow{\circ}
$$
$$
h(G(dom(\phi)),H(codom(\phi)))
$$

(iii). Let $l : Ob(Hom^{(1)}_{U-\mathfrak{Cat}^{2}}(J,C))^{2} \rightarrow Ob(A)$ assign to each pair $(F,G)$ the $(sk,\varepsilon_{0})$-limit of the diagram $D(F,G)$ consisting of the two arrows $f$ and $g$, where $\varepsilon_{0}$ is the inclusion of the sub-category of $dom(D(F,G))$ with the same objects, and only identity arrows.

(iv). Suppose that for each $G_{1},G_{2},G_{3} \in Ob(Hom^{(1)}_{U-\mathfrak{Cat}^{2}}(J,C))$,
$$
s(G_{1},G_{2},G_{3}) : colim(\otimes \circ (p_{D(G_{2},G_{3})} \times p_{D(G_{1},G_{2})})) \rightarrow colim(p_{D(G_{2},G_{3})}) \otimes colim(p_{D(G_{1},G_{2})})
$$
is defined to be that which is induced by sending $((a,\alpha),(b,\beta)) \in Ob(P_{D(G_{2},G_{3})} \times P_{D(G_{1},G_{2})})$ to the tensor of its colimit arrows, $e_{\alpha} \otimes e_{\beta} : a \otimes b \rightarrow colim(p_{D(G_{2},G_{3})}) \otimes colim(p_{D(G_{1},G_{2})})$, and that $s(G_{1},G_{2},G_{3})$ is an isomorphism.

(v). Suppose that for each $G_{1},G_{2} \in Ob(Hom^{(1)}_{U-\mathfrak{Cat}^{2}}(J,C))$, for any $(a,\alpha) \in Ob(P_{D(G_{1},G_{2})})$, the colimit arrow $\lambda_{p_{D(G_{1},G_{2})}}(a,\alpha) : a \rightarrow colimit(p_{D(G_{1},G_{2})})$ is the unique arrow for which $\pi_{i} \circ \lambda_{\pi} \circ \lambda_{p_{D(G_{1},G_{2})}}(a,\alpha) = \alpha (i)$ for each $i \in Ob(dom(D(G_{1},G_{2})))$, where
$$
\lambda_{\pi} : colim(p_{D(G_{1},G_{2})}) \rightarrow \prod_{i \in Ob(dom(D(G_{1},G_{2})))} D(G_{1},G_{2})(i)
$$
is induced by the arrows $\alpha(i)$ for $(a,\alpha) \in Ob(P_{D(G_{1},G_{2})})$.\ftt{This condition is as in the conclusion of the uniqueness via monic lemma of II.1}

(vi). Define a function $t : Ob(Hom^{(1)}_{U-\mathfrak{Cat}^{2}}(J,C))^{3} \rightarrow Arr(A)$ so that for any $G_{1},G_{2},G_{3} \in Ob(Hom^{(1)}_{-\mathfrak{Cat}^{2}}(J,C))$,
$$
t(G_{1},G_{2},G_{3}) : colim(\otimes \circ (p_{D(G_{2},G_{3})} \times p_{D(G_{1},G_{2})})) \rightarrow colim (p_{D(G_{1},G_{3})} )
$$
is the colimit arrow assigned to $(colim(\otimes \circ (p_{D(G_{2},G_{3})} \times p_{D(G_{1},G_{2})}),\alpha) \in Ob(P_{D(G_{1},G_{3})})$, where if $Ob(dom(D(G_{1},G_{2}))) = Ob(dom(D(G_{2},G_{3}))) = Ob(dom(D(G_{1},G_{3}))) = \{ c,d,m \}$ so that
$$
c \text{ corresponds to } \prod_{\phi \in Arr(J)} h(G_{i}(codom(\phi)),G_{j}(codom(\phi))) \text{ and }
$$
$$
d \text{ corresponds to } \prod_{\phi \in Arr(J)} h(G_{i}(dom(\phi)),G_{j}(dom(\phi))) \text{ and }
$$
$$
m \text{ corresponds to } \prod_{\phi \in Arr(J)} H(G_{i}(dom(\phi)),G_{j}(codom(\phi))),
$$

then $\alpha : \Delta_{(dom(\varepsilon_{0}),A)}(colim (\otimes \circ (p_{D(G_{2},G_{3})} \times p_{D(G_{1},G_{2})} )) \rightarrow D(G_{1},G_{3}) \circ \varepsilon_{0}$ is given by 
$$
\alpha (c) : colim (\otimes \circ (\otimes \circ (p_{D(G_{2},G_{3})} \times p_{D(G_{1},G_{2})} )) \rightarrow \prod_{\phi \in Arr(J)} h(G_{1}(codom(\phi)),G_{3}(codom(\phi)))
$$ 
corresponding to the assignment $(\phi \mapsto \circ \circ \otimes (( \pi_{\phi} \circ \beta(c) , \pi_{\phi} \circ \gamma(c))))_{\phi \in Arr(J)}$ for any $(b,\beta ) \in Ob(P_{D(G_{2},G_{3})})$ and any $(c,\gamma ) \in Ob(P_{D(G_{1},G_{2})})$,
$$
\alpha (d) : colim (\otimes \circ (\otimes \circ (p_{D(G_{2},G_{3})} \times p_{D(G_{1},G_{2})} )) \rightarrow \prod_{\phi \in Arr(J)} h(G_{1}(dom(\phi)),G_{3}(dom(\phi)))
$$
corresponding to the assignment $(\phi \mapsto \circ \circ \otimes (( \pi_{\phi} \circ \beta(d) , \pi_{\phi} \circ \gamma(d))))_{\phi \in Arr(J)}$ for any $(b,\beta ) \in Ob(P_{D(G_{2},G_{3})})$ and any $(c,\gamma ) \in Ob(P_{D(G_{1},G_{2})})$,

and any 
$$
\alpha (m) : colim (\otimes \circ (\otimes \circ (p_{D(G_{2},G_{3})} \times p_{D(G_{1},G_{2})} )) \rightarrow \prod_{\phi \in Arr(J)} h(G_{1}(dom(\phi)),G_{3}(codom(\phi)))
$$ 

for which $(colim(\otimes \circ (p_{D(G_{2},G_{3})} \times p_{D(G_{1},G_{2}})) , \alpha ) \in Ob(P_{D(G_{1},G_{3})})$. (Herein is a choice required in the construction)

(vii). Then $\bar{\circ} :_{t}= ((G_{1},G_{2},G_{3}) \mapsto t(G_{1},G_{2},G_{3}) \circ s(G_{1},G_{2},G_{3})^{-1})_{G_{1},G_{2},G_{3} \in Ob(Hom^{(1)}_{U-\mathfrak{Cat}^{2}} (J,C))}$ implies that $(Ob(Hom^{(1)}_{U-\mathfrak{Cat}^{2}}(J,C)), l, \bar{\circ})$ is a very weak $A$-enrichment.

(viii). If the enrichment is $F$-associative, then the above construction yields a weak enrichment over the category whose objects are functors $J \longrightarrow C$ and arrows are \newline $\coprod_{G_{1},G_{2} \in Hom_{U-\mathfrak{Cat}}(J,C)}  F_{(0)}(\lambda (G_{1}, G_{2}))$, with composition given by application of $F_{(1)}$ to the enriched composition arrows in $A$.
\end{lem}

\begin{ex}
In the case in which $C = U-\mathfrak{Cat}$ and $sk = Skel$, if $F$ is the object functor, then the underlying category in the second part of the lemma would have as objects the functors and arrows weak natural transformations, for which $G_{2}(\phi) \circ f_{dom(\phi)} \cong f_{codom(\phi)} \circ G_{1}(\phi)$.
\end{ex}

\begin{lem}
(Induced Enrichments on Categories of Diagrams in $(A,\otimes )-\mathfrak{Cat}$, given $sk: A \rightarrow B$). If the amnesia is representable, and there is an isomorphism $\times_{U-\mathfrak{Set}} \circ (F \times F) \longrightarrow F \circ \otimes$ (i.e. the amnesia can made into a ``strong" arrow in $\mathfrak{TCat}$), then $(A,\otimes)-\mathfrak{Cat}$, has a natural weak enrichment over itself. 
\end{lem}

\begin{cor}
Auto-enrichment of $(\mathfrak{Ab},\otimes)-\mathfrak{Cat}$.
\end{cor}

\begin{lem}
A full functor $p$ whose domain is weakly enriched over some $(A,\otimes)$ induces a weak enrichment over the same on the codomain.
\end{lem}

\end{document}